\documentclass[11pt,namelimits,sumlimits]{amsart}
\usepackage{citesort}
\usepackage{amssymb,amsmath}
\usepackage[mathscr]{eucal}

\usepackage{psfrag}
\usepackage{enumerate}

\usepackage{graphicx}
\usepackage{amsmath,amsthm}
\usepackage{graphics}
\usepackage{color}
\usepackage{epsfig}
\usepackage{fullpage}
\usepackage{amssymb,amsmath}
\usepackage[mathscr]{eucal}

\newcounter{mtheorem}

\newtheorem{theorem}{Theorem}[section]
\newtheorem{lemma}[theorem]{Lemma}
\newtheorem{prop}[theorem]{Proposition}
\newtheorem{corollary}[theorem]{Corollary}

\newtheorem{definition}[theorem]{Definition}

\theoremstyle{remark}
\newtheorem{remark}[theorem]{Remark}

\numberwithin{equation}{section}

\newcommand{\Real}{\mathbb R}

\newcommand{\func}[1]{\ensuremath{\mathrm{#1} \:} }

\newcommand{\arccosh}[0]{\func{arccosh}}

\newcommand{\arcsinh}[0]{\func{arcsinh}}
\newcommand{\supp}[0]{\func{supp}}

\newcommand{\sech}[0]{\func{sech}}

\newcommand{\dist}[0]{\mathrm{dist}}

\newcommand{\xX}[0]{\mathbf{x}}
\newcommand{\yY}[0]{\mathbf{y}}

\newcommand{\Bv}[0]{\mathbf{v}}
\newcommand{\Be}[0]{\mathbf e}
\newcommand{\Bvp}[0]{\mathbf{v}}

\newcommand{\eL}[0]{\ell_\Lambda}
\newcommand{\taue}[0]{\tau_{d} (e)}
\newcommand{\outaupp}[0]{\outau''}
\newcommand{\Gd}[0]{\Gamma(d,0)}

\newcommand{\Ghd}[0]{\Gamma(\hat d,0)}
\newcommand{\Ghdl}[0]{\Gamma(\hat d,\ell)}
\newcommand{\hd}[0]{{\hat d}}

\newcommand{\oux}[0]{\overline{\underline x}}

\newcommand{\pe}[0]{[p,e]}
\newcommand{\ppe}[0]{[p^+[e],e]}
\newcommand{\pme}[0]{[p^-[e],e]}
\newcommand{\pen}[0]{[p,e,n]}

\newcommand{\Gudl}[0]{\Gamma(d,\ell)}

\newcommand{\ephdl}[0]{[e; \hat d, \ell]}

\newcommand{\uthte}[0]{{\tau_0(e)}}
\newcommand{\RRR}[0]{\mathsf R}
\newcommand{\TTT}[0]{\mathsf T}
\newcommand{\UUU}[0]{\mathsf U}

\newcommand{\hYtdz}[0]{ Y_{d,{\boldsymbol \zeta}}}
\newcommand{\Htdz}[0]{H_{d,{\boldsymbol \zeta}}}
\newcommand{\Mtdz}[0]{M_{d, {\boldsymbol \zeta}}}
\newcommand{\Ntdz}[0]{N_{d, {\boldsymbol \zeta}}}

\newcommand{\Ss}[0]{\mathbb S}
\newcommand{\delt}[0]{\epsilon}

\newcommand{\outau}[0]{{\tau}}
\newcommand{\sep}[0]{{\sigma_{e,p}}}
\newcommand{\penpp}[0]{[p,e,n'']}
\newcommand{\penp}[0]{[p,e,n']}
\newcommand{\RH}[0]{\mathbf P[e]}
\newcommand{\RHj}[0]{\mathbf P[e_j]}
\newcommand{\uphipp}[0]{\phi_{dislocation}'[p] }
\newcommand{\uphip}[0]{\phi_{dislocation}''[p]}

\newcommand{\cf}[0]{\check f}

\newcommand{\cfw}[0]{\frac{|A|^2}{2\rho^2}}
\newcommand{\cfE}[0]{\frac{2\rho^2}{|A|^2}}

\begin{document}

\title[CMC]{Embedded Constant Mean Curvature Surfaces in Euclidean three-space}
\author[C.~Breiner]{Christine~Breiner}

\address{Department of Mathematics, Columbia University, New York, NY 10027}
\email{cbreiner@math.columbia.edu  }

\author[N.~Kapouleas]{Nikolaos~Kapouleas}
\address{Department of Mathematics, Brown University, Providence,
RI 02912} \email{nicos@math.brown.edu}
\thanks{The authors were supported in part by NSF grants DMS-0902718 and DMS-1105371, respectively.}


\date{\today}



\begin{abstract}
In this paper we refine the construction and related estimates
for complete Constant Mean Curvature surfaces in Euclidean three-space
developed in \cite{KapAnn}
by adopting the more precise and powerful version of the methodology 
which was developed in \cite{KapWente}.
As a consequence we remove the severe restrictions in establishing
embeddedness for complete Constant Mean Curvature surfaces in \cite{KapAnn}
and we produce a very large class of new embedded examples of finite topology.
\end{abstract}

\maketitle

\nopagebreak

\section{Introduction}
Critical points to the area functional subject to an enclosed volume
constraint have constant mean curvature $H \equiv c \neq 0$.
We let ``CMC surface'' denote a complete, constant mean curvature
immersed smooth surface in $\Real^3$.
The only classically known CMC examples of finite topological type were the round spheres and cylinders
and more generally the rotationally invariant surfaces discovered by Delaunay in 1841 \cite{Delaunay}.
Hopf \cite{Hopf} proved that the only closed CMC surfaces of genus zero are the round spheres.
In 1986 Wente constructed genus one closed immersed examples \cite{Wente}.
Using a general gluing methodology developed in \cite{schoen,KapAnn},
and using the Delaunay surfaces as building blocks,
most of the possible finite topological types were realized as immersed CMC surfaces for the first time \cite{KapAnn,KapJDG}.
However, no genus two closed examples could be produced with the constructions in \cite{KapAnn,KapJDG}. 
In \cite{KapWente} a systematic and detailed refinement of the original gluing methodology
made it possible to construct genus two closed examples by using the Wente tori as building blocks.
Since then, many other gluing problems have been successfully resolved by using this refined approach
\cite{KapSurvey,Yang,HaskKap,HaskKap2,KapClay,KapJDGCMC,KapPNAS,KapJDG,KapWente}. 

We discuss now the case of embedded, or more generally Alexandrov embedded, CMC surfaces. 
Let $\mathcal M_{g,k}$ denote the moduli space of all Alexandrov embedded CMC surfaces
with $H \equiv 1$, finite genus $g$, and $k$ ends where two surfaces are considered
equivalent if they differ by a rigid motion of $\Real^3$.
Using a reflection technique, Alexandrov \cite{Alexandrov} proved the only embedded
$\Sigma \in \mathcal M_{g,0}$ is the round sphere. 
Meeks \cite{MeeksCMC} proved the space $\mathcal M_{g,1}$ is empty and that every end of $\Sigma \in \mathcal M_{g,k}$ is cylindrically bounded.
Motivated by \cite{MeeksCMC,KapAnn},
Korevaar, Kusner, and Solomon \cite{KKS} showed that each end converges
exponentially fast to a Delaunay surface and any $\Sigma \in \mathcal M_{g,2}$ is necessarily a Delaunay embedding.
Kusner, Mazzeo, and Pollack \cite{KuMaPo} proved that $\mathcal M_{g,k}$ is a real-analytic variety.
Moreover, for a non-degenerate $\Sigma \in\mathcal M_{g,k}$ there exists a neighborhood of $\Sigma$
in $\mathcal M_{g,k}$ that is a real-analytic manifold of dimension $\leq 3k-6$.
Here strict inequality occurs when one quotients by the finite isotropy subgroup of $\Sigma$
in the group of Euclidean motions.

Some further constructions of CMC surfaces have been carried out.
Gro{\ss}-Brauckmann \cite{GB} used a conjugate surface construction to construct the surfaces in $\mathcal M_{0,k}$
with maximal ($k$-fold dihedral) symmetry.
This includes those which have large neck size (the examples in \cite{KapAnn} all have small neck size).
Mazzeo and Pacard \cite{MaPa} extended the gluing construction to produce CMC surfaces
by attaching Delaunay ends to the ends of a
non-degenerate {Alexandrov} embedded minimal surface with finite total curvature, genus $g$, and $k$ catenoidal ends.
Mazzeo, Pacard, and Pollack \cite{MaPaPoClay} proved that a Delaunay end can be attached to a non-degenerate $\Sigma \in \mathcal M_{g,k}$
to produce a non-degenerate $\Sigma' \in \mathcal M_{g,k+1}$.

In this paper we return to the general construction in \cite{KapAnn}
and refine the construction and estimates
by applying the improved methodology developed in \cite{KapWente}.
In particular we can now produce a large class of
embedded examples of finite topological type.
In \cite{KapAnn} too much ``bending'' of the catenoidal regions was required to ensure
a successful construction and this destroyed embeddness (but not Alexandrov embeddeness).
It was possible to avoid the need for this bending only in a few cases of very high symmetry,
more precisely only when a few Delaunay ends were attached to a central sphere and the symmetries
of a Platonic solid were imposed.
In the current construction no such bending is needed; therefore we can ensure embeddedness for all reasonable candidates.
The examples we construct in $\mathcal M_{g,k}$ have $3k-6$ continuous parameters as expected.
Because no symmetries need to be imposed in our construction,
we hope that it will serve as a model for applying the general methodology in general settings
without symmetries.
The current construction can also be extended to higher dimensions \cite{Hig}.

We now provide a moderately detailed outline of the construction in the subsections that follow. 
 
\subsection*{The graph and the parameters}
We begin the construction by designating a finite graph $\Gamma$ that satisfies a few necessary geometric conditions and a flexibility condition. $\Gamma$ consists of a finite collection of vertices, edges, and rays and to each edge or ray we assign a parameter $\hat \tau(e)$. The geometric conditions arise as a result of certain properties of CMC surfaces as well as the asymptotic geometry of the building blocks we describe below. We require the flexibility condition to guarantee the existence of a family of graphs that are smooth in two modifications we denote $d,\ell$. The parameter $d$ will assign a vector to each vertex that is the sum of each unit direction of an edge or ray times its associated parameter $\hat \tau$ \eqref{unbalancingeq}. The parameter $\ell$ will vary the length of each edge in the graph $\Gamma$. Notice that changing lengths may result in a change of direction for the edges and thus $\ell$ can influence $d$. We will be interested in graphs $\Gamma$ that vary smoothly in $d,\ell$ for sufficiently small $d, \ell$.

We will immerse an initial surface in $\Real^3$ based on the parameters that come from the graph $\Gamma$ and from two additional parameters $d,\boldsymbol \zeta$. We call $d$ the \emph{unbalancing} parameter as it describes the deviation from balancing of a surface based on $\Gamma$ and $d$ (with $\boldsymbol \zeta = \boldsymbol 0$). An observation of Kusner regarding a homological invariant on CMC surfaces implies that CMC surfaces will be balanced, see \cite{KKS,Kus}, and thus we presume $\Gamma$ has $d \equiv 0$. We call $\boldsymbol \zeta$ the \emph{dislocation} parameter as it will induce a dislocation of the surface from the structure induced by the graph. $\Gamma$ and $\boldsymbol \zeta$ together determine $\ell$ \eqref{ellprimedefinition}.
\subsection*{Building blocks and the initial surface}
The building blocks will be one of two types. The first type corresponds to an $\Ss^2$ with geodesic disks removed and each block of this type will belong to a \emph{standard region} on the initial surface. The second type of building block corresponds to an immersion of a piece of a cylinder that is a Delaunay surface except near the boundary.  We designate the parameter of the family of Delaunay immersions by $\tau$ and presume $\tau\in (-\infty,0)\cup(0, 1/4]$. We let $Y_\tau: \Real \times \Ss^1 \to \Real^3$, see \eqref{DelEq}, denote the immersion of a Delaunay surface with parameter $\tau$. The asymptotic geometry of Delaunay surfaces is well understood; as $\tau \to 0$, the surface converges to a string of $\Ss^2$ joined at antipodal points. Two types of geometric limits exist, where the type of limit depends on the sign of the Gauss curvature $K$. On a symmetric region where $K>0$, as $\tau \to 0$ the region converges to the round sphere.
 Rescaling by $\tau^{-1}$ about the symmetric region where $K<0$ produces the standard catenoid -- a complete minimal annulus in $\Real^3$. We refer to both of these types of regions as \emph{standard} or \emph{almost spherical regions}. To justify this designation on a negatively curved region, consider the conformal metric $h=\frac{|A|^2}{2}g$, where $g$ is the induced metric. Notice that $h$ is invariant under dilations of the surface and recall that for a minimal surface $h$ is precisely $\nu^*g_{\Ss^2}$. In fact, in the metric $h$, the two geometric regions described are isometric. The convergence of the positively curved region to the round sphere thus justifies the name given to both regions.

We designate a neighborhood of the region $K=0$ as a \emph{transition} or \emph{neck region}. In a second conformal metric, $\chi$ (see \eqref{confmets}), each transition region is conformal to a flat $[0,\ell]\times \Ss^1$ where $\ell \approx -\log \tau$. In Figure \ref{Regions} we identify the standard and transition regions on a Delaunay immersion. An \emph{extended standard region} includes one standard region and the two adjacent transition regions.

\begin{figure}[h]\label{Regions}
\includegraphics[width=4in]{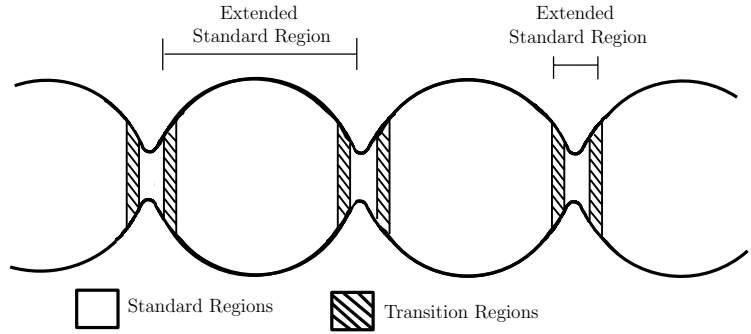}
            \caption{We label the regions of interest on a Delaunay piece. Notice both the positively curved and negatively curved standard regions look like $\Ss^2$ in the $h$ metric. Each shaded region becomes a long cylinder in the conformally flat $\chi$ metric.}
      \end{figure}

We incorporate into the second building block a dislocation, designated by the parameter $\boldsymbol \zeta$. If one wishes to transit smoothly between a sphere and a Delaunay surface, the asymptotic geometry of a Delaunay immersion implies there exists a natural positioning for the two surfaces. The dislocation repositions the center of the sphere so the transit is no longer natural. In fact, to understand precisely the error terms, we apply two separate transitions. The first transition is applied to a Delaunay piece and a sphere placed in their natural positions, the second transition is from the original sphere to an off center sphere placed in a position designated by the dislocation parameter (Definition \ref{DelEmbDef}).

The positioning of the pieces depends on the modified graph $\Gudl$ and the parameter $\boldsymbol \zeta$. At each vertex, we position a sphere with disks removed and at each edge or ray we position a Delaunay building block with axis depending on $\boldsymbol \zeta$. To guarantee the immersion is well defined, we must identify the neighborhood of each end of a Delaunay piece with the appropriate annular domain on an adjacent sphere. The initial surface has CMC identically one except on certain domains of each Delaunay building block, and the error is determined by the parameters $\tau,\boldsymbol \zeta$.

\subsection*{The PDE and linearized equation} 
 
For a surface $M \subset \Real^3$ and a sufficiently small smooth function $f:M \to \Real$, it is well known that for $M_f:=\{ \xX + f(\xX)N_M(\xX) | \xX \in M\}$ where $N_M$ is a section of the normal bundle to $M$, one has 
\[
H_f = H_M + (\Delta + |A|^2) f + Q_f.
\]Here $H_M, H_f$ represent the mean curvature of the surfaces $M,M_f$ respectively, $\Delta$ and $|A|^2$ are the Laplacian and norm squared of the second fundamental form in the induced metric, and $Q_f$ represents quadratic and higher terms in $f$ and its derivatives. To obtain a CMC surface with mean curvature $1$, we need to solve
\[
\mathcal L f:=(\Delta + |A|^2) f =1-H_M-Q_f.
\]
As in \cite{KapSurvey,Yang,HaskKap,HaskKap2,KapBAMS,KapClay,KapDrops,KapThes,KapJDGCMC,KapPNAS,KapJDG,KapWente}, we solve the linearized problem on various regions of the surface in one of the conformal metrics $h,\chi$. The metric $\chi$ sets a natural scale by behaving -- up to a small reparameterization -- like the flat metric on the cylinder for each Delaunay piece. In our particular setup, we determine $1-H_M$ is small in the $\chi$ metric and thus we expect $f$ should have small norm. The scale induced by $\chi$ implies that the term $Q_f$ should not dominate.

Our first goal is to prove Proposition \ref{LinearSectionProp}, which provides a global solution to a modification of the original linearized problem $\mathcal L f = 1-H_M:=E$. Obstructions to solving the linearized problem correspond to the existence of kernel for the operator $\mathcal L$. Additional technicalities arise because the process allows for small modifications to our surface which may induce small changes to the spectrum of the operator $\mathcal L$. For this reason, we must identify the space of eigenfunctions for the linearized operator with small eigenvalue, a space we identify as the \emph{approximate kernel}. On each standard region, we determine these potential obstructions to solving the linearized problem by taking advantage of the well understood geometric limit of each standard region in the $h$ metric. Following the strategy of \cite{HaskKap,KapWente,KapYang}, we solve the linearized problem modulo the \emph{extended substitute kernel} $\mathcal K$, a space of functions with properties outlined in Lemmas \ref{extsubslemma}, \ref{vlemma}. We introduce the extended substitute kernel to satisfy two criteria. On each extended standard region, we introduce a three dimensional space of functions in $\mathcal K$ so that for each $E$ there exists $w\in \mathcal K$ such that $E+w$ is $L^2$ orthogonal to the approximate kernel. This allows us to solve the modified semi-local problem on each extended standard region.
We proceed to solve the global linearized problem by applying a global partition of unity to the inhomogeneous term. Since each transition region, in the $\chi$ metric, is conformal to a long, flat cylinder, we solve the Dirichlet problem on each transition region with no modification to the inhomogeneous term. We then add these solutions to the semi-local solutions on their common regions of support to obtain a global one. To guarantee our global solution possesses good estimates, we must modify our semi-local solutions by prescribing the low harmonics on the boundary of an adjacent transition region to guarantee fast exponential decay on that transition region. This modification represents the second use of the extended substitute kernel. On each non-central standard region, the introduction of \emph{substitute kernel} for the first use is also sufficient for the second. On each central standard region, we define a three dimensional space of functions on the core for the first use. For second use, we define a three dimensional space of functions at each attachment that modifies the solution along each of the attached standard regions.

\subsection*{Prescribing extended substitute kernel and the geometric principle}
At this stage, we have obtained a global solution to a modification of the linearized problem, thus solving one problem but introducing another.
To correct this, we perturb the surface in such a way that we induce prescribed changes to the linearized equation.
These changes take the form either of perturbations of the solution or perturbations of the initial surface.
In both cases, our goal is to modify the mean curvature in such a way as to prescribe exactly
the extended substitute kernel that was introduced in the previous step.
Of course we cannot match the modification exactly, but we show the resulting error is small enough to close the argument.
Small perturbations are accomplished by modifying the solution; larger perturbations require we modify the surface itself.
At this step we rely on the flexibility of $\Gamma$ which allows us to define a smooth family of immersions
$\Mtdz \subset \Real^3$ which depend on $\Gamma$ as well as $d,\boldsymbol \zeta$. 

We prescribe extended substitute kernel via the Geometric Principle \cite{KapProc,KapWente,KapClay,KapYang,HaskKap}.
On each non-central standard region, we use a balancing formula to prescribe substitute kernel.
Balancing can be studied at the linearized level where it reduces to Green's second identity; thus the modification on these regions is done on the level of the function.
On each central standard region, the parameters $d,\boldsymbol \zeta$ modify the surface in fundamentally different ways,
although in both cases they control dislocations of the surface.
As in \cite{KapAnn} we introduce the parameter $d$ to prescribe an element of the substitute kernel.
Geometrically, it corresponds to an unbalancing of the graph $\Gamma$ that results from repositioning
the attached edges or rays at a vertex and changing each Delaunay parameter.
We introduce the parameter $\boldsymbol \zeta$ to prescribe the element of the extended substitute kernel
that corresponds to each attachment.
As mentioned previously, at each attachment $\boldsymbol \zeta$ provides a vector that repositions the central sphere relative to each attached Delaunay piece.
On the level of the function, the dislocation amounts to a linear combination of the Killing fields corresponding to the coordinates of the normal vector to the surface.
We describe this relationship precisely in Lemma \ref{quadflemma}.

Prescribing the extended substitute kernel on the central spheres has some novel features because of the lack of symmetry in the construction.
We prescribe the substitute kernel and extended substitute kernel in two separate steps and show that the error introduced
in each of these steps is small enough to be absorbed.
For more detail, see Proposition \ref{prescribequad}. 

Finally, using the geometric principle and careful estimates on the non-linear modification that we describe in Appendix \ref{quadapp}, we use Schauder's fixed point theorem to complete the proof of our main theorem (Theorem \ref{maintheorem}).

\subsection*{Outline of the Paper}The paper consists of seven sections and two appendices. 
In Section \ref{Graphs}, we describe finite graphs and their properties. We then provide examples of graphs that produce embedded surfaces. 
Section \ref{BuildingBlocks} includes a thorough description of the building blocks we use as well as estimates on the norm of the inhomogeneous term of the linearized equation. In Section \ref{InitialSurface} we describe the immersion of an initial surface $\Mtdz$ based on all of the parameters from $\Gamma$ as well as $d,\boldsymbol \zeta$. 
In Section \ref{LinearizedEq} we develop the linear theory we need to solve the linearized problem $\mathcal L_h \phi:= (\Delta_h +|A_h|^2)\phi =E$. 
Finally, in Section \ref{GeoP} we describe how to modify the function and the surface to prescribe extended substitute kernel. We prove the main theorem in Section \ref{MainTheorem}. Appendix \ref{quadapp} contains technical estimates corresponding to modifications of the initial surface. In Appendix \ref{appendixb} we define a function that describes the dislocation of the surface locally as a graph over the original surface. Additionally, we provide proofs for the estimates needed in the proof of the main theorem.

\subsection*{Notation and Conventions} We set the convention that the mean curvature $H:= \frac{\kappa_1 + \kappa_2}{2}$ where $\kappa_1,\kappa_2$ are the two principle curvatures for the surface. Thus, a unit sphere has mean curvature identically one.

Throughout this paper we make extensive use of cut-off functions, and we adopt the following notation:  Let $\Psi:\Real \to [0,1]$ be a smooth function such that
\begin{enumerate}
 \item $\Psi$ is non-decreasing
\item $\Psi \equiv 1$ on $[1,\infty)$ and $\Psi \equiv 0$ on $(-\infty, -1]$
\item $\Psi-1/2$ is an odd function.
\end{enumerate}
For $a,b \in \Real$ with $a \neq b$, let $\psi[a,b]:\Real \to [0,1]$ be defined by $\psi[a,b]=\Psi \circ L_{a,b}$ where $L_{a,b}:\Real \to \Real$ is a linear function with $L(a)=-3, L(b)=3$.
Then $\psi[a,b]$ has the following properties:
\begin{enumerate}
 \item $\psi[a,b]$ is weakly monotone.
\item $\psi[a,b]=1$ on a neighborhood of $b$ and $\psi[a,b]=0$ on a neighborhood of $a$.
\item $\psi[a,b]+\psi[b,a]=1$ on $\Real$.
\end{enumerate}
Finally, it will be convenient for us to use weighted H\"older norms and to that end define
\[
||\varphi: C^{k,\alpha}(\Omega,g,f)||:=\sup_{x \in \Omega}\frac{1}{f(x)} ||\varphi:C^{k,\alpha}(B_x\cap \Omega,g)||
\]where $\Omega$ is a domain in a Riemannian manifold $(M,g)$, $f$ is a weight function, and $B_x$ is a geodesic ball centered at $x$ with radius equal to the minimum of $1$ and half the injectivity radius in the metric $g$.

\section{Finite Graphs}\label{Graphs}
Our construction technique relies on properties of a finite graph that both determines the parameters of various building blocks and describes the relationship between the building blocks. We begin by defining a more general class of finite graphs and then proceed to describe specific properties we will need. The geometry of the building blocks will impose certain restrictions on lengths of edges in the graph, and the need for deformations will imply we are interested in graphs that have some flexibility. We impose additional conditions when we want to guarantee embeddedness of the surface in the main theorem.
\addtocounter{equation}{1}
\begin{definition}
Let $\Gamma$ denote a finite graph, by which we mean a collection 

$\{V(\Gamma),E(\Gamma), R(\Gamma), \hat \tau\}$ such that
\begin{enumerate}
\item $V(\Gamma)$ is a finite collection of vertices, placed in $\Real^3$.
\item $E(\Gamma)$ is a finite collection of edges in $\Real^3$, each with two endpoints in $V(\Gamma)$.
\item $R(\Gamma)$ is a finite collection of rays in $\Real^3$, each with one endpoint in $V(\Gamma)$.
\item $\hat \tau: E(\Gamma) \cup R(\Gamma) \to \Real\backslash \{0\}$ is a function.
\end{enumerate}
\end{definition}

We define two graphs as isomorphic if there exists a one-to-one correspondence between the vertices, edges, and rays such that corresponding rays and edges emanate from the corresponding vertices.
\addtocounter{equation}{1}
\begin{definition}\label{vedef}
Let $E_p$ denote the collection of edges and rays that have $p \in V(\Gamma)$ as an endpoint.  
We then have
\[
\cup_{p \in V(\Gamma)}E_p = E(\Gamma) \cup R(\Gamma).\]
Also, for ease of notation we define the set
\addtocounter{theorem}{1}
\begin{equation}
A(\Gamma) \subset V(\Gamma) \times \left(E(\Gamma) \cup R(\Gamma)\right)
\end{equation}such that $[p,e] \in A(\Gamma)$ if $e \in E_p$.

Note we choose the letter ``$A$" here so that one thinks naturally of an \emph{attachment}.
\end{definition}
\addtocounter{equation}{1}
\begin{definition}\label{Vpdefn}
For each $e \in E_p$, let $\mathbf{v}_{e,p}$ denote a unit vector pointing away from $p$ and parallel to $e$.  For any $p \in V(\Gamma)$, let 
\addtocounter{theorem}{1}
\begin{equation}\label{vectorset}
V_p=\{\Bv_{e,p}|e \in E_p\}.\end{equation}

We let $\hd_\Gamma:V(\Gamma) \to \Real^3$ such that
\addtocounter{theorem}{1}
\begin{equation}\label{unbalancingeq}
\hd_\Gamma(p) = \sum_{e \in E_p} \hat \tau(e)\mathbf{v}_{e,p}
\end{equation}
measures the deviation from balancing at the vertex $p$.
\end{definition}
\addtocounter{equation}{1}
\begin{definition}
If $\hd_\Gamma(p) = 0$ for all $p \in V(\Gamma)$, we say $\Gamma$ is a \emph{balanced} graph. 
\end{definition}

To each edge of the graph we will associate a Delaunay piece with parameter $\tau \hat \tau(e)$ where  $\tau>0$ is a global 
parameter that we choose close to zero so that $\tau \hat \tau(e)$ is small enough to meet all needs.  
(See Section \ref{DelSection} for a description of the Delaunay pieces.)  Since the period of Delaunay surfaces tends to $2$ as the parameter $\tau$ tends to zero, we are interested in graphs with edge lengths close to multiples of $2$. 
\addtocounter{equation}{1}
\begin{definition}\label{deltadefn}
 Let $\Gamma$ be a finite, balanced graph.  We say $\Gamma$ is a \emph{central graph} if each edge in $E(\Gamma)$ has even integer length. 
 Define $l:E(\Gamma) \to \mathbb{Z}^+$ such that for $e \in E(\Gamma)$, $2l(e)$ equals the length of $e$.  
 \end{definition}Finally, we note the conditions needed to guarantee embeddedness.
 \addtocounter{equation}{1}
 \begin{definition}
 We say $\Gamma$ is \emph{pre-embedded} if it is a central graph with $\hat \tau :E(\Gamma)\cup R(\Gamma) \to \Real^+$ and 
 \begin{enumerate}
  \item For all $p \in V(\Gamma)$ and all $e_i \neq e_j \in E_p$, $\angle(\Bv_{e_i,p} , \Bv_{e_j,p}) \geq \pi/3$, where $\angle(\xX,\yY)$ 
measures the angle between the two vectors $\xX,\yY$.
\item There exists $\epsilon>0$ such that for all $e,e' \in E(\Gamma) \cup R(\Gamma)$ that do not share any common endpoints, the Euclidean distance between $e,e'$ is greater than $2+\epsilon$.
\item For any two rays $e\in E_p, e' \in E_{p'}$, $1-\Bv_{e,p}\cdot \Bv_{e',p'} > \epsilon$. 
 \end{enumerate}
 
\end{definition}
The conditions imposed on a pre-embedded $\Gamma$ are the weakest we can currently allow to guarantee embeddedness of the CMC surface. The need for an $\epsilon>0$ in the second item comes from the fact that the maximum radius of an embedded Delaunay surface is on the order $1- \tau + O(\tau^2)$ but we allow for the edges to move with order $\underline C \tau$ where $\underline C$ can be quite large. The first angle condition does not require an added $\epsilon$ because the change in the period for $\tau$ small (on the order $-\tau \log \tau$) dominates both the radius change and the changes we allow via unbalancing and dislocation (on the order $\tau$). The final condition is necessary because two parallel rays pointing into the same half-plane may cross in the process of unbalancing.

\subsection*{Deforming the Graphs}
In the construction process, we deform the central graph $\Gamma$ via the parameters $d,\boldsymbol \zeta$. We need this deformation to be continuous and therefore are interested in graphs $\Gamma$ that allow such continuous deformations. We describe how to deform a central graph based on two functions $\hd, \ell$. The function $\hd$ will correspond to $d/\outau$ and the function $\ell$ will depend on $\boldsymbol \zeta, \hat \tau$ and $l(e)$. Finally, the order of $\ell$ ($-\tau \log \tau$) will be much smaller than the order of $\hd$ (order 1).
\addtocounter{equation}{1}
\begin{definition}
Let $L(\Gamma)$ be a space of functions from $E(\Gamma)$ to $\Real$ with norm defined such that for $\ell \in L(\Gamma)$,
\[
 ||\ell||_L= \max_{e \in E(\Gamma)} \frac{|\ell(e)|}{l(e)}.
\]
If $\Gamma, \widetilde \Gamma$ are isomorphic graphs with corresponding edges $e, e'$, where $\Gamma$ is a central graph, we define $\ell_\Gamma \in L(\Gamma)$ such that
\[
 2l(e) + 2\ell_\Gamma(e)
\]is the length of the edge $e'$.
\end{definition}
\addtocounter{equation}{1}
\begin{definition}
 For a central graph $\Gamma$ we let $D(\Gamma)$ denote a space of functions from $V(\Gamma)$ to $\Real^3$. Let $\Gamma, \widetilde \Gamma$ be isomorphic graphs.
Then $D(\Gamma), D(\widetilde \Gamma)$ can be identified in the natural way so that $\hd_{\widetilde \Gamma} \in D(\Gamma)$. Moreover we define, for any $\hd \in D(\Gamma)$,
\[
 ||\hd||_D= \max_{p \in V(\Gamma)}|\hd(p)|.
\]
\end{definition}
We now define the types of graphs we will use.
\addtocounter{equation}{1}
\begin{definition}\label{flexdef}
Let $\Gamma$ be a central graph. We say $\Gamma$ is \emph{flexible} if there exists $\varepsilon>0$ such that for any $\ell:E(\Gamma) \to [-\varepsilon, \varepsilon]$ and $\hd:V(\Gamma) \to \{ \Bv \in \Real^3 | |\Bv|\leq \varepsilon\}$, there exists a graph $\Ghdl$ isomorphic to $\Gamma$ and smoothly dependent on $(\hd,\ell)$ such that:
\begin{enumerate}
\item If $\tilde \Gamma = \Ghd$ then $\hd_{\tilde \Gamma} =d$.
\item If $\tilde \Gamma = \Ghdl$ then $\ell_{\tilde \Gamma}=\ell$. 
\end{enumerate}

\end{definition}
Any flexible $\Gamma$ will have an associated family of graphs $\mathcal F(\Gamma)$. The existence of this family follows immediately from the definitions.
\addtocounter{equation}{1}
\begin{definition}\label{FamilyDefinition}
 We define a {\it family of graphs}, $\mathcal{F}(\Gamma)$, to be a collection of graphs containing a central graph $\Gamma$, parameterized by a ball  
 $B(\Gamma) \subset D(\Gamma) \times L(\Gamma)$, in a small neighborhood of the origin:
\begin{enumerate}
 \item $\Gamma=\Gamma(0,0)$.
\item $\Ghdl$ is isomorphic to $\Gamma(0,0)$ and depends smoothly on $(\hd, \ell)$.
\item For $\widetilde \Gamma = \Ghd$, $\hd_{\widetilde \Gamma} = \hd$.
\item For $\widetilde \Gamma = \Ghdl$, $\ell_{\widetilde \Gamma} = \ell$.

\item If $e, e'$ are corresponding edges or rays for $\Ghd, \Ghdl$ then
\addtocounter{theorem}{1}
\begin{equation}\label{tauunchanged}
 \hat \tau_{\Ghd} (e) = \hat \tau_{\Ghdl} (e').
\end{equation}
\end{enumerate}
\end{definition}We remark that while \eqref{tauunchanged} implies that the function $\hat \tau_{\Ghd}$ equals the function $\hat \tau_{\Ghdl}$ on corresponding edges, the two graphs $\Ghd$ and $\Ghdl$ may have different balancing at each vertex. That is, we do not expect that $\hd_{\Ghd}(p) = \hd_{\Ghdl}(p')$ for corresponding $p,p'$. The change in unbalancing must be allowed as $\ell$ will change the lengths of the edges and thus possibly the position of edges and rays about a vertex. The next lemma makes precise by how much the two unbalancing parameters can disagree.

\addtocounter{equation}{1}
\begin{lemma}\label{hunbalancing}
 Consider $\Ghdl \in \mathcal F(\Gamma)$ such that $(\hd,\ell)$ are in a small neighborhood of the origin. Then there exists $C=C(\Gamma)>0$ depending only on the graph $\Gamma$ such that
\addtocounter{theorem}{1}
\begin{equation}
 ||\hd_{\Ghd}-\hd_{\Ghdl}||_D < C(\Gamma) ||\ell||_L.
\end{equation}
\end{lemma}
\begin{proof}
Notice that for a fixed $p \in V(\Gamma)$,
\[|\hd_{\Ghd}(p_\hd)-\hd_{\Ghdl}(p_\ell)|\leq \sum_{e_\hd \in E_p} \hat \tau(e_\hd)|\Bv_{e_\hd,p_\hd}-\Bv_{e_\ell,p_\ell}|\]
where $p_\hd \in V(\Ghd), e_\hd \in E_{p_\hd}$ and $p_\ell \in V(\Ghdl)$, $e_\ell \in E_{p_\ell}$ are corresponding vertices and edges or rays. Then the smooth dependence on $\ell$ implies that in a small neighborhood of the origin one immediately has
\[
|\Bv_{e_\hd,p_\hd}-\Bv_{e_\ell,p_\ell}| \leq C||\ell||_L.
\]As the graph is finite, the value $\hat \tau(e_d)$ is absorbed into $C(\Gamma)$ and the result follows.
\end{proof}

We now choose a frame associated to each edge in the graph $\Gamma$ and use this frame to determine a frame on each edge for any graph in $\mathcal F(\Gamma)$.
\addtocounter{equation}{1}
\begin{definition}\label{gammaframe}
For $e \in E(\Gamma) \cup R(\Gamma)$ we choose once and for all an ordered orthonormal frame 
$F_\Gamma [e]=\{\mathbf{v}_{1}[e], \mathbf{v}_{2}[e], \mathbf{v}_{3}[e]\}$ such that $\mathbf v_1[e]$ is parallel to $e$. Moreover, for $e \in R(\Gamma)$,
 $\Bv_1[e]=\Bv_{e,p}$.  
\end{definition}
Because every edge is associated with two vertices, the choice of frame determines a direction along $e \in E(\Gamma)$ that allows is to distinguish between the two vertices. For convenience, we set the following notation. 
\addtocounter{equation}{1}
\begin{definition}\label{pmvertexdef}
 For $e \in E(\Gamma)$ let $p^\pm[e]$ be the  endpoints of $e$ such that
\[\Bv_{e,p^+[e]} = \Bv_1[e]=-\Bv_{e,p^-[e]}.\]
For $e \in R(\Gamma)$, let $p^+[e]$ be the endpoint of $e$.

Also, let $\sep\in \{-1,1\}$ be defined by
\[
\sep:= \Bv_{e,p} \cdot \Bv_1[e].
\]
\end{definition}

We define a general rotation that will be used repeatedly throughout the paper. 
\addtocounter{equation}{1}
\begin{definition}\label{rotationdefn}
Given two unit vectors $\xX, \yY \in \Real^3$ such that $\angle(\xX,\yY)<\pi$,  let $\RRR[\xX, \yY]$ denote the unique rotation such that
\begin{itemize}
\item If $\xX \times \yY \neq 0$, 
let $\RRR[\xX, \yY]$ be the rotation that takes $\xX$ to $\yY$ by choosing the smallest rotation from $\xX$ to $\yY$ about the axis made by $\xX \times \yY$. 
\item 
 If $\xX = \yY$, take $\RRR[\xX,\yY]$ to be the identity.
 \end{itemize}
\end{definition}

Finally, we use the rotation defined above to describe an orthonormal frame on any graph in the family $\mathcal F(\Gamma)$. Notice that as $\angle(\Bv_{e,p} ,\Bv_{e',p'}) <\pi$ for any $[p,e] \in A(\Gamma)$ and corresponding $[p',e'] \in A(\Ghdl)$, the rotation can be defined. 
\addtocounter{equation}{1}
\begin{definition}\label{FrameLemma}\label{FrameDefn}
 For $(\hd,\ell)\in B(\Gamma)$ and $e \in E(\Gamma) \cup R(\Gamma)$ we define an orthonormal frame $F_{\Ghdl}[e]=\{\Bvp_1\ephdl, \Bvp_2\ephdl, \Bvp_3\ephdl\}$
uniquely by requiring the following:
\begin{enumerate}
\item $\Bvp_1\ephdl=\Bv_{e', p'}$ where $e', p'$ are the edge (or ray) and vertex in the graph $\Ghdl$ corresponding to $e,p^+[e]$ on $\Gamma$.
\item $\Bvp_i\ephdl=\RRR[\Bv_1[e], \Bvp_1\ephdl](\Bv_i[e])$ for $i=2,3$.
\end{enumerate}
\end{definition}
\addtocounter{equation}{1}
\begin{remark}
 $F_{\Ghdl}[e]$ depends smoothly on $\hd, \ell$.
\end{remark}

\subsection*{Examples of Graphs}
As we wish to highlight the novel surfaces produced by our construction, we specifically mention graphs that produce embedded  CMC surfaces.

Our first examples are graphs that produce surfaces in $\mathcal M_{0,k}$ where $k \in \{3,\dots 12\}$; we do not need to presume any symmetry (though when $k=12$ embeddedness is only guaranteed when the rays are placed symmetrically). We begin with $\Gamma$ having one vertex and $k$ rays emanating from that central vertex. The vertex position provides three free parameters and there are three free parameters for each of $k-1$ rays. The balancing condition implies that this will entirely fix the $k$-th ray both in magnitude and direction. Modulo rigid motions, there are $3k-6$ continuous parameters for each of these constructions. 
\begin{figure}[h]\label{GraphExamples}
\includegraphics[width=2in]{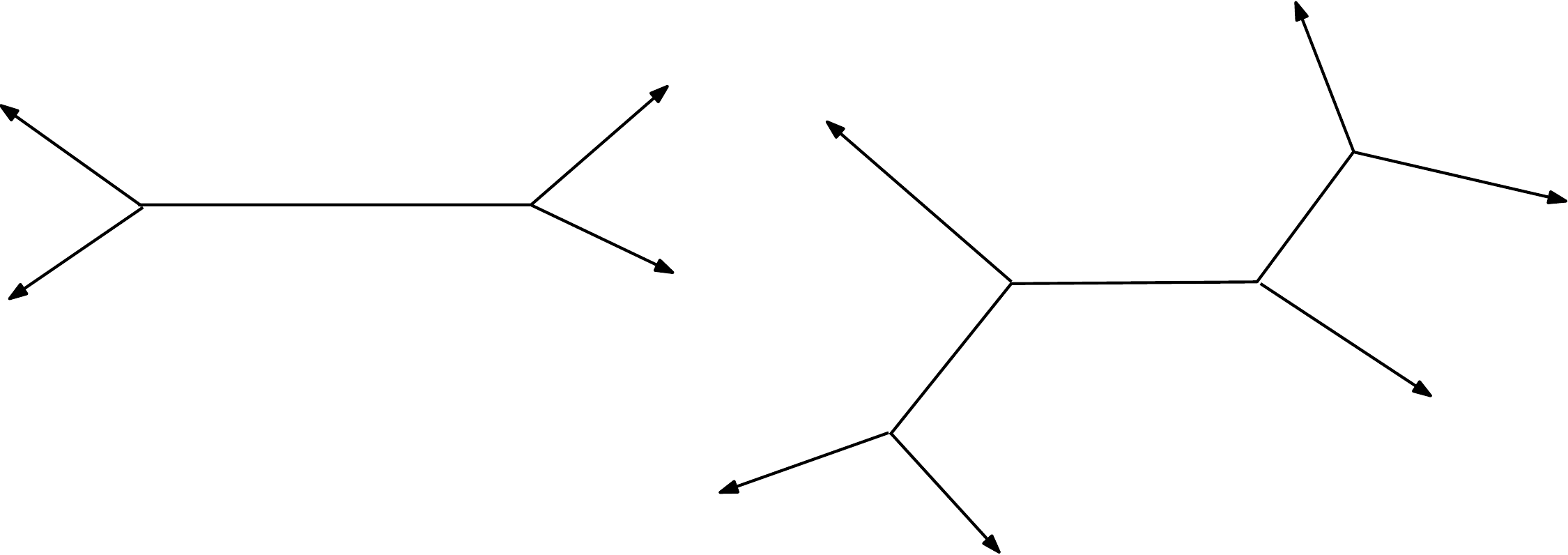}
            \caption{Examples of pre-embedded $\Gamma$ that produce surfaces in $\mathcal M_{0,4}$, $\mathcal M_{0,6}$.}
      \end{figure}
      
We also consider graphs with a fixed number of edges and rays in $\mathcal M_{0,k}$ with restrictions on $k$ coming from the number of edges in the graph. See Figure \ref{GraphExamples} for a few examples. In the first setup, we consider two vertices, one edge, and two rays emanating from each endpoint of the edge. The graph is not presumed to be planar, though we draw it as such for convenience. Choosing the position of the first vertex allows for three free parameters; with a fixed edge length, the position of the remaining vertex has two free parameters. We are free to choose $\hat \tau$ on the edge. Moreover, at each vertex we are free to choose one ray direction and $\hat \tau$, resulting in six more free parameters. These choices and the balancing condition entirely fix the direction and weights of the remaining two rays. Modulo rigid motion, there are $6=3\cdot 4-6$ free parameters and one discrete parameter. We provide another genus zero example in the second drawing in this same figure. We prescribe three edge lengths and position four vertices to accommodate these edges. The positioning of the vertices gives us $9$ free parameters. Notice the edges need not be coplanar. To the interior vertices we add one ray and to the outer vertices we add two rays. If we prescribe $\hat \tau$ on each edge, the remaining $6$ degrees of freedom come from the position and weight of two rays -- one ray at each of the vertices that contain two rays. Modulo rigid motions we have $12=3\cdot 6-6$ degrees of freedom. We note that each of these examples obviously demonstrate the required flexibility for performing a construction. We also note the pre-embedded conditions can be easily satisfied in each case.

We can also construct an example of a finite, balanced, flexible, pre-embedded $\Gamma$ with one edge and up to 22 rays -- up to 11 rays at each vertex. With 11 rays at each vertex, embeddedness is guaranteed by positioning the rays and edge symmetrically (pointing to the center of the faces of a regular dodecahedron) and assigning the same parameter $\hat \tau$ to each edge and ray. We are free to rotate one of the ray configurations about the edge and thus can guarantee no two rays are both parallel and pointing into the same half-plane. Moreover, for an edge of sufficient length, we can presume all the rays emanating from one vertex are of distance greater than 2 from rays emanating from the other vertex. 

By this method, we can easily produce infinitely many pre-embedded, flexible, finite central graphs that allow us to construct an embedded surface in $\mathcal M_{0,k}$. Generically, we may choose $e$ edges and position $e+1$ vertices to accommodate these edges. At each interior vertex, we can position between one and nine rays and at each final vertex we can position between two and ten rays. The limitation on the number of rays allows us to easily determine $\hat \tau$ on each edge and on the free rays so that the graph is balanced and pre-embedded. If we presume the necessary symmetry in the positioning of edges and rays, we can position 11 rays at the outer vertices and 10 rays at the interior vertices. We may have to increase the edge lengths in certain cases to guarantee pre-embeddedness.

We can prescribe genus in the simplest way by introducing an equilateral triangle into the graph. For embedded surfaces, we require the equilateral condition. In Figure 3 we provide examples of central graphs that produce embedded CMC surfaces with genus 1, 2, and 3. For genus 1, we begin with an equilateral triangle which assigns to each edge the same length $l(e_1)$. Designate the parameter $\hat\tau_i>0$ at $e_i$. Then the position and weight of each ray is predetermined and we see the graph has three continuous and one discrete parameter. For genus 2, we begin with two equilateral triangles, sharing a common edge. We will position the minimum number of rays needed to satisfy the balancing condition presuming all weights are positive. We determine the edge length $l(e_1)$ and choose a weight $\hat\tau_i>0$ for each $e_i$. Notice this provides us with five continuous parameters, but in fact we have one more. In this case, $\Gamma$ need not be a planar graph. Thus, the angle between the planes containing the two triangles provides the final continuous parameter. With these six parameters, the balancing condition determines the position and weight of each of the four rays. We can construct an embedded genus 3 CMC surface in two ways. We use either a central graph $\Gamma$ containing three equilateral triangles and at least five rays, or we can use a tetrahedral structure made out of equilateral triangles and at least four rays. We can also add more rays at each vertex which will increase the number of continuous parameters by three. In the case of three triangles and five rays, the graph $\Gamma$ has one discrete parameter (edge length) and nine continuous parameters. Seven continuous parameters come from assigning a weight to each edge and the remaining two come by assigning an angle between the adjacent triangles. For the tetrahedron example, there exist six continuous parameters by assigning $\hat\tau_i>0$ for each edge $e_i$. When we position one ray at each vertex, this completely determines the surface. When we position additional rays at each vertex, the number of continuous parameters goes up by three.
\begin{figure}[h]\label{GenusFigs}
\includegraphics[width=3in]{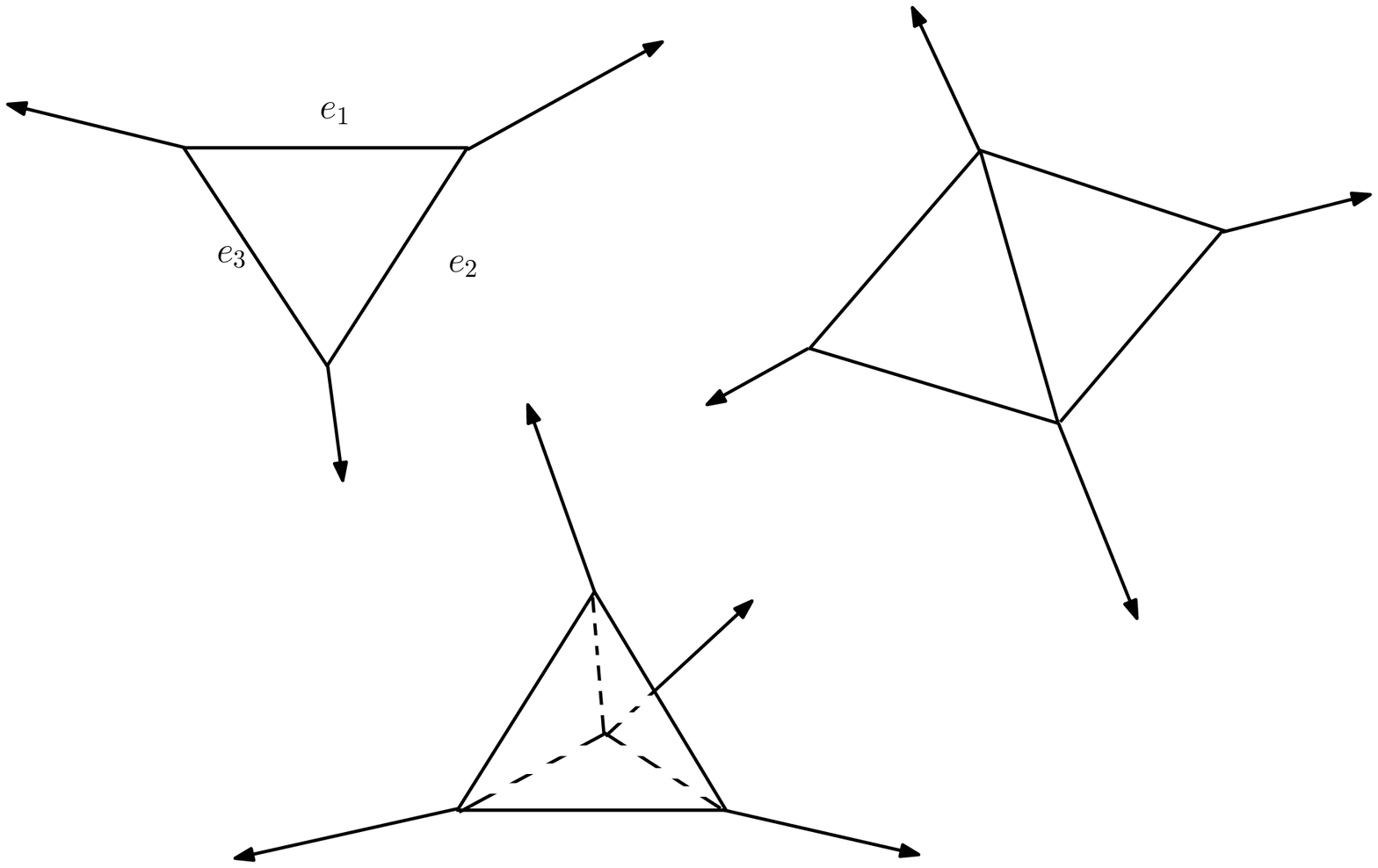}
            \caption{Graphs that produce surfaces in $\mathcal M_{1,3}, \mathcal M_{2,4}, \mathcal M_{3,4}$.}
      \end{figure}
Notice we can use the tetrahedral structure to build more non-planar finite graphs with $v$ vertices and genus $2v-5$. From these graphs, we construct embedded surfaces in $\mathcal M_{2v-5,v}$. 
\addtocounter{equation}{1}
\begin{lemma}
Given $v\geq 4$, there exists an {\bf embedded} $\Sigma \in \mathcal M_{2v-5,v}$.
\end{lemma}
\begin{proof}
We proceed by induction. For $v=4$, we choose a single tetrahedron made of equilateral triangles with edge length $l$. We put the same weight $\hat\tau>0$ on each edge; the symmetry of the construction guarantees the rays point out radially in the direction connecting the barycenter (of the solid tetrahedron) with the endpoint of the ray. Therefore, no two rays are parallel and the graph is pre-embedded. Flexibility is easily see as any small $\hd$ can be achieved by varying the ray weight and direction and changes to $\ell$ can be accomplished by sliding vertices along the various edges or extensions of edges.

At each stage of the process, we add one vertex and three edges, increasing the genus by two. Suppose for $v$ vertices we have a flexible, pre-embedded, central graph $\Gamma_v$. A quick calculation shows the genus of the embedded surface we construct is $2v-5$. We add one vertex $v_1$ and three edges $e_1,e_2,e_3$ to $\Gamma_v$. We presume that each of the weights on the edges of $\Gamma_v$ remain fixed; the addition of the three edges provides three free parameters at this step. Notice that each edge $e_i$ has an endpoint at a vertex on $\Gamma_v$ and an endpoint on $v_1$. Thus, the choice of $\hat\tau_i>0$ at each $e_i$ will influence the direction of four total rays. The pre-embedded conditions the rays must satisfy are the non-parallel condition and the distance condition. Thus, we have an open set of choices for the three $\hat\tau_i$ and only finitely many conditions to satisfy. Therefore, we can easily determine $\hat\tau_i>0$ such that $\Gamma_{v+1}$ is pre-embedded. The flexibility also follows easily as again we can prescribe any small $\hd$ by varying the ray direction and weight. Moreover, by adding a single tetrahedron to an already flexible graph, we are still free to smoothly vary $\ell$ by sliding vertices along the particular edge of interest.

Thus, for any $v$, we produce a pre-embedded, flexible, central graph $\Gamma$ with $v$ vertices and $3v-6$ edges. The graph has $3v-6$ free parameters, coming from prescribing the $\hat \tau$ parameter to each edge. Moreover, the resulting embedded CMC surface has genus $2v-5$.
\end{proof}

\section{The Building Blocks}\label{BuildingBlocks}
The surface we construct is built out of appropriately glued pieces of spheres and perturbed Delaunay surfaces. The positioning of these pieces and
the parameter that describes the Delaunay immersion associated to each edge or ray rely on the graph $\Gamma$ and the two parameters $d,{\boldsymbol \zeta}$. However, one can determine appropriate building blocks without any
initial reference to the graph $\Gamma$. To highlight this fact, we first develop immersions of the building blocks that depend upon more general parameters but not on any graph. In Section \ref{InitialSurface} we explain how to use these general immersions, given a flexible, central graph $\Gamma$ and parameters $d,\boldsymbol \zeta$ to construct a complete surface in $\Real^3$ on which we prove the main theorem.

\subsection*{Spherical Pieces}We first describe immersions of spheres with geodesic disks removed. In the initial immersion, these spherical pieces will be positioned at the vertices of the perturbed graph. Moreover, the centers of the geodesic disks will correspond to directions from which edges and rays of this graph emanate. As we are currently avoiding any reference to the graph, we consider immersions where geodesic disks are removed based solely on a collection of vectors.

For a fixed, small $\delta>0$, choose $a$ such that
\addtocounter{theorem}{1}
\begin{equation}\label{aeq}
\tanh (a+1) =  \cos \delta/8.
\end{equation}
Then for the sphere embedding we describe in Definition \ref{sphere}, $Y_0(\{a+1\}\times \mathbb{S}^1) \subset \mathbb{S}^2$ is the boundary of a geodesic disk of radius $\delta/8$ centered at $(1,0,0)$. 

Consider a collection of unit vectors $V=\{\xX_1, \dots, \xX_n\}$ such that $\angle( \xX_i , \xX_j) > 3\delta$ for $i\neq j$. For such collections of vectors, we will be particularly interested in the set
\addtocounter{theorem}{1}
\begin{equation}\label{centralspheq}\mathbb S^2_V = \mathbb{S}^2 \backslash \coprod_{i=1,\dots, n} D_{\delta/8}^{\Ss^2}(\xX_i)
\end{equation}
where $D_{r}^{\Ss^2}(\xX)$ denotes a geodesic disk on the unit sphere with radius $r$, centered at $\xX$. Notice the condition 
$\angle( \xX_i , \xX_j) >  3\delta$ for $i\neq j$ implies the distance between the removed geodesic disks is greater than $2\delta$.

\addtocounter{equation}{1}
\begin{prop}
Let $V=\{\xX_1, \dots, \xX_n\}$ be a collection of unit vectors in $\Real^3$ 
such that $\angle(\xX_i , \xX_j)>3\delta$ for $i \neq j$. If
$V' = \{\yY_1, \dots \yY_n\}$  such that $\angle(\xX_i , \yY_i) < \delta/16$ for $i=1, \dots, n$, then there exists a family of
diffeomorphisms $\hat Y[V,V']: \Ss^2_V \to \Ss^2_{V'}\subset \Real^3$, smoothly dependent on $V'$,   with the following property:

 \noindent For $i =1, \dots, n$ and $x \in D^{\Ss^2}_{\delta/4}(\xX_i)\backslash D^{\Ss^2}_{\delta/8}(\xX_i)$, 
\[
 \RRR[\xX_i,\yY_i](x)=\hat Y[V,V'](x).
\]
\end{prop}For ease of notation, let $\hat Y[V,V]$ denote the identity map on $\Ss^2_V$.
\begin{proof}
For $V, V'$ as described, consider a smooth cutoff function $\psi_V:\Ss^2_V \to \Real$ defined so that $\psi_V(x)=1$ for all $x \in \Ss^2 \backslash (\cup_i D^{\Ss^2}_{\delta/3}(\xX_i))$,
$\psi_V(x)=1$ in a neighborhood of each $\partial D^{\Ss^2}_{\delta/3}(\xX_i)$, $\psi_V(x)=0$ in a neighborhood of each $D^{\Ss^2}_{\delta/4}(\xX_i)$ and smoothly transits between these values on each annulus centered at $\xX_i$. Define
\begin{equation*}
\hat Y[V,V'](x):=\left\{ \begin{array}{ll} 
 \hat Y[V,V](x) & \text{if } x \in \Ss^2 \backslash \cup_{\xX_i \in V}D^{\Ss^2}_{\delta/2}(\xX_i)\\
\psi_V(x)\hat Y[V,V](x) + (1-\psi_V(x)) \RRR[\xX_i,\yY_i](x)& \text{if } x \in D^{\Ss^2}_{\delta/2}(\xX_i)\backslash D^{\Ss^2}_{\delta/4}(\xX_i),\\
\RRR[\xX_i,\yY_i](x)& \text{if } x \in D^{\Ss^2}_{\delta/4}(\xX_i)\backslash D^{\Ss^2}_{\delta/8}(\xX_i)
\end{array}
\right.\end{equation*} 
The main observation on smoothness comes by noting that the matrix $\RRR[\xX_i, \yY_i]$ is determined completely from smooth functions on $\xX_i \cdot \yY_i$ and from the components of $\xX_i \times \yY_i$. Since these values are smooth in $\yY_i$, the proposition holds.
\end{proof}
In order to deal with a small annoyance that will manifest when we identify pieces of the abstract surface $M$, we must introduce a second diffeomorphism. The process involves post-composing the previous map with a diffeomorphism that slightly rotates the image of the annuli under $\hat Y[V,V']$. In this way we guarantee that the diffeomorphism sends one specified orthonormal frame to a second specified orthonormal frame.
\addtocounter{equation}{1}
\begin{prop}
 Let $V_1 \times V_2=\{(\xX_1^1,\xX_1^2), \dots, (\xX_n^1, \xX_n^2)\}$ be a collection of pairwise associated unit vectors in $\Real^3$ such that $\xX_i^1 \cdot \xX_i^2 =0$ for 
$i=1, \dots, n$ and $\angle(\xX_i^1, \xX_j^1)>3\delta$ for $i \neq j$. Let $V_1' \times V_2'=\{(\yY_1^1, \yY_1^2), \dots, (\yY_n^1, \yY_n^2)\}$ be a second collection of pairwise
associated unit vectors with $\yY_i^1\cdot \yY_i^2=0$ for $i=1, \dots, n$ such that $\angle(\xX_i^j , \yY_i^j) < \delta/16$ for $i=1, \dots, n$ and $j=1,2$.
There exists a family of diffeomorphisms, $\underline Y[V_1 \times V_2, V_1'\times V_2']:\Ss^2_{V_1} \to \Ss^2_{V_1'} \subset \Real^3$, smoothly dependent on
$V_1', V_2'$, with the following property:

\noindent For $i =1, \dots, n$ and $x \in D^{\Ss^2}_{\delta/4}(\xX_i^1)\backslash D^{\Ss^2}_{\delta/8}(\xX_i^1)$, 
\[
 \underline Y[V_1 \times V_2, V_1'\times V_2'](x) = \RRR[\hat Y[V_1, V_1'] (\xX_i^2), \yY_i^2] \circ \RRR[\xX_i^1, \yY_i^1](x).
\]
\end{prop}
\begin{proof}
The existence of such a diffeomorphism follows by using the same cutoff function and logic from the previous proposition.
\end{proof}
Observe that since $\hat Y[V_1, V_1'] (\xX_i^2)$ is orthogonal to $\yY_i^1$, the second rotation applied to each annulus fixes
the vector $\yY_i^1$ for each $i$.

\subsection*{Delaunay Surfaces}\label{DelSection}
While descriptions of Delaunay surfaces are readily available in the literature,
our construction will require an understanding of our choice of conformal immersion into $\Real^3$.  To that end, we present, albeit in an abbreviated form, the conformal map and fundamental
geometric quantities related to it. 


For $\tau \in (0,1/4]$ define the immersion $Y_\tau: \Real \times \mathbb{S}^1 \to \Real^3$ 
\addtocounter{theorem}{1}
\begin{equation}\label{DelEq}
Y_\tau(t,\theta) = \left(k(t), r(t) \cos \theta, r(t) \sin \theta\right)
\end{equation}
where
\begin{equation*}
 r(t)= \sqrt{\tau} e^{w(t)},
\end{equation*}
\begin{equation*}
\left\{
 \begin{array}{l}
 k'(t)=  2\sqrt{\tau}r(t) \cosh w(t)\\
 k(0)=0,
 \end{array}
\right . \end{equation*}
with 
\addtocounter{theorem}{1}
\begin{equation}\label{weq}
\left\{\begin{array}{l}
({w'})^2 +4\tau\cosh^2 w = 1 \\w(0)>0,\: w'(0)=0.\end{array}\right.
\end{equation}

As a result, for any $\tau \in  (0,1/4]$, we have Gauss map, metric, and second fundamental form
\addtocounter{theorem}{3}
\begin{align}\label{normalvector}
 \nu_\tau(t,\theta)&=\left( -w', 2\sqrt{\tau}\cosh w \cos \theta, 2\sqrt{\tau}\cosh w \sin \theta \right),\\
g_\tau &= \tau e^{2w}(dt^2 + d\theta^2),\\
 A_\tau &= -2\tau e^w(\sinh w\: dt^2 + \cosh w\: d\theta^2).
\end{align}Notice for $\tau<0$ we replace each $\cosh w$ with a $\sinh w$ and each $\tau$ with $|\tau|$.

With this definition, we see $H \equiv 1$ and
\addtocounter{theorem}{1}
\begin{equation}\label{curvatureeq}
|A_{\tau}|^2 = 2+2e^{-4w}, \qquad K_{\tau}=1-e^{-4w}.
\end{equation}Note the sign on the curvature of the surface corresponds precisely with the sign of $w$.

By the nature of the equation and based on its initial conditions, $w$ is periodic with period we designate $4P_\tau$. Moreover, $w$ has even symmetry about $t=0$ and odd symmetry about $t=P_\tau$.
  It has a maximum at $t=0$ and a
minimum at $t=2P_\tau$.
The periodicity of $w$ is preserved in the image surface, and we let $1+p_\tau=k(2P_\tau)$.  Thus, the period of the image surface is
$2+2p_\tau$.

\addtocounter{equation}{1}
\begin{lemma}
 A Delaunay surface with parameter $\tau$, as described above, is rotationally symmetric about the $x_1$ axis and is embedded for $\tau \in (0,\frac{1}{4}]$. Let 
$r_\tau^{max}, r_\tau^{min}$ denote the largest and smallest radii of the circles in the $x_2, x_3$ plane. Then 
\[
 r_\tau^{max} = 1+O(\tau); \: \: r_\tau^{min} = |\tau| + O(\tau^2).
\]
\end{lemma}
\begin{proof}
 The rotational symmetry and embeddedness follow immediately from the definition of the immersion and fact that $k'>0$ for $\tau \in (0,1/4]$. The ODE for $w$ and initial conditions imply that for $\tau>0$ we have $w(0) = \arccosh\left(\frac{1}{2\sqrt {\tau}}\right)$ and $w(2P_\tau)=-\arccosh\left(\frac{1}{2\sqrt {\tau}}\right)$. Using the log formulation for $\arccosh$,
\[
 r_\tau(0) =\sqrt{\tau}e^{w(0)} = \sqrt{\tau} \frac{1}{2\sqrt{\tau}}\left(1+ \sqrt{1- 4\tau}\right)=\frac{1}{2}\left(1+ \sqrt{1- 4\tau}\right)=1+O(\tau),
\]
\[
 r_\tau(2P_\tau) = \sqrt{\tau}e^{w(2P_\tau)} = \sqrt{\tau}\left(2 \sqrt{\tau}\right) \left(1+\sqrt{1-4\tau}\right)^{-1} = 2\tau\left(2-2\tau + O(\tau^2)\right)^{-1} = \tau +O(\tau^2).
\]Note when $\tau <0$ we replace $\arccosh$ by $\arcsinh$ and each $\tau$ by $|\tau|$. Finally, any instance of $-4\tau$ becomes $+4|\tau|$. 
\end{proof}

We will need a reference for an embedding of $\Ss^2 \backslash \{(\pm 1, 0, 0)\}$ from the cylinder and so state that here.
\addtocounter{equation}{1}
\begin{definition}\label{sphere}
Let $Y_0: \Real\times \mathbb{S}^1 \to \Real^3$ where
\[Y_0(t, \theta) = \left( \tanh t, \sech t\cos \theta, \sech t \sin \theta\right).\]

Then, $g_0=\sech^2t(dt^2 + d\theta^2)$ and $|A_0|^2 = 2$.
\end{definition}
The construction relies on the fact that certain regions of Delaunay immersions possess well understood geometric limits as $\tau \to 0$. We solve the linearized problem with respect to a conformal metric $h$, which behaves on these regions much like the pull back of the Gauss map. In fact, $\nu_\tau^*g$ provides an isometry between the regions $[-b,b]\times \Ss^1$ and $[2P_\tau -b, 2P_\tau +b]\times \Ss^1$ via the map $Y_\tau(t,\theta) \to Y_\tau(2P_\tau -t,\theta)$. (Here $b$ is a large constant, fixed in \eqref{decaymain}.) Therefore, it suffices to understand the asymptotics of the immersion of $[-b,b]\times \Ss^1$.

In \cite{KapAnn}, the regions and their geometric limits are described in some detail. The next lemma will be stated without proof. The interested reader should consult Lemmas 2.1 and 2.2 in Appendix A of \cite{KapAnn} for the details.
\addtocounter{equation}{1}
\begin{lemma}\label{radiuslemma}
Let $r_\tau:[-x_\tau, x_\tau] \to \Real$ be the function whose graph, rotated about the $x_1$ axis, gives $Y_\tau\left([-2P_\tau, 2P_\tau]\times \Ss^2\right)$. 
Then as $\tau \to 0$
\[
p_\tau \to 0, \qquad x_\tau \to 1. 
\]
 Let $r_0(x_1):[-1,1] \to \Real$ be defined by $r_0(x_1) = \sqrt{1-x_1^2}$.
 Given $\epsilon >0$, there exists $\tau_\epsilon>0$ such that if $0<|\tau |< \tau_\epsilon$, then $r_\tau$ restricted to $[-1+\epsilon, 1-\epsilon]$ depends smoothly on $\tau$ and
 \[
 ||r_0-r_\tau:C^k([-1+ \delt, 1- \delt])|| \leq C(\delt,k)|\tau|.
\]
Moreover, we have the following period limits as $\tau \to 0$:
\addtocounter{theorem}{1}
\begin{equation}\label{periodlimits}
\lim_{\tau \to 0} \frac{1}{-\log \tau}\cdot \frac{dp_\tau}{d\tau}=1; \: \: \lim_{\tau \to 0} \frac{p_\tau}{-\tau \log \tau} = 1.
\end{equation}
\end{lemma}
We have the following corollary comparing the metric on the sphere and the Delaunay immersion. 
\addtocounter{equation}{1}
\begin{corollary}\label{taudeltaandb}\label{metriccompg}
For $\delt\in(0,1), k \in \mathbb Z^+$, there exists $\tau_\delt>0$ such that for all $0<|\tau|<\tau_\delt$,  
\addtocounter{theorem}{1}
\begin{equation}\label{gmetricrelation}||g_\tau - g_0:C^k([-1+\epsilon, 1-\epsilon]\times\mathbb S^1),g_\tau)||\leq C(\delt,k)|\tau|.
\end{equation}
\end{corollary}
For a fixed large $b$ (choose for example the largest $b$ such that $\tanh b = 1-\epsilon$), Lemma \ref{radiuslemma} expresses the limit, as $\tau \to 0$, of the immersions of the regions $[4P_\tau n -b,4P_\tau n + b] \times \Ss^1$ for each $n \in \mathbb Z$. Regions of the form $[2(2n-1)P_\tau -b, 2(2n-1)P_\tau+b] \times \Ss^1$ are isometric to these regions in the metric $\nu_\tau^*g_\tau$ under the mapping $Y_\tau(t,\theta) \to Y_\tau(2P_\tau -t,\theta)$. On regions in between, one cannot appeal to natural geometric limiting behavior. Instead, we understand the behavior of these portions of the cylinder in the flat metric $dt^2 + d\theta^2$. We determine the limiting length of such a cylindrical piece in this metric.
\addtocounter{equation}{1}
\begin{lemma}\label{Plemma}There exists $b \gg 1$ such that 
\addtocounter{theorem}{1}
\begin{equation}\label{elllength1}
\lim_{\tau \to 0} \frac{2P_\tau-2b}{-\log \tau}=1.
\end{equation}
\end{lemma}
\begin{proof} For $\epsilon' >0$, \eqref{gmetricrelation} provides $b$ sufficiently large (for any $0<|\tau|<\tau_{\epsilon'}$ small) such that $-w_\tau' \in (1-\epsilon', 1]$ on
$[b,P_\tau]\times \Ss^1$. Thus, we can determine, as $w(P_\tau)=0$,
\[
(P_\tau -b)(1-\epsilon') \leq - \int_b^{P_\tau}w_\tau' \: dt = w_\tau(b) \leq P_\tau -b.
\] 
For $b$ fixed, recall that $|r_0(b) - r_\tau(b)|=\left|\sqrt{1-\tanh^2b} -\sqrt{\tau} e^{w(b)}\right|<C(b)\tau.$ Thus,
\[
1+\frac{\log\left( \sqrt{1-\tanh^2b}-C(b)\tau\right)}{-\log \tau} < \frac{2 w_\tau(b)}{-\log \tau} < 1+ \frac{\log\left( \sqrt{1-\tanh^2b}+C(b)\tau\right)}{-\log \tau}.
\]Letting $\tau \to 0$ proves the result.
\end{proof}

\subsection*{Delaunay Building Blocks} We now describe a general immersion of an appropriately perturbed Delaunay piece. A verbal description of the immersion follows Definition \ref{DelEmbDef}. 
Throughout this subsection we presume a few technical conditions on parameters of interest, which will always be stated in the hypotheses. These conditions are satisfied for the setup in Section \ref{InitialSurface} and are necessary to attain the desired estimates. Throughout this subsection, let $0<\outaupp\ll 1$ be small enough to ensure the immersions described below are smooth, well-defined immersions and let $\underline C$ denote a large universal constant independent of $\outaupp$.

We first define four cutoff functions that will be useful.
\addtocounter{equation}{1}
\begin{definition}
Let
 $\psi_{dislocation^\pm}, \psi_{gluing^\pm}:( t,\theta) \in [a, 4P_\tau l -a] \times \mathbb S^1\to\Real^3$ be cutoff functions such that:
\begin{itemize}
\item $\psi_{dislocation^+}=\psi[a+2,a+1]$,
\item $\psi_{dislocation^-}=\psi[4P_\tau l-(a+2),4P_\tau l-(a+1)]$,
\item $\psi_{gluing^+}=\psi[a+3,a+4]$,
\item $\psi_{gluing^-}=\psi[4P_\tau l-(a+3),4P_\tau l-(a+4)]$.
\end{itemize}
\end{definition}
With these cutoff functions we define the immersion. Observe the smallness conditions on $\tau, |\boldsymbol \zeta^\pm|$ defined at the outset are to ensure the immersion is a smooth, well-defined immersion.
\addtocounter{equation}{1}
\begin{definition}\label{DelEmbDef}
Given $\tau\in \Real \backslash \{0\},a\in \Real^+,l \in \mathbb Z^+, \boldsymbol \zeta^\pm \in \Real^3$ with $0<|{\boldsymbol \zeta}^\pm| \leq \underline C\outaupp,0<|\tau| \leq \outaupp, 1\ll a$, we define two smooth immersions
$ Y_{edge}[\tau,l,{\boldsymbol \zeta}^+,{\boldsymbol \zeta}^-]:[a, 4P_\tau l -a] \times \mathbb S^1 \to \Real^3$ and $ Y_{ray}[\tau,{\boldsymbol \zeta}^+]:[a, \infty) \times \mathbb S^1 \to \Real^3$ such that, for $x=(t,\theta)$,
\begin{align*} Y_{edge}[\tau,l,{\boldsymbol \zeta}^+,{\boldsymbol \zeta}^-](x)=& \psi_{dislocation^+}(t) (Y_0 (x) + {\boldsymbol \zeta}^+)+(1-\psi_{dislocation^+}(t))(1-\psi_{gluing^+}(t))Y_0(x)\\ 
&+\psi_{gluing^+}(t) \cdot \psi_{gluing^-}(t) \cdot Y_\tau(x)\\
&+ (1-\psi_{dislocation^-}(t))(1-\psi_{gluing^-}(t))Y_0^-(x)+\psi_{dislocation^-}(t) (Y_0^- (x) + {\boldsymbol \zeta}^-)                            
\end{align*}
\begin{align*}
 Y_{ray}[\tau,{\boldsymbol \zeta}^+](x)= &\psi_{dislocation^+}(t) (Y_0(x)+ {\boldsymbol \zeta}^+ )+(1-\psi_{dislocation^+}(t))(1-\psi_{gluing^+}(t))Y_0(x)\\
 &+\psi_{gluing^+}(t) \cdot Y_\tau(x)
\end{align*}
where $Y_0^-(x)= Y_0(t-4P_\tau l ,\theta) + \left(2+2p_\tau\right)l \Be_1$.
\end{definition}To aid the reader, we describe the geometry of the $Y_{edge}$ immersion in some detail.  For $t \in [a,a+1]$, the image is a geodesic
annulus sitting on a unit sphere centered at ${\boldsymbol \zeta}^+$. The annulus is centered at ${\boldsymbol \zeta}^+ + \Be_1$ with inner radius $\delta/8$. When $t \in [a+1,a+2]$, the immersion 
smoothly transits between
an annular region  on the dislocated sphere and an annular region on a unit sphere centered at the origin. 
For $t \in [a+2,a+3]$, the immersion remains on the unit sphere centered at the origin, while for $t \in [a+3,a+4]$, the immersion smoothly transits between this sphere and
a Delaunay piece with parameter $\tau$. The same procedure happens toward the other end. First, the Delaunay piece transits back to a unit sphere centered
at $ \left(2+2p_\tau\right)l \Be_1$. This position represents the location of the end of
a Delaunay piece with parameter $\tau$ and $l$ periods, with initial end on the $\{x_1=0\}$ plane and axis on $\Be_1$. Finally,
this sphere transits to a unit sphere centered at ${\boldsymbol \zeta}^- + \left(2+2p_\tau\right)l \Be_1$, a dislocation of ${\boldsymbol \zeta}^-$ from the previously
described sphere.

Of course, the $Y_{ray}$ immersion has the same behavior as $Y_{edge}$ near the origin. The only difference is that the Delaunay immersion continues out to infinity and there is no transiting back to a sphere. 

In determining the initial immersion, the unbalancing parameter $d$ will induce changes in the $\tau$ parameter associated to each edge or ray. As we stress here the independence of the building blocks on the background graph, we consider a diffeomorphism of a piece of the cylinder that will account for length changes based on a change of parameter associated with unbalancing. Again, this diffeomorphism can be determined from parameters that are independent of any graph $\Gamma$. 
\addtocounter{equation}{1}
\begin{definition}Given $\tau, \tau'$ such that $|\tau| < \outaupp$ and 
\addtocounter{theorem}{1}
\begin{equation}\label{tauratio1}
 \frac{\tau}{\tau'}\in \left(1-C\outaupp, 1+C\outaupp\right)
\end{equation}
 for some $C >0$, we define two families of diffeomorphisms
\begin{equation*}
\hat Y_{edge}[\tau,\tau']:[a, 4P_{\tau}l -a]\times \mathbb S^1\to [a, 4P_{\tau'}l-a]\times \mathbb S^1
\end{equation*}and
\begin{equation*}
 \hat Y_{ray}[\tau,\tau']:[a, \infty)\times \mathbb S^1\to [a, \infty)\times \mathbb S^1
\end{equation*}
which depend smoothly on $\tau, \tau'$, 
such that 
\begin{align*}
 \hat Y_{edge}[\tau,\tau'](t,\theta)=&\psi[a+5,a+4](t) \cdot (t,\theta)\\&+ \psi[a+4, a+5](t)\cdot \psi [4P_{\tau}l-(a+4), 4P_{\tau}l-(a+5)](t)\cdot
 \left(  t\frac{P_{\tau'}}{P_{\tau}},\theta \right)\\&+
\psi [4P_{\tau}l-(a+5), 4P_{\tau}l-(a+4)](t)\cdot \left(
 t+4l(P_{\tau'}-P_{\tau}), \theta \right).
 \end{align*}
 and
\begin{equation*} 
 \hat Y_{ray}[\tau,\tau'](t,\theta)=\psi[a+5,a+4](t) \cdot (t,\theta)+ \psi[a+4, a+5](t) \left(  t\frac{P_{\tau'}}{P_{\tau}},\theta \right) \end{equation*}
\end{definition}
 Again, the conditions on $\tau, \tau'$ in the definition above will always hold for immersions of interest to us. Moreover, the following lemma makes clear that by presuming these conditions we can state $C^k$-norm bounds in Proposition \ref{geopropcentral} without involving $\frac{P_{\tau'}}{P_\tau}$. 
 \addtocounter{equation}{1}
\begin{lemma}
\addtocounter{theorem}{1}
\begin{equation}\label{diffeodifference1}
\left| 1 - \frac{P_{\tau'}}{P_{\tau}}\right| \leq \frac{C\outaupp}{-\log \tau}
\end{equation}
\end{lemma}
\begin{proof}
The estimate follows immediately from \eqref{tauratio1} and \eqref{elllength1}.
\end{proof}

 \addtocounter{equation}{1}
\begin{definition}We define two new immersions, smooth in all parameters, 
$\underline Y_{edge}[\tau, \tau',l,{\boldsymbol \zeta}^+,{\boldsymbol \zeta}^-]: [a, 4P_{\tau}l -a]\times \mathbb S^1 \to \Real^3$ and 
$\underline Y_{ray}[\tau, \tau',{\boldsymbol \zeta}^+]: [a, \infty )\times \mathbb S^1 \to \Real^3$ such that
\addtocounter{theorem}{1}
\begin{equation}
\underline Y_{edge}[\tau, \tau',l,{\boldsymbol \zeta}^+,{\boldsymbol \zeta}^-](t,\theta) = Y_{edge}[\tau', l, {\boldsymbol \zeta}^+, {\boldsymbol \zeta}^-] \circ \hat Y_{edge}[\tau, \tau'](t,\theta)
\end{equation}
and
\addtocounter{theorem}{1}
\begin{equation}
\underline Y_{ray}[\tau, \tau',{\boldsymbol \zeta}^+](t,\theta) = Y_{ray}[\tau', {\boldsymbol \zeta}^+] \circ\hat Y_{ray}[\tau, \tau'](t,\theta).
\end{equation}
\end{definition}

\addtocounter{equation}{1}
\begin{prop}\label{geopropcentral}Let $g:= (\underline Y_{edge})^*g_{\Real^3}$ or $g:=( \underline Y_{ray})^*g_{\Real^3}$ as the situation dictates. For a fixed, large constant $b>a+5$,
\[
||\underline Y_{edge}[\tau,\tau',l,{\boldsymbol \zeta}^+,{\boldsymbol \zeta}^-]-Y_0:C^k((a, b) \times \Ss^1, g)|| \leq C(k,b)(|\boldsymbol \zeta^+| + |\tau|)
\]
\[
||\underline Y_{edge}[\tau,\tau',l,{\boldsymbol \zeta}^+,{\boldsymbol \zeta}^-]-Y_0^-:C^k((4P_\tau l -b, 4P_\tau l-a) \times \Ss^1, g)|| \leq C(k,b)(|\boldsymbol \zeta^-| + |\tau|)
\]and
\[
||\underline Y_{ray}[\tau,\tau',{\boldsymbol \zeta}^+]-Y_0:C^k((a, b) \times \Ss^1, g)|| \leq C(k,b)(|\boldsymbol \zeta^+|  + |\tau|).
\]

\end{prop}
\begin{proof}
The statement follows immediately from the definition of the immersions, \eqref{gmetricrelation}, and the conditions on $\tau, \tau'$ which allow us to invoke \eqref{diffeodifference1} as needed.
\end{proof}
\addtocounter{equation}{1}
\begin{prop}
The image surface $\underline Y_{edge}[\tau,\tau',l,{\boldsymbol \zeta}^+,{\boldsymbol \zeta}^-]([a, 4P_\tau l -a] \times \mathbb S^1)$ has mean curvature identically one except when $t\in[a+1,a+2]\cup[a+3,a+4]\cup[4P_\tau l-(a+2),4P_\tau l-(a+1)]\cup[4P_\tau l-(a+4),4P_\tau l-(a+3)]$.

The image surface $\underline Y_{ray}[\tau, \tau', {\boldsymbol \zeta}^+]([a, \infty) \times \Ss^1)$ has mean curvature identically one except when $t \in[a+1, a+2]\cup[a+3, a+4]$.
\end{prop}

\begin{proof}The statement is obviously true for the immersions $Y_{edge},Y_{ray}$ based on their definition. Moreover, $\hat Y_{edge}, \hat Y_{ray}$ do not change the mean curvature of the immersions $Y_{edge}, Y_{ray}$.
\end{proof}

\addtocounter{equation}{1}
\begin{definition}Let $H[edge]$ and $H[ray]$ represent the mean curvature of the surfaces immersed by $\underline Y_{edge}$ and $\underline Y_{ray}$ respectively. Let $H_{error}[edge] = H[edge]-1$ and $H_{error}[ray]=H[ray]-1$. 
Further, let 
\[H_{dislocation}[\cdot ]+H_{gluing}[\cdot] =H_{error}[\cdot]
\]
 where 
\addtocounter{theorem}{1}
 \begin{align*}
  &\supp(H_{dislocation}[edge])\subset\left([a+1,a+2] \cup [4P_\tau l-(a+2),4P_\tau l-(a+1)]\right)\times \Ss^1,\\
&\supp(H_{dislocation}[ray])\subset\left([a+1,a+2]\right)\times \Ss^1,\\
&\supp(H_{gluing}[edge])\subset\left([a+3,a+4] \cup [4P_\tau l-(a+4),4P_\tau l-(a+3)]\right)\times \Ss^1,\\
&\supp(H_{gluing}[ray])\subset[a+3,a+4]\times \Ss^1,
 \end{align*}
where $\supp(f)$ represents the support of the function $f$.
\end{definition}
With this definition, Proposition \ref{geopropcentral} provides the following corollary.
\addtocounter{equation}{1}
\begin{corollary}\label{Hbounds}Let $g:= (\underline Y_{edge})^*g_{\Real^3}$ or $g:=( \underline Y_{ray})^*g_{\Real^3}$ as the situation dictates. Then:
\begin{enumerate}
\item $||H_{gluing}[edge]:{C^k}( [a, 4P_\tau l -a] \times \mathbb S^1),g)|| \leq C(a,k)|\tau|.$
\item $||H_{gluing}[ray]:C^k([a, \infty)\times \Ss^1),g)|| \leq C(a, k)|\tau|$. 
\item $||H_{dislocation}[edge]:C^k( [a, 4P_\tau l -a] \times \mathbb S^1),g)|| \leq C(a,k)\left(|{\boldsymbol \zeta}^+|+ |{\boldsymbol \zeta}^-|\right)$.
 \item $||H_{dislocation}[ray]:{C^k}([a, \infty)\times \Ss^1,g)|| \leq C(a,k)|{\boldsymbol \zeta}^+|$.
\end{enumerate}
\end{corollary}

\section{Construction of the Initial Surface}\label{InitialSurface}
Given a flexible, central graph $\Gamma$ and parameters $d,\boldsymbol \zeta$, we determine an immersion into $\Real^3$ by appropriately positioning the building blocks described in Section \ref{BuildingBlocks}. Because our construction relies on the geometric limit of certain pieces of a Delaunay surface, we need the parameter associated to each $e \in E(\Gamma) \cup R(\Gamma)$ to be small. To that end, we choose a parameter $0<\tau\ll 1$ and define $e \in E(\Gamma) \cup R(\Gamma)$,
\[ \uthte:= \tau \hat \tau(e).\] The argument requires that all $|\uthte|$ be sufficiently small so that we may appeal to various geometric estimates. Thus $\outau$ will depend on $\max \hat \tau$, but not on $d,\boldsymbol \zeta$, or the structure of $\Gamma$. 
The finiteness of $\Gamma$ implies there exists a fixed constant $C$ such that $\tau/C \leq |\uthte| \leq C\tau$. We follow the convention that parameters with a ``hat'' are on scale one and those with no hat are on scale $\outau$.

\subsection*{The Abstract Surface $M$}
Given a flexible, central graph $\Gamma$ with rescaled functions $\uthte$, we define an abstract surface $M$ that will serve as the domain for the immersion into $\Real^3$. $M$ depends only on $\Gamma$ and is independent of the parameters $d, \boldsymbol \zeta$.  The construction of $M$ proceeds as follows. To each vertex $p \in V(\Gamma)$, we associate $M[p]$, a two-sphere with geodesic disks removed. The centers of these geodesic disks correspond exactly with the directions of the vectors $\Bv_{e,p}\in V_p$. The radius of the disks depends upon a $\delta$ we choose based on the graph. To each $e \in E(\Gamma)\cup R(\Gamma)$ we associate a piece of $\Real \times \Ss^1$ with left boundary at $\{a\} \times \Ss^1$ where $a$ depends on $\delta$ as in \eqref{aeq}. If $e$ is an edge, the length of the piece is determined based on the functions $\uthte$ and $l(e)$. Finally, we define the abstract surface $M$ by appropriately identifying a geodesic annulus on $M[p]$ centered at $\Bv_{e,p}$ with a neighborhood of the appropriate boundary of the cylindrical piece $M[e]$. 

 For the graph $\Gamma$, fix a small $0<\delta\ll 1$ such that for each $p \in V(\Gamma)$ and $e,e' \in E_p$, $\angle(e,e')>3\delta$. We now fix once and for all the constant $a$ so that 
 \addtocounter{theorem}{1}
 \begin{equation}
 \tanh(a+1)=\cos(\delta/8).
\end{equation}

\addtocounter{equation}{1}
\begin{definition}For $p \in V(\Gamma)$
define 
\addtocounter{theorem}{1}
\begin{equation}
M[p]= \mathbb S^2_{V_p}
\end{equation}where this notation follows that of \eqref{centralspheq} and $V_p$ is defined in \ref{Vpdefn}.
As the enumeration along a cylinder corresponding to an edge will vary depending on the direction, we simplify notation by defining
\addtocounter{theorem}{1}
\begin{equation}
\RH:= 4P_{\uthte}l(e).
\end{equation}

For $e \in E(\Gamma)$, let
\begin{equation*}
M[e] = [a,\RH-a]\times \mathbb S^1
\end{equation*}
while for $e \in R(\Gamma)$, let
\begin{equation*}
M[e] = [a,\infty)\times \mathbb S^1.
\end{equation*}

\end{definition}
To make the proper identification between $M[e]$ and $M[p]$, for $e \in E_p$, we define a rotation that takes the standard frame in $\Real^3$ to the frame $F_\Gamma[e]$ associated to $e \in E_p$.
\addtocounter{equation}{1}
\begin{definition}
For $e \in E(\Gamma) \cup R(\Gamma)$, let $\RRR[e]:\Real^3 \to \Real^3$ denote the rotation such that
\[
\RRR[e](\Be_i)=\Bv_i[e]
\]for $i=1,2,3$, where here the $\Bv_i[e]$ refer to the ordered orthonormal frame chosen in Definition \ref{gammaframe}.
\end{definition}
With this rotation in hand, we define the abstract surface $M$.
\addtocounter{equation}{1}
\begin{definition}
Let 
\[
M'= \left(\coprod_{p \in V(\Gamma)}M[p]\right) \coprod\left( \coprod_{e \in E(\Gamma) \cup R(\Gamma)} M[e]\right)\]
and let 
\addtocounter{theorem}{1}
\begin{equation}\label{mdef}
M=M'/\sim 
\end{equation}
where we make the following identifications:\\
For $[p,e] \in A(\Gamma)$ with $p=p^+[e]$ and $x \in M[e] \cap \left([a,a+1] \times \mathbb S^1\right)$,
\[
 x \sim \left(\RRR[e] \circ Y_0(x)\right) \cap M[p].
\]
For $[p,e] \in A(\Gamma)$ with $p=p^-[e]$ and $ x=(t,\theta)\in M[e] \cap \left([\RH-(a+1),\RH-a] \times \mathbb S^1\right)$,
\[
 x \sim \left( \RRR[e] \circ Y_0 (t-\RH,\theta )\right) \cap M[p].
\]

\end{definition}

\subsection*{Standard and Transition Regions}Following the strategy in \cite{HaskKap,KapBAMS,KapAnn,KapPNAS,KapWente,KapProc,KapJDG,KapJDGCMC,KapThes,KapClay,KapYang} we carefully define various regions on $M$. We require slightly more burdensome notation than in previously cited work as our setup lacks any symmetry. For each $p \in V(\Gamma)$, the region $S[p]$ is a \emph{central standard region} \cite{HaskKap,KapClay} or a \emph{central almost spherical region} in the terminology of \cite{KapBAMS,KapAnn,KapPNAS,KapWente,KapProc,KapJDG,KapJDGCMC,KapThes}. The geometric limits of certain regions of the immersion $Y_\tau:\Real \times \Ss^1 \to \Real^3$ motivates our identification of various regions on each $M[e]$. 
Recall that $2l(e)$ denotes the length
of each edge on the initial graph $\Gamma$. Thus, each $M[e]$ will have
$2l(e)-1$ geometrically well understood regions and $2l(e)$ regions connecting them. This leads us to define the following sets.
\addtocounter{equation}{1}
\begin{definition}\label{pendef}
We let $[p,e,n]\in S(\Gamma)$ where 
\begin{align*}
 S(\Gamma) := &\{[p,e,n] |e \in E(\Gamma),\ppe\in A(\Gamma), n \in \{1, 2, \dots, l(e)\}\}\\&\cup\{\pen | e \in E(\Gamma),\pme\in A(\Gamma), n \in\{1, 2, \dots, l(e)-1\} \}
 \\&\cup \{\pen | e \in R(\Gamma), [p,e] \in A(\Gamma), n \in \mathbb N\}.
\end{align*}Further, let $[p,e,n']\in N(\Gamma)$
where
\begin{align*}
N(\Gamma):=&\{ [p,e,n']|e \in E(\Gamma),[p^\pm[e],e] \in A(\Gamma), n' \in \{1, 2, \dots, l(e)\}\} \\&\cup \{[p,e,n'] | e \in R(\Gamma), [p,e] \in A(\Gamma), n' \in \mathbb N\}.
\end{align*}Finally, we let $[p,e,n''] \in N^+(\Gamma)$ where
\begin{align*}
N^+(\Gamma):=&\{ [p,e,n'']| e\in E(\Gamma),[p^\pm[e],e] \in A(\Gamma), n'' \in \{0, 1, \dots, l(e)-1\}\}\\&\cup \{[p,e,n'']| e \in R(\Gamma), [p,e] \in A(\Gamma), n'' \in \mathbb N\cup \{0\}\}.
\end{align*}
\end{definition}
With this notation, each $S\pen\subset M$ will correspond to a \emph{standard region} or \emph{almost spherical region}. Each $\Lambda[p,e,n']$ will correspond to a \emph{transition} or \emph{neck} region. 
Notice that $S(\Gamma) \subset N(\Gamma)$. In fact, $N(\Gamma) \backslash S(\Gamma)= \{ [p^-[e],e,l(e)] : e \in E(\Gamma)\}$. We choose this notation so that the set $S(\Gamma)$ enumerates every standard region exactly once. Moreover, the enumeration of the standard regions is such that it increases along $M[e]$ as one moves further away from the nearest boundary. For $e \in E(\Gamma)$, the middle standard region on $M[e]$ bears the label $S[p^+[e],e,l(e)]$. Each $\widetilde S\pen$ is an \emph{extended standard region} and contains both the standard region and the two adjacent transition regions. The $\widetilde S[p]$ are \emph{central extended standard regions} and contain all adjacent transition regions, where adjacency is determined by $e \in E_p$. The $C^+[p,e,n''], C^-[p,e,n']$ represent the meridian circles on the boundary between a standard and transition region. $\widetilde C[p,e,n']$ is the center meridian circle of $\Lambda[p,e,n']$.

Recall that $a$ is determined by $\delta$ in the definition of $M$ and controls the size of the removed geodesic disks. The constant $b$ determines the size of each standard and transition region. We use $x,y$ in subscripts to modify the size of the regions and the boundary circles. For example, $S[p] \subset S_x[p]$ while $\widetilde S_x[p] \subset \widetilde S[p]$.
\begin{figure}[h]\label{STRPthin}
\includegraphics[width=5in]{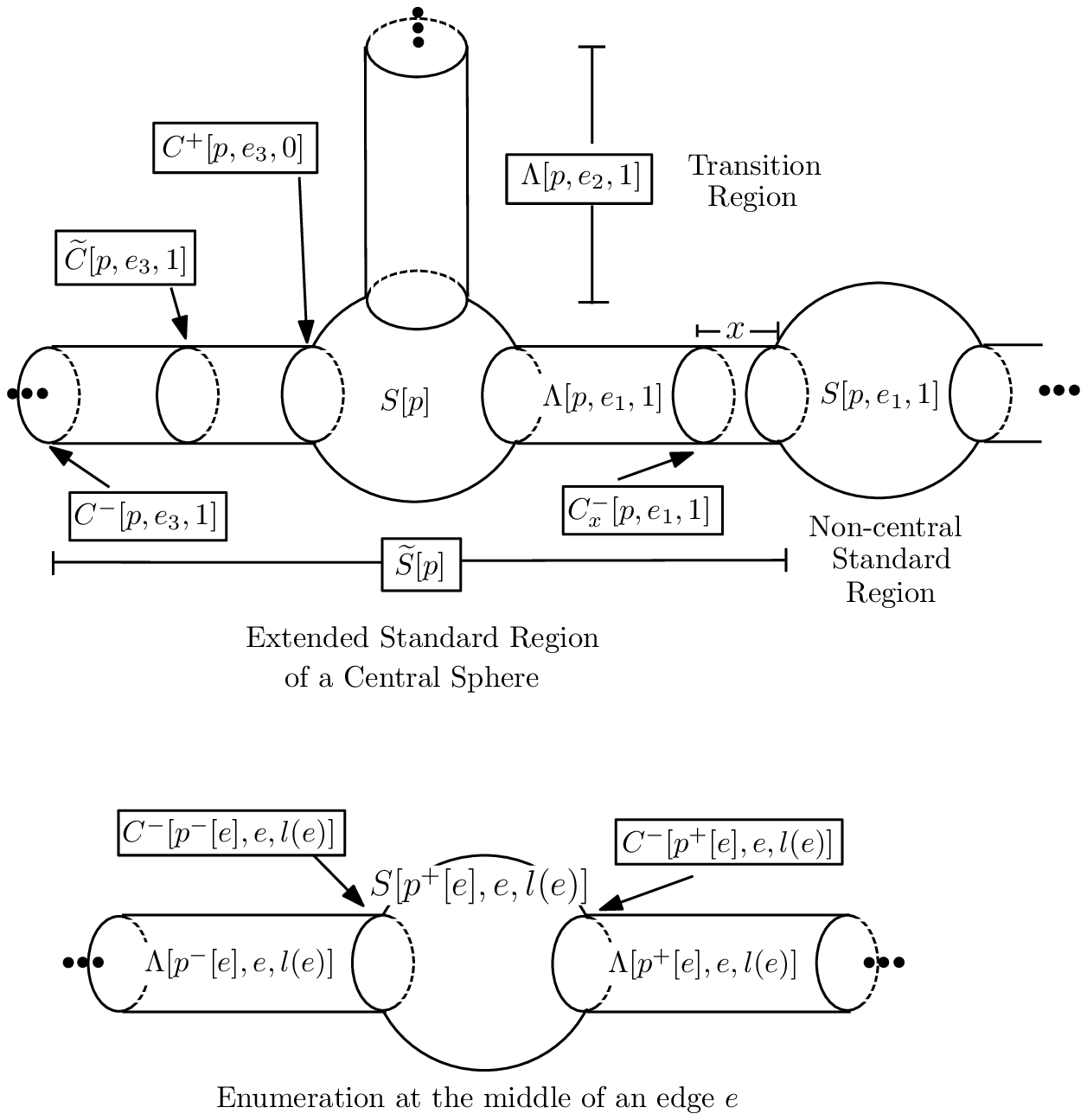}
            \caption{Two schematic renderings of $M$. The top one is near a vertex $p$ with $|E_p|=3$ and the bottom one at the standard region associated to center of an edge $e$. Note that standard regions appear spherical and transition regions appear cylindrical.}
      \end{figure}

\addtocounter{equation}{1}
\begin{definition}\label{regions} Let $p \in V(\Gamma)$. For $S(\Gamma), N(\Gamma), N^+(\Gamma)$ as outlined above, we define the following regions on $M$.
\begin{enumerate}
\item \label{centstand} $S_x[p]:=M[p] \cup_{\{e|p=p^+[e]\}}\left(M[e]\cap[a, b+x]\times \mathbb{S}^1\right)$\\
\indent \indent \indent \indent $\cup_{\{e|p=p^-[e]\}}\left(M[e] \cap [\RH-(b+x),\RH-a] \times \Ss^1\right)$
\item \label{centextstand}$\widetilde S_x[p]:=M[p]\cup_{\{e|p=p^+[e]\}}\left(M[e]\cap[a, 2P_{\uthte}-( b+x)]\times \mathbb{S}^1\right)$
\\ \indent \indent \indent\indent $\cup_{\{e|p=p^-[e]\}}\left(M[e] \cap [\RH-(2P_{\uthte}-( b+x)),\RH-a] \times \Ss^1\right)$
\item \label{stand1}$S_x[p^+[e],e,n]:= M[e]\cap [2nP_{\uthte}-(b+x),2nP_{\uthte }+(b+x)]\times \mathbb{S}^1$
\item \label{stand2}$S_x[p^-[e],e,n]:= M[e]\cap [\RH-(2nP_{\uthte}+(b+x)),\RH-(2nP_{\uthte }-(b+x))]\times \mathbb{S}^1$
\item \label{extstand1}$\widetilde S_x[p^+[e],e,n]:=  M[e]\cap [(2n-2)P_{\uthte}+(b+x),(2n+2)P_{\uthte }-(b+x)]\times \mathbb{S}^1$
\item \label{extstand2}$\widetilde S_x[p^-[e],e,n]:=  M[e]\cap [\RH-((2n+2)P_{\uthte}-(b+x)),\RH-((2n-2)P_{\uthte }+(b+x))]\times \mathbb{S}^1$
\item\label{neckregion1} $\Lambda_{x,y}[p^+[e],e,n']:=  M[e]\cap  [(2n'-2)P_{\uthte }+(b+x),2n'P_{\uthte }-(b+y)]\times \mathbb{S}^1$
\item\label{neckregion2} $\Lambda_{x,y}[p^-[e],e,n']:=  M[e]\cap  [\RH-(2n'P_{\uthte }-(b+y)),\RH-((2n'-2)P_{\uthte }+(b+x))]\times \mathbb{S}^1$
\item $C^+_x[p^+[e],e,n'']:=  M[e]\cap \{2n''P_{\uthte}+(b+x)\}\times \mathbb{S}^1$
\item $C^+_x[p^-[e],e,n'']:=  M[e]\cap \{\RH-(2n''P_{\uthte}+(b+x))\}\times \mathbb{S}^1$
\item $C^-_x[p^+[e],e,n']:=  M[e]\cap  \{(2n'P_{\uthte}-(b+x)\}\times \mathbb{S}^1$
\item $C^-_x[p^-[e],e,n']:=  M[e]\cap  \{\RH-((2n'P_{\uthte}-(b+x))\}\times \mathbb{S}^1$
\item $\widetilde C[p^+[e],e,n']:=  M[e]\cap  \{(2n'-1)P_{\uthte}\}\times \mathbb{S}^1$
\item $\widetilde C[p^-[e],e,n']:=  M[e]\cap  \{\RH-((2n'-1)P_{\uthte})\}\times \mathbb{S}^1$
\end{enumerate}
Here $b>a+5$ is chosen in \eqref{decaymain} and is independent of $\outau$ while $0<x,y<P_{\outau}-b$ where positivity of $P_{\outau}-b$ is guaranteed by the smallness of $\outau$.
We set the convention to drop the subscript $x$ when $x=0$; i.e. $S[p]=S_x[p]$.  Moreover, we denote $\Lambda_{x,x}=\Lambda_x$.
For the sake of clarity, we point out:
\begin{enumerate}
 \item $\widetilde S_x[p]=S_x[p]\bigcup_{e \in E_p} \Lambda_x[p,e,1]$ 
\item $\widetilde S_x[p,e,n]=\Lambda_x[p,e,n]\cup S_x[p,e,n] \cup \Lambda_x[p, e , n+1]$ for all $n \in \{1, \dots, l(e)-1\}$
\item $\widetilde S_x[p,e,l(e)] = \Lambda_x[p^+[e],e,l(e)]\cup S_x[p,e,l(e)]\cup \Lambda_x[p^-[e],e,l(e)]$
\item $C_x^+[p,e,0] = S_x[p] \cap \Lambda_x[p,e,1]$
\item $C_x^+[p,e,n'']= S_x[p,e,n''] \cap \Lambda_x[p,e,n''+1]$ for $n'' \in \{1, \dots, l(e)-1\}$
\item $C_x^-[p,e,n]= S_x[p,e,n] \cap \Lambda_x[p,e,n]$ except when $n=l(e)$ and $p=p^-[e]$ in which case\\
\indent \indent \indent $C_x^-[p^-[e],e,l(e)]= S_x[p^+[e],e,l(e)] \cap \Lambda_x[p^-[e],e,l(e)]$
\item $\partial S_x[p]=\cup_{e_j \in E_p} C_x^+[p,e_j,0]$
\item $\partial S_x[p,e,n]=C_x^+[p,e,n]\cup C_x^-[p,e,n]$ for $n \in \{1, \dots, l(e)-1\}$
\item $\partial S_x[p,e,l(e)]= C_x^-[p^+[e],e,l(e)] \cup C_x^-[p^-[e],e,l(e)]$
\item $\partial \Lambda_{x,y}[p,e,n']= C_x^+[p,e,n'-1]\cup C_y^-[p,e, n']$
\end{enumerate}
\end{definition}

\subsection*{The graph $\Gamma(d, \ell)$} In this subsection, we explain how to modify the initial graph $\Gamma$ by prescribed parameters $d,\boldsymbol \zeta$. The modification by $d$ is immediate while the modification by $\boldsymbol \zeta$ induces $\ell$. As $\Gamma$ is flexible, there exists a family $\mathcal F(\Gamma)$ such that for $(\hd,\ell)$ in an $\varepsilon$ neighborhood of the origin, $\Ghdl \in \mathcal F(\Gamma)$ varies smoothly in $\hd,\ell$.

Because we scaled the function $\hat \tau$ by $\outau$, we do the same for $\hd$ and thus we are interested in $d \in D(\Gamma)$ such that
\addtocounter{theorem}{1}
\begin{equation}\label{drestriction}
||d||_D \leq \varepsilon \outau.
\end{equation}Here $\varepsilon$ is taken from Definition \ref{flexdef}. For any such $\Gd$, let $\tau_d:E(\Gamma) \cup R(\Gamma) \to \Real$ denote the rescaled $\hat \tau_{\Ghd}$. 

We now determine the function $\ell:E(\Gamma) \to \Real$ that will rely -- for each $e$ -- on $l(e), \taue,$ and two vectors ${\boldsymbol \zeta}\ppe, \boldsymbol \zeta \pme \in \Real^3$.
The parameter ${\boldsymbol \zeta}$ describes a dislocation of each Delaunay piece from its central sphere and thus we describe the full parameter as a map from each attachment. 

\addtocounter{equation}{1}
\begin{definition}Let $\boldsymbol\zeta:A(\Gamma)\to \Real^3$ such that 
\addtocounter{theorem}{1}
\begin{equation}
\boldsymbol \zeta[p,e]=\sum_{i=1}^3 \zeta_i[p,e]\mathbf e_i.
\end{equation}
\label{zetadef} For the graph $\Gamma$, 
we define $\boldsymbol \zeta \in  \Real^{A}$, the finite dimensional vector space with values in $\Real$, indexed over $i=1,2,3$ and $\pe \in A(\Gamma)$.
Let
\[
 || \boldsymbol\zeta|| =  \max_{[p,e] \in A(\Gamma)} | \boldsymbol\zeta[p,e]|.
\] \end{definition}
As we will see, the norm of $\boldsymbol\zeta$ can be quite large compared to the norm of $d$. Throughout the paper, we allow 
\addtocounter{theorem}{1}
\begin{equation}\label{zetarestriction}
|| \boldsymbol\zeta||\leq \underline C   \outau
\end{equation}where $\underline C$ is a large, universal constant that is independent of $\outau$.

To each edge in the graph we associate a Delaunay building block with length determined by the parameter $\taue$ and the number of periods $l(e)$. 
Let $\widetilde l:E(\Gamma) \to \Real$ such that
\addtocounter{theorem}{1}
\begin{equation}
 \widetilde l(e):= \left(2+2p_{\taue}\right) l(e).
\end{equation}Thus, a Delaunay piece with $l(e)$ periods and parameter $\taue$ will have axial length equal to $\widetilde l(e)$.
\addtocounter{equation}{1}
\begin{definition}
For $e \in E(\Gamma)$ we define the function $\ell:E(\Gamma) \to \Real$ by
\addtocounter{theorem}{1}
\begin{equation}\label{ellprimedefinition}
2 \ell(e):=\left|{\boldsymbol \zeta}\ppe - \left( {\boldsymbol \zeta}\pme + (\widetilde l(e),0,0) \right) \right|-2l(e).
\end{equation}
\end{definition} 
\begin{figure}[h]\label{Dis}
\includegraphics[width=5in]{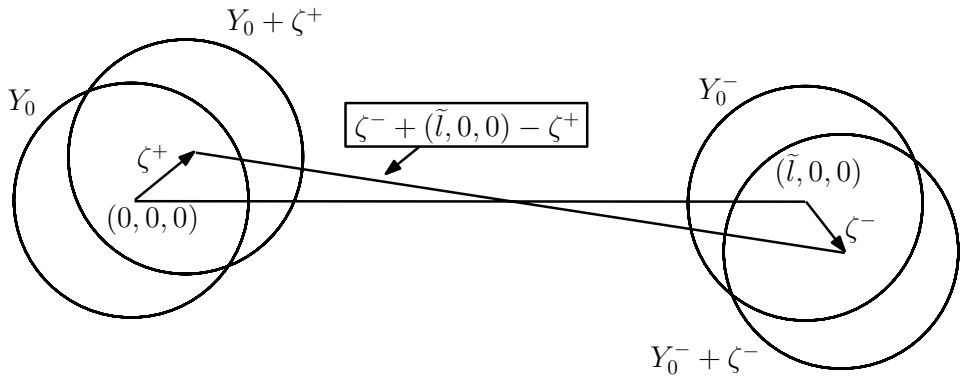}
            \caption{In the figure, we let $\zeta^+, \zeta^-$ correspond to $\boldsymbol \zeta[p^+[e],e], \boldsymbol \zeta[p^-[e],e]$ respectively. Also, notice that $Y_0^-$ is defined so that its center is at $(2+2p_{\taue}) l(e)$.}
      \end{figure}
The definition of $\ell(e)$ amounts to the following steps (see Figure 5). First, we position a segment of length $\widetilde l(e)$ so that it sits on
 the positive $x_1$-axis with one end fixed at the origin. Then we dislocate the two ends of this segment corresponding to ${\boldsymbol \zeta}\ppe$ and ${\boldsymbol \zeta}\pme$
where ${\boldsymbol \zeta}\ppe$ is the dislocation from the origin. We then measure the length of the segment connecting these two points. Finally,
we compare that length with the length of the edge $e$ in the graph $\Gamma$.
\addtocounter{equation}{1}
\begin{lemma}\label{lrestriction}
For $\ell \in L(\Gamma)$ with $\ell(e)$ as defined in \eqref{ellprimedefinition} and $\outau$ sufficiently small,
\[
 ||\ell(e)||_L:=\frac{|\ell(e)|}{l(e)} \leq C \underline C \outau-C\taue \log \taue \leq -C\outau \log \outau.
\]

\end{lemma}
\begin{proof}
 Notice
\[
-C||\boldsymbol \zeta|| +{p_{\taue}} \leq \frac{\ell(e)}{l(e)} \leq C ||\boldsymbol \zeta|| +{p_{\taue}}. \]
Thus, applying \eqref{periodlimits} and the uniform bound on $||\boldsymbol \zeta||$ immediately implies the result.
\end{proof}
We now highlight a few flexibility and asymptotic results we will use throughout the proof
\addtocounter{equation}{1}
\begin{lemma}
For $\tau$ sufficiently small and $||d||_D \leq \varepsilon \tau, ||\boldsymbol \zeta || \leq \underline C \outau$ we have the following:
\begin{enumerate}\item If $e,e'$ are corresponding edges or rays for $\Gamma, \Gd$ then there exists some constant $C$ such that
\addtocounter{theorem}{1}
\begin{equation}\label{tauratio}
 \frac{\tau_d(e')}{\uthte}\in \left(1-C\outau, 1+C\outau\right), \text{ and }
\angle(e,e') \leq C\varepsilon \outau
\end{equation}
\item and
\addtocounter{theorem}{1}
\begin{equation}\label{diffeodifference}
\left| 1 - \frac{P_{\tau_d(e')}}{P_{\uthte}}\right| \leq \frac{C\outau}{-\log \outau}.
\end{equation}
\item For $\Gd, \Gudl$ and corresponding edges $e,e'$,
\addtocounter{theorem}{1}
\begin{equation}\label{ddifftau}
\angle(e,e') \leq -C \outau \log \outau,
\text{ and } ||d_{\Gd} -d_{\Gudl}||_D \leq -C\outau^2 \log \tau.
\end{equation}
\end{enumerate}
\end{lemma}
\begin{proof}
The proof follows from the flexibility of $\Gamma$, an appropriate modification of Lemma \ref{hunbalancing}, and Lemma \ref{lrestriction}. The estimate \eqref{diffeodifference} follows immediately from \eqref{diffeodifference1} and the uniform bounds on the ratio of $|\taue|$ and $\outau$.
\end{proof}

\addtocounter{equation}{1}
\begin{remark}
Notice that the finiteness of the graph $\Gamma$, the definition of $\uthte$, and the condition \eqref{tauratio} implies that for any $||d||_D \leq \varepsilon \outau$, there exists $C$ such that $1/C \leq \frac{|\taue|}{\outau} \leq C$. This gives us the freedom to replace any bounds in $|\taue|^{\pm 1}$ by $(C\outau)^{\pm 1}$, reducing notation and bookkeeping.
\end{remark}

\subsection*{The Smooth Initial Surface}
We describe an immersion of the initial surface $M$ which depends upon the initial graph $\Gamma$ and the parameters $d, \boldsymbol \zeta$. The immersion amounts to an appropriate positioning of the building blocks described in Section \ref{BuildingBlocks}. As is evident from the notation, the building blocks themselves depend upon the parameters and the graph.

The positioning of the building blocks will depend in part upon a rotation associated to each edge that takes an orthonormal frame of the edge connecting $\widetilde l(e) + \boldsymbol \zeta[p^-[e],e]$ and $\boldsymbol \zeta[p^+[e],e]$ to the orthonormal frame of the edge $e'\in E(\Gudl)$ corresponding to $e\in E(\Gamma)$. 
\addtocounter{equation}{1}
\begin{prop}\label{zetaframe}
For ${\boldsymbol \zeta}$ from Definition \ref{zetadef} and each $e \in E(\Gamma)$ there exists a unique orthonormal frame $F_{\boldsymbol \zeta}[e]=\{\Be_1[e], \Be_2[e], \Be_3[e]\}$,
 depending smoothly on ${\boldsymbol \zeta}$,
 such that
\begin{enumerate}
\item  $\Be_1[e]$ is the unit vector parallel to  $\left({\boldsymbol \zeta}\pme + (\widetilde l(e),0,0) \right)-{\boldsymbol \zeta}\ppe$ such that $\Be_1[e] \cdot \mathbf e_1>0$.
\item For $i=2,3$, $\Be_i[e]=\RRR[\Be_1, \Be_1[e]](\Be_i)$.
\item For any $\Bv \in \Real^3$, $|\Bv - \RRR[\Be_1,\Be_1[e]](\Bv)|\leq C||\boldsymbol \zeta||\cdot |\Bv|$.
\end{enumerate}
\end{prop}
\begin{proof}
The first two items can be done by definition. To see the last item, observe that 
\[\left|\Be_1 \times \Be_1[e]\right|= \left|\sin \angle(\Be_1 , \Be_1[e])\right|:=\left| \sin \theta[e]\right| \leq C||\boldsymbol \zeta||.
\]As $|\Bv-\RRR[\Be_1, \Be_1[e]](\Bv)| \leq \sqrt 2 |\Bv|\sqrt{1- \cos \theta[e]}$, the result follows.
\end{proof}
\addtocounter{equation}{1}
\begin{definition}
 For $e \in R(\Gamma)$ we simply let $\Be_i[e]=\Be_i$.
\end{definition}

With this frame, we can describe the rigid motion that will position each building block associated to an edge or ray in $\Real^3$.
\addtocounter{equation}{1}
\begin{definition}
For each $e \in E(\Gamma)\cup R(\Gamma)$ with $e'$ denote the corresponding edge or ray on the graph $\Gudl$, 
let $\RRR'[e]$ denote the rotation in $\Real^3$ such that $\RRR'[e] (\Be_i[e])=\Bvp_i[e;d,\ell]$ for $i=1,2,3$.
Let $\TTT[e]$ denote the translation in $\Real^3$ such that $\TTT[e](\RRR'[e]({\boldsymbol \zeta}\ppe))=p^+[e']$. 
Let $\UUU[e]= \TTT[e] \circ \RRR'[e]$.
\end{definition}
Then for all $c_i \in \Real$,
\addtocounter{theorem}{1}
\begin{equation}
\UUU[e]\left({\boldsymbol \zeta}\ppe + c_i \Be_i[e]\right)= p^+[e'] + c_i \Bvp_i[e;d,\ell].
\end{equation}
The building blocks positioned at vertices of the graph $\Gudl$ correspond to spherical pieces with geodesic disks removed. The rigid motion required for positioning these in $\Real^3$ is simply a translation, but the building blocks themselves are determined by a diffeomorphism that depends upon the initial frame $F_\Gamma[e]$ and the new frame $F_{\boldsymbol \zeta}[e]$ for each $e \in E_p$.

For each $p \in V(\Gamma)$, let $\{e_1, \dots, e_{|E_p|}\}$ be an ordering of the edges and rays in the set $E_p$. Let 
\addtocounter{theorem}{1}
\begin{equation}\label{vertexsets}
 V_p^1 \times V_p^2=\{(\Bv_{e_1,p}, \Bv_2[e_1]), \dots, (\Bv_{e_{|E_p|},p}, \Bv_2[e_{|E_p|}])\}.
\end{equation}This set of ordered pairs represents the direction the edge $e$ emanates from $p$ and the second element in the frame $F_\Gamma[e]$. Notice that if $p=p^+[e]$ then the ordered pair is actually the first two elements of the frame $F_\Gamma[e]$ and if $p = p^-[e]$ then the first element in each pair corresponds to $-\Bv_1[e]$ while the second element is unchanged.
Recalling Definition \ref{pmvertexdef}, let
\[V_p^{{\boldsymbol \zeta},1} =\{ \sigma_{e_1,p}\RRR'[e_1](\Be_1), \dots ,\sigma_{e_{|E_p|},p}\RRR'[e_{|E_p|}](\Be_1) \}; \qquad V_p^{{\boldsymbol \zeta},2} =\{ \RRR'[e_1](\Be_2), \dots, \RRR'[e_{|E_p|}](\Be_2)\}.
\]These sets represent the first two elements of the frame that will position the building block corresponding to an edge or ray emanating from $p$. One should note that in general $\RRR'[e]\Be_i \neq \Bv_i[e;d, \ell]$. Figure 6 may help to highlight this fact. Thus, the geodesic disks removed under the diffeomorphism do not correspond to the frame $F_{\Gudl}[e]$ but to another frame that is dislocated from this one depending on the vectors $\boldsymbol \zeta\ppe, \boldsymbol \zeta \pme$. With these tools in hand, we define the initial immersion.

\begin{figure}[h]\label{EmbeddingPic}
\includegraphics[width=5in]{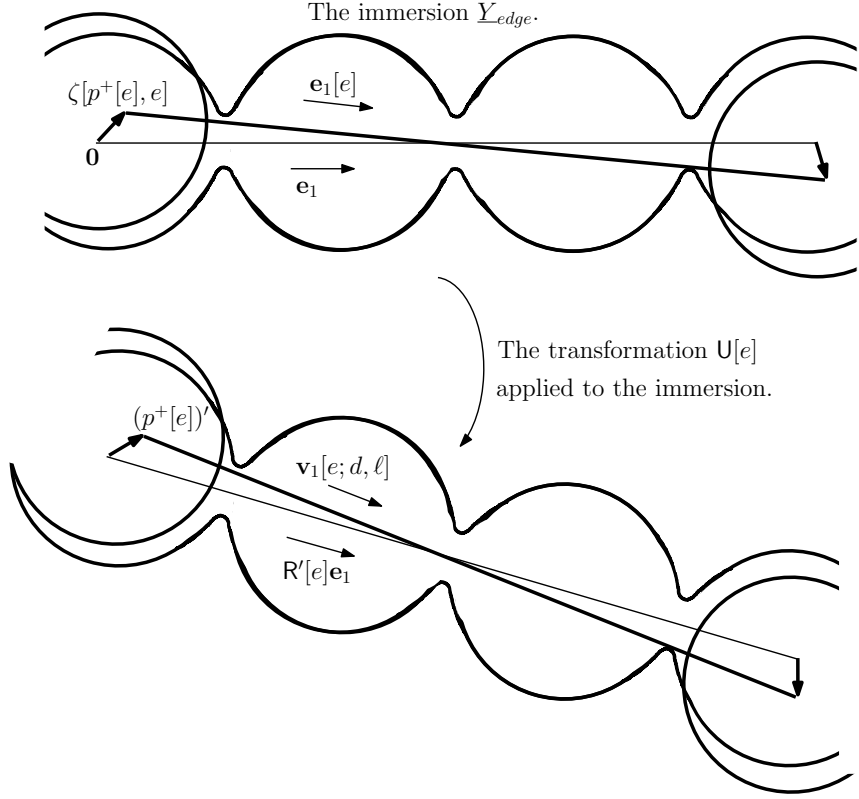}
            \caption{A rough idea of the immersion of one edge. Notice that the transformation $\UUU[e]$ sends the dislocated spheres to the vertices of the graph. The bold segment in the bottom picture corresponds to the positioning of the edge on the graph $\Gudl$. The Delauney piece has axis parallel to $\RRR'[e]\Be_1$, not parallel to the corresponding edge on the graph $\Gudl$.}
                  \end{figure}
                  
\addtocounter{equation}{1}
\begin{definition}\label{MainImmersionDef}
 Let $\hYtdz:M \to \Real^3$ be the smooth immersion defined by
\addtocounter{theorem}{1}
\begin{equation}
 \hYtdz(x)=\left\{ \begin{array}{ll}
                     p'+     \underline Y[V_p^1 \times V_p^2,V_{p}^{{\boldsymbol \zeta},1}\times V_p^{{\boldsymbol \zeta},2}](x)& \text{for } x \in M[p]\\
\UUU[e] \circ \underline  Y_{edge}[ \uthte,\taue, l(e), {\boldsymbol \zeta}\ppe, {\boldsymbol \zeta}\pme](x)& \text{for } x \in M[e], e \in E(\Gamma)\\
\UUU[e] \circ \underline Y_{ray}[\uthte, \taue, {\boldsymbol \zeta}\pe](x)& \text{for } x \in M[e], e \in R(\Gamma)
                           \end{array}\right.
\end{equation}
where here $p'\in V(\Gudl)$ is the vertex corresponding to $p \in V(\Gamma)$.

We denote
\addtocounter{theorem}{1}
\begin{equation}
 \Mtdz =\hYtdz(M)
\end{equation}and let $\Htdz:M \to \Real$ denote the mean curvature in the induced metric.
\end{definition}
\addtocounter{equation}{1}
\begin{remark}
When $\Gamma$ is pre-embedded, the embeddedness of the image surface $\Mtdz$ follows immediately from the smallness of $\outau$, the embeddedness of each of the building blocks and the smallness of $d,\boldsymbol \zeta$.
\end{remark}
On each $M[p]$, $\hYtdz$ embeds a sphere with geodesic disks removed where the centers of these disks match with the directions $\sep \RRR'[e]\Be_1$ for $e \in E_p$. The twisting diffeomorphism that describes the building blocks $M[p]$ must be included so the immersion is well defined on $M$ -- in particular on the regions $M[p] \cap M[e]$.

The immersion of each $M[e]$ amounts to positioning a building block of a Delaunay piece from Section \ref{BuildingBlocks}, where the parameters of the building block and its positioning depend on $\boldsymbol \zeta$ and $\Gudl$. The rigid motion of this piece positions the Delaunay portion of the immersion with axis parallel to $\RRR'[e]\Be_1$. Notice the Delaunay piece may not have axis parallel to the associated edge of the graph $\Gudl$. In fact, the piece will only have axis parallel to the associated edge of $\Gudl$ if $\boldsymbol \zeta\ppe =\boldsymbol \zeta\pme$, which implies $\Be_1[e]=\Be_1$.

\addtocounter{equation}{1}
\begin{definition}For $e \in E(\Gamma)$ define $H_{dislocation}[e],H_{gluing}[e]:M \to \Real$ in the following manner:
\begin{equation*}
H_{dislocation}[e](x):=\left\{ \begin{array}{ll} \Htdz - 1&\text{if } x \in M[e] \cap \left(\left([a,a+2] \cup [\RH-(a+2), \RH-a] \right) \times \Ss^1\right)\\
  0&\text{otherwise},\end{array}\right.  
\end{equation*}
\begin{equation*}
H_{gluing}[e](x):=\left\{ \begin{array}{ll} \Htdz - 1&\text{if } x \in M[e] \cap \left(\left([a+3,a+5] \cup [\RH-(a+5),\RH-(a+3)] \right) \times \Ss^1\right)\\
  0&\text{otherwise}.\end{array}\right.  
\end{equation*}For $e \in R(\Gamma)$ define 
\begin{equation*}
H_{dislocation}[e](x):=\left\{ \begin{array}{ll} \Htdz - 1&\text{if } x \in M[e] \cap \left([a,a+2]   \times \Ss^1\right)\\
  0&\text{otherwise},\end{array}\right.  
\end{equation*}
\begin{equation*}
H_{gluing}[e](x):=\left\{ \begin{array}{ll} \Htdz - 1&\text{if } x \in M[e] \cap \left([a+3,a+5]  \times \Ss^1\right)\\
  0&\text{otherwise}.\end{array}\right.  
\end{equation*}
\end{definition}
As an immediate consequence of the immersion and the definitions we have the following corollary.
\addtocounter{equation}{1}
\begin{corollary}All of the functions described above are smooth. Moreover we can define $H_{dislocation}:= \sum_{e \in E(\Gamma)\cup R(\Gamma)} H_{dislocation}[e] $ and $H_{gluing}$ likewise. Then
 the smooth function $H_{error}:=\Htdz -1:M \to \Real$ can be decomposed as
\[
 H_{error} = H_{dislocation} + H_{gluing}.
\]
\end{corollary}
\subsection*{Geometry}
We introduce maps $\widetilde Y[p]$ and $\widetilde Y\pen$ on $\widetilde S[p], \widetilde S\pen$, respectively, which one can interpret as the geometric limit of the normal of $\hYtdz$ as
$\outau \to 0$. A comparison between these maps and $\hYtdz$ will accomplish two things. First, in the next subsection, we use these limits to compare two conformal metrics $h, \chi$ to the pull back by the Gauss map of the induced metric $g$. Second, using techniques from Appendix B of \cite{KapAnn}, these maps help us understand the kernel of the linearized operator on $\hYtdz$ by considering the linearized operator in the induced metric of $\widetilde Y$ (which corresponds to the metric on the round sphere).
\addtocounter{equation}{1}
\begin{definition}
 Let $\widetilde Y[p]:\widetilde S[p] \to \mathbb S^2$ such that for $x=(t,\theta)$,
\addtocounter{theorem}{1}
\begin{equation}\label{tildeYp}
\widetilde Y[p](x)=\left\{ \begin{array}{ll}
  \hat Y[V_p,V_p](x) &\text{if }x \in M[p],\\
\RRR[e]\circ Y_0(x)& \text{if } p=p^+[e], x \in M[e]\cap\left([a,2P_{\uthte}-b]\times \Ss^1\right), \\
\RRR[e]\circ Y_0 (t-\RH,\theta )&\text{if } p=p^-[e], \\
&\:\:x \in M[e]\cap\left( [\RH-(2P_{\uthte}-b),\RH-a]\times \Ss^1\right).
 \end{array}\right.
\end{equation}
Let $\widetilde Y[p,e,n]: \widetilde S[p,e,n] \to \mathbb S^2$ such that 
\addtocounter{theorem}{1}
\begin{equation}\label{tildeYpen}
\widetilde Y[p,e,n](t,\theta) = \left\{\begin{array}{ll}Y_0 
( t-2P_{\uthte}n,\theta),& \text{if } n \text{ is even},\\
Y_0 (2P_{\uthte}n-t,\theta), & \text{if } n \text{ is odd}.
\end{array} \right.
\end{equation}
\end{definition}
The definition of $\widetilde Y\pen$ when $n$ is odd has the effect of reflecting the map when $n$ is even over the $\{x_1=0\}$-plane. As we wish to compare the normals of $\widetilde Y\pen$ and a rotated $\hYtdz$, this change is necessary as the normal for $\hYtdz$ when $n$ is odd is a reflection of the normal when $n$ is even.
\addtocounter{equation}{1}
\begin{prop}\label{geolimit}
 Let $p \in V(\Gamma)$ and let $p'$ be the corresponding vertex in the graph $\Gudl$. Then
 \[
 ||(\hYtdz-p') - \widetilde Y[p]:C^k(S_x[p], (\widetilde Y[p])^* g_{\Ss^2})||\leq C(k,x) \outau^{1/2}.
\]
For $e \in E(\Gamma)\cup R(\Gamma)$ 
\[
 || \UUU[e]^{-1}\circ\Ntdz(x) - \widetilde Y\pen(x):C^k(S_x[p,e,n], (\widetilde Y\pen)^*g_{\Ss^2})||\leq C(k,x)|\taue|\leq C(k,x)\outau.
\]
\end{prop}
\begin{proof}Recall by definition
\begin{align*}
\underline Y_{edge}&[\uthte, \taue,l,{\boldsymbol \zeta}\ppe,{\boldsymbol \zeta}\pme](t,\theta) = \\&Y_{edge}[\taue, l, {\boldsymbol \zeta}\ppe, {\boldsymbol \zeta}\pme] \circ \hat Y_{edge}[\uthte, \taue](t,\theta).
\end{align*}
By \eqref{diffeodifference}, the reparameterization function will only impact derivatives on this function by multiplicative factors of $\frac{\partial \underline t}{\partial t}$ which are bounded by $1-\frac{C\outau}{\log \outau}$. Therefore, we consider $Y_{edge}, Y_{ray}$ in place of $\underline Y_{edge},\underline Y_{ray}$ throughout this proof as the bounds on the estimates we determine can absorb the multiplicative factors.

We consider first the behavior of $\widetilde Y[p]$ on the region $M[e]\cap S_x[p]$ for $e \in E_p$. For $e \in E(\Gamma)$, the $C^k$ bounds are controlled by bounding
\[
\left( \RRR'[e]\circ Y_{edge}[\taue, l, {\boldsymbol \zeta}\ppe, {\boldsymbol \zeta}\pme] - \RRR'[e]\circ Y_0\right) + \left(\RRR'[e]\circ Y_0 - \RRR[\Be_1,\Bv_{e,p}]\circ Y_0\right).
\]Recalling Proposition \ref{geopropcentral}, the $C^k$ bounds of the first parenthesis are controlled by $C(k,x)(||\boldsymbol \zeta||+|\taue|)$. For the second parenthesis, observe 
\[
|(\RRR'[e] - \RRR[\Be_1,\Bv_{e,p}])\Bv| \leq C(||d||_D + ||\ell||_L + ||\boldsymbol \zeta||)|\Bv|.
\]As \eqref{drestriction}, \eqref{zetarestriction}, and Lemma \ref{lrestriction} imply
\[
||d||_D + ||\ell||_L + ||\boldsymbol \zeta||\leq (\underline C + \varepsilon)\outau -C\tau \log \tau \leq  \outau^{1/2},
\]the result follows. (Similar arguments hold for $Y_{ray}$.)
The diffeomorphism on $M[p]$ that describes the immersion depends smoothly on $||\boldsymbol \zeta||$ and thus the $C^k$ bounds on $M[p]$ follow immediately.

For the second norm, first notice that for a general Delaunay immersion $Y_\tau$, the evenness of $w'$ about $P_\tau$ and \eqref{normalvector} imply $\nu_\tau(t,\theta) = \nu_\tau(2P_\tau-t,\theta)$. 
Thus, it is enough to prove the inequality for any $n$ even as the description of the function $\widetilde Y\pen$ for $n$ odd accounts for the reflection. 
Recall that on $S_x\pen$, for $e \in E(\Gamma)$,
\[\UUU[e]^{-1}\circ\hYtdz(x)=\underline Y_{edge}[\uthte, \taue,l,{\boldsymbol \zeta}\ppe,{\boldsymbol \zeta}\pme](t,\theta).
\]Using \eqref{gmetricrelation} to compare the immersions $Y_{\taue}$ and $Y_0$ on $S_x\pen$ for $n$ even implies the result. 
\end{proof}

\subsection*{Conformal Metrics} Following \cite{KapAnn,HaskKap,KapYang,KapSurvey,KapClay,KapWente}, we solve the linearized problem and prove the main theorem by appealing to conformally changed metrics. The metric $h$ will be useful for understanding the approximate kernel as every standard region will behave like an $\Ss^2$ in this metric. We prove the main theorem in the $\chi$ metric, a conformal dilation of the induced metric, which provides a natural scale for the immersion.

We first define a cutoff function that will be helpful for a global description of the metric $\chi$.
\addtocounter{equation}{1}
\begin{definition}
Let $\hat \psi:M \to [0,1]$ be a globally defined smooth function on $M$ such that $\hat \psi =1$ for $x \in S[p] \cap M[p]$, $\hat \psi:= \psi[a+5,a+4](t)$ for $(t,\theta) \in S[p] \cap M[e]$ and $\hat \psi$ is identically zero elsewhere. 
\end{definition}
Thus, $\hat \psi(x)$ is identically 1 on each central sphere $S[p]$ except in a neighborhood of its boundary, where $\hat \psi$ smoothly transits to zero. Notice $\hat \psi, 1-\hat \psi$ form a subordinate partition of unity on $M$.

\addtocounter{equation}{1}
\begin{definition}\label{rhodef}We define the function $\widetilde \rho:\cup_{e \in E(\Gamma) \cup R(\Gamma)}M[e] \to \Real$
such that for each $ (t,\theta) \in M[e]$,
\begin{equation*}\label{rhoeq}
\widetilde \rho(t, \theta) =
{r_{\taue}(\underline t, \theta)}^{-1} .
\end{equation*} Here we define
\addtocounter{theorem}{1}
\begin{equation}
\label{underlineudefn}\underline t =\frac{P_{\taue}}{P_{\uthte}}t.
\end{equation}

Let $\rho:M \to \Real$ be the globally defined smooth function such that
\addtocounter{theorem}{1}
\begin{equation}
\rho(x) = \hat \psi(x) \frac{|A(x)|}{\sqrt{2}} + (1- \hat \psi(x)) \widetilde \rho(x).
\end{equation}
\end{definition}
\addtocounter{equation}{1}
\begin{definition}
Let $g = (\hYtdz)^*g_{\Real^3}$.  On $M$ we define two global metrics $h,\chi$ such that
\addtocounter{theorem}{1}
\begin{equation}\label{confmets}
h = \frac{|A|^2}{2}g; \:\:\: \chi = \rho^2 g.
\end{equation}
\end{definition}

We now list a number of metric relations that will be useful throughout the paper.
\addtocounter{equation}{1}
\begin{lemma}\label{normcomparisons}\label{metriccomparisons}
Assuming $\outau$ is small enough, the following hold:
\begin{enumerate}
\item \label{nc6}$||h: C^k(S_x[p],g)|| \leq C(k,x)$
\item \label{nc7}$||h-(\widetilde Y[p])^* g_{\Ss^2}):C^k(S_x[p], (\widetilde Y[p])^* g_{\Ss^2})||\leq C(k,x) \outau^{1/2}$
\item \label{nc4}$||\chi:C^k(S_x[p],h)||\leq C(k,x)$
\item \label{nc8}$||h-(\Ntdz)^*g:C^k(S_x\pen,h)||\leq C(k,x)\outau$
\item \label{nc10}$||h - (\widetilde Y\pen)^*g_{\Ss^2}:C^k(S_x\pen,(\widetilde Y\pen)^*g_{\Ss^2})||\leq C(k,x)\outau$
\item \label{nc3}$||\chi:C^k(S_x\pen, h)|| \leq C(k,x)$
\item \label{nc2}$||\rho^{\pm 1}:C^k(S_x[p],\chi)||\leq C(k,x)$
\item \label{nc1}$||\rho^{\pm 1}:C^k(M, \chi, \rho^{\pm 1})|| \leq C(k)$
\end{enumerate}
\end{lemma}
\begin{proof}The definition of $h$ and the uniform geometry of the region $S_x[p]$ for $\outau$ small implies \eqref{nc6}; \eqref{nc7} follows immediately from Proposition \ref{geolimit} and the triangle inequality.
Lemma \ref{radiuslemma}, Proposition \ref{geolimit}, and the definition of the metrics $h$ and $\chi$ immediately imply \eqref{nc4}.

Using \eqref{curvatureeq}, on each $S_x\pen$,
\addtocounter{theorem}{1}
\begin{equation}\label{metricrelation}
 h = \left(r^2 + \frac{\taue^2}{r^2}\right) \chi.
\end{equation}
Further, $h = 2|\taue| \cosh(2w(\underline t))(d\underline t^2 + d\theta^2)$ and $\chi =d\underline t^2+ d\theta^2$. Thus, $(\Ntdz)^*g$, $h$, and $\chi$ all provide a natural isometry between positively curved standard regions and negatively curved standard regions on each $M[e]$. Therefore, it is enough to prove the next three items on a standard region with limit $\Ss^2$, that is, a standard region with Gauss curvature $K>0$. Item \eqref{nc8} is proven in \cite{KapAnn}, Appendix A. We adapt the statement there to fit with our notation. For \eqref{nc10} we appeal to Proposition \ref{geolimit} and \eqref{nc8}. Item \eqref{nc3} follows from Lemma \ref{radiuslemma} and \eqref{metricrelation}.

On $S_x[p]$, \eqref{nc2} follows from the definition of $\rho$, Proposition \ref{geolimit} and the uniform equivalence of $g,h,\chi$ on each $S_x[p]$.
For \eqref{nc1} we define a new function and consider its $C^k$ bounds.
For any $x=(s,\theta) \in M[e] \backslash S[p]$ let $\hat \rho(x):= \frac{\rho(x)}{\rho({\hat t},\theta)}$ for some fixed ${\hat t} \in M[e]\backslash S[p]$. By definition,
\[
\hat \rho(x) =  e^{-w(\underline s)+w(\underline {\hat t})}.
\]We are interested, of course, in $C^k$ bounds on $\hat \rho(x)^{\pm 1}$ for $x \in ({\hat t}-3/2,{\hat t}+3/2)\times \Ss^1:={\hat D}$. Recall the bounds on $\frac{P_{\taue}}{P_{\uthte}}$ imply that ${\hat D}$ contains a unit ball about the point $(\hat t, \theta)\in M[e]$ with respect to the $\chi$ metric. The definition of $w$ implies $|w'| \leq 1$ and thus
\[
|w(\underline s)-w(\underline {\hat t})| =\left|\int_{\underline {\hat t}}^{\underline s} w'(\underline t) d\underline t \right| \leq C|{\hat t}-s|.
\]Thus, $||\hat \rho^{\pm1}:C^0({\hat D}, \chi)|| \leq C$. Now
\[
\frac{\partial}{\partial \underline t} \hat \rho(x) = -w'(\underline s)\hat \rho(x); \qquad \frac{\partial^2}{\partial t^2} \hat \rho(x) = -w''(\underline s) \hat \rho(x) + (w'(\underline s))^2 \hat \rho(x).
\]The definition of $w$ implies $w''(\underline s) = -4|\taue| \cosh w(\underline s) \sinh w(\underline s)= |\taue| (e^{2w(\underline s)}-e^{-2w(\underline s)})$. Recalling that $C |\taue| \leq r_{\taue} \leq 1 + C|\taue|$ on each $M[e]$, $r_{\taue} = |\taue|^{1/2} e^w$, and $w$ is odd about $P_{\taue}$ we observe that $|w''| \leq C$ on each $M[e]$. The uniform bound on $|w'|$ along with the previous analysis implies $||\hat \rho^{\pm1}:C^2({\hat D}, \chi)|| \leq C$ for some constant $C$. For any $k \in \mathbb N$, we can write $\frac{\partial^k}{\partial \underline t^k} \hat \rho (x)$ as a function of $\frac{\partial^m}{\partial \underline t^m} w(x)$ and $\hat \rho(x)$ for $m=0, \dots, k-2$. This follows simply from repeated differentiation of the ODE of $w$ and the fact that $\hat \rho$ is an exponential in the variable $\underline t$. Thus, we obtain $C^k$ upper bounds for $\hat \rho$ on each $M[e]$ via the $C^{k-2}$ upper bounds. Moreover, for derivatives of $\hat \rho^{-1}$, we apply the same logic and note that derivatives will now include factors of $\hat \rho^{-k-1}$, which is bounded as a function of $k$. Putting all of this together with item \eqref{nc2} implies item \eqref{nc1}.
\end{proof}

Finally, as we choose to solve the global linear problem in the metric $\chi$, we state precisely the estimates we have for the error on the mean curvature with respect to the metric $\chi$.

\addtocounter{equation}{1}
\begin{prop}\label{Hestimates}
$||\rho^{-2}H_{gluing}:{C^k}(M,\chi)|| \leq C(k) \outau$.

 $||\rho^{-2}H_{dislocation}:{C^k}(M,\chi)|| \leq C(k)||{\boldsymbol \zeta}|| \leq C(k) \underline C  \outau$.

\end{prop}
\begin{proof}
The bounds follow immediately from Proposition \ref{Hbounds} and item \eqref{nc2} of Lemma \ref{metriccomparisons} using the uniform equivalence of the metrics $g, \chi$ on the support of the functions involved.
\end{proof}

\section{The linearized equation}
\label{LinearizedEq}
The goal of this section is to clearly state and prove Proposition \ref{LinearSectionProp}. Recalling \eqref{confmets}, we start by defining the operators 
\[
\mathcal L_g:= \Delta_g + |A|^2; \:\: \mathcal L_h = \frac{2}{|A|^2}\mathcal L_g; \: \: \mathcal L_\chi = \rho^{-2} \mathcal L_g.
\]We choose to study the inhomogenous linearized equation in one of the following forms:
\[
\mathcal L_\chi \phi = E; \: \: \mathcal L_h \phi = \cfE E.
\]

\subsection*{The linear equation on the transition regions}
We first consider the linearized equation on $\Lambda_{x,y}[p,e,n']$ -- the transition regions. 
Throughout this section, we refer to each of these regions simply by $\Lambda$ and their boundary components, $C^+_x[p,e,n'-1]$ and $C^-_y[p,e,n']$ respectively as  $C^+, C^-$.

On the transition regions, $\chi$ takes the form $d\underline t^2 + d\theta^2$. That is, $\chi$ is the flat metric on the reparameterized cylinder.
Let $\underline x, \overline x,\overline{ \underline { x}}$ denote the $t$-coordinate distance from $C^+, C^-$ and $\partial \Lambda=C^+ \cup C^-$. Then
\addtocounter{theorem}{1}
\begin{equation}
(2n'-2)P_{\uthte} + b + x + \underline{ x} = t, \ \ \ 2n'P_{\uthte} -b -y - \overline x = t, \ \ \ \overline {\underline {x}}:= \min\{\underline{ x}, \overline x\}.
\end{equation}
Let $\eL$ define the $t$-coordinate length of the cylinder parameterizing $\Lambda$.  Then
\addtocounter{theorem}{1}
\begin{equation}
\eL = 2P_{\uthte} -2b -x -y.
\end{equation}

We consider weighted Sobolev spaces in the $\chi$ metric where the weight function is with respect to $\underline x$. This poses no problem as the weight function is a decaying exponential, and \eqref{diffeodifference} implies that exponential decay with respect to $\underline t$ and $t$ differ only by multiplicative factors of the form $e^{\outau} = 1+ O(\outau)$.

For each $\epsilon_1>0$, there exists $b \gg 1$ so that for $\Lambda$ defined in terms of that $b$, 
\addtocounter{theorem}{1}
\begin{equation}\label{decaymain}
||\rho^{-2}|A|^2:C^k(\Lambda, \chi,e^{- 2\overline {\underline x}})|| \leq \epsilon_1. 
\end{equation}
This follows from the work done in Lemma \ref{Plemma}, once we observe 
$\rho^{-2}|A|^2 = \tau e^{2w}(2+2e^{-4w})=4\tau \cosh (2w)$. Recall $\underline t =  \frac{P_{\taue}}{P_{\uthte}}t$. Each $\Lambda$ with parameter $\tau_0$, equipped with the metric $\chi$, is isometric to $[b,2P_{\tau_0} - b] \times \mathbb{S}^1$ 
equipped with the metric $d\underline t^2 + d\theta^2$. Moreover Lemma \ref{Plemma} and the fact that $\outau/C \leq \tau_0 \leq \outau$ implies
 \addtocounter{theorem}{1}
\begin{equation}\label{elllength}
{|\ell_\Lambda}{+\log \outau|}\leq C
\end{equation}where $C$ depends on $\epsilon_1$ and on $\Gamma$.

We now choose $\epsilon_1>0$ appropriately small to satisfy the desired conditions going forward and $b$ large enough to satisfy \eqref{decaymain}. For this fixed $b\gg 1$,
\addtocounter{equation}{1}
\begin{prop} \label{lowestevtrans}
The lowest eigenvalue for the Dirichlet problem for $\mathcal{L}_\chi$ on $\Lambda$ is $> C\eL^{-2}$.
\end{prop}
\begin{proof}
The proof follows exactly as in \cite{KapWente}, in particular refer to the proof of Lemma 2.26. Using the Rayleigh quotient to determine the lowest eigenvalue gives the lower bound described.
\end{proof}
Immediately, we get the following corollary.
\addtocounter{equation}{1}
\begin{corollary}
\begin{enumerate}
\item The Dirichlet problem for $\mathcal{L}_\chi$ on $\Lambda$ for given $C^{2, \beta}$ Dirichlet data has a unique solution.
\item 
For $E \in C^{0,\beta}$ there is a unique $\varphi \in C^{2,\beta}$ such that $\mathcal{L}_\chi \varphi = E$ on $\Lambda$ and $\varphi=0$ on $\partial \Lambda$.  Moreover $||\varphi:C^{2,\beta}(\Lambda,\chi)|| \leq C(\beta) \eL^2 ||E:C^{0,\beta}(\Lambda, \chi) ||$.
\end{enumerate}
\end{corollary}

\addtocounter{equation}{1}
\begin{definition}
 Let $\mathcal H_k[C]$ denote the finite dimensional space of spherical harmonics on the meridian sphere $C$ 
that includes all of the functions in the first $k$-eigenspaces of $\Delta_{\Ss^1}$. 
\end{definition}
Observe $\mathcal H_0[C]$ is the space of constant functions and
$\mathcal H_1[C]$ is spanned by $\{1, \cos \theta, \sin \theta\}$. In the next proposition and corollary, we examine the Dirichlet problem where we are allowed to modify the low harmonics on one boundary circle. This induces fast exponential decay for the norm of the solution to the Dirichlet problem on the transition region.
\addtocounter{equation}{1}
\begin{prop}\label{RLambda}Given $\beta \in(0,1)$ and $\gamma \in (1,2)$ there exists $\epsilon_1>0$ such that the following holds.\\
There is a linear map $\mathcal{R}_\Lambda:C^{0,\beta}(\Lambda) \to C^{2, \beta}(\Lambda)$ such that the following holds for $E \in C^{0,\beta}(\Lambda)$ and $V:= \mathcal{R}_\Lambda E$:
\begin{enumerate}
\item $\mathcal{L}_\chi V =E$ on $\Lambda$.
\item $V|_{C^+} \in \mathcal H_1[C^+]$ and vanishes on $C^-$.
\item $||V:C^{2,\beta}(\Lambda, \chi, e^{-\gamma \underline { x}})|| \leq C||E:C^{0,\beta}(\Lambda, \chi, e^{-\gamma \underline{ x}})||$.
\item $\mathcal{R}_\Lambda$ depends continuously on all of the parameters of the construction.
\end{enumerate}
The proposition still holds if the roles of $C^+$ and $C^-$ are reversed and $\underline { x}$ is replaced by $\overline x$.  
\end{prop}
\begin{proof}
The first two items follow by standard linear theory. For the third and fourth, let $\widetilde{\mathcal R}_\Lambda$ denote the linear map applied to the operator $\Delta_\chi$ and suppose
$V_1$ is such that $\Delta_\chi V_1 =E$, i.e. $\widetilde{\mathcal R}_\Lambda E=V_1$. Then 
\[
 ||\mathcal L_\chi V_1-E: C^{0,\beta}(\Lambda, \chi, e^{-\gamma \underline x})|| \leq C(\beta, \gamma)\epsilon_1 ||E: C^{0,\beta}(\Lambda, \chi, e^{-\gamma \underline x})||.
\]For $\epsilon_1$ sufficiently small, $\mathcal L_\chi$ is a small perturbation of $\Delta_\chi$ and thus an iteration implies the result.
\end{proof}

\addtocounter{equation}{1}
\begin{corollary}\label{Rpartial}
Assuming $\epsilon_1$ is small enough in terms of $\epsilon_2$ and $\beta \in(0,1), \gamma \in (1,2)$, there is a linear map:
\[\mathcal{R}_\partial :\{f \in C^{2, \beta}(C^+): f \text{ is } L^2(C^+, d\theta^2) \text{-orthogonal to } \mathcal H_1[C^+]\} \to C^{2, \beta}(\Lambda)\]
such that the following hold for $f$ in the domain of $\mathcal{R}_\partial$ and $V:= \mathcal{R}_\partial f$:
\begin{enumerate}
\item $\mathcal{L}_\chi V=0$ on $\Lambda$.
\item $V|_{C^+}-f \in \mathcal H_1[C^+]$  and $V$ vanishes on $C^-$.
\item \label{rp3}$||V|_{C^+} - f:C^{2, \beta}(C^+, d\theta^2)|| \leq \epsilon_2||f:C^{2,\beta}(C^+, d\theta^2)||$.
\item \label{rp4}$||V:C^{2,\beta}(\Lambda, \chi, e^{-\gamma\underline{ x}})|| \leq C(\beta, \gamma)||f:C^{2, \beta}(C^+, d\theta^2)||$.
\item $\mathcal{R}_\partial$ depends continuously on all of the parameters of the construction.
\end{enumerate} The proposition still holds if the roles of $C^+$ and $C^-$ are exchanged and $\underline { x}$ is replaced by $\overline x$.
\end{corollary}
\begin{proof}
 Let 
\[
 \widetilde {\mathcal R}_\partial:\{ 
f \in C^{2, \beta}(C^+): f \text{ is } L^2(C^+, d\theta^2) \text{-orthogonal to } \mathcal H_1[C^+]\} \to C^{2, \beta}(\Lambda)
\]
define the linear map such that for $f$ in the domain and $\widetilde V = \widetilde{\mathcal R}_\partial f$ one has
\begin{enumerate}
 \item $\Delta_\chi \widetilde V=0$ on $\Lambda$.
\item $\widetilde V|_{C^+}=f$ and $\widetilde V$ vanishes on $C^-$.
\item $||\widetilde V: C^{2,\beta}(\Lambda, \chi, e^{-\gamma \underline x})|| \leq C(\beta, \gamma)||f:C^{2,\beta}(C^+, d\theta^2)||$.
\end{enumerate}
Such a map exists by standard linear theory. 
The corollary follows immediately by applying the previous proposition to the map
\[
 \mathcal R_\partial f:= \widetilde{\mathcal R}_\partial f - \mathcal R_\Lambda \mathcal L_\chi \widetilde{\mathcal R}_\partial f.
\]

\end{proof}

We consider the behavior of solutions to the Dirichlet problem with prescribed boundary data in $\mathcal H_1[C^+]\cup\mathcal H_1[C^-]$ for the operator $\mathcal L_\chi$ by comparing these solutions to known solutions for the operator $\Delta_\chi$. As we frequently reference and use these solutions, we provide a thorough definition and description. 
\addtocounter{equation}{1}
\begin{definition}
 For $i=1,2,3$, let $V_i[\Lambda, a_1,a_2]$, $\widetilde V_i[\Lambda, a_1,a_2]$ denote the solutions to the Dirichlet problems on $\Lambda$ given by 
\[
 \mathcal L_\chi V_i[\Lambda, a_1,a_2]=0, \qquad \Delta_\chi \widetilde V_i[\Lambda, a_1,a_2]=0,
\]
with boundary data
\begin{equation*}
\begin{array}{lr}
 V_1[\Lambda, a_1,a_2]=\widetilde V_1[\Lambda, a_1,a_2]=a_1 &\text{on } C^+,\\

 V_1[\Lambda, a_1,a_2]=\widetilde V_1[\Lambda, a_1,a_2]=a_2 &\text{on } C^-,\\

 V_2[\Lambda, a_1,a_2]=\widetilde V_2[\Lambda, a_1,a_2]=a_1\cos \theta &\text{on } C^+,\\

 V_2[\Lambda, a_1,a_2]=\widetilde V_2[\Lambda, a_1,a_2]=a_2 \cos \theta &\text{on } C^-,\\

 V_3[\Lambda, a_1,a_2]=\widetilde V_3[\Lambda, a_1,a_2]=a_1\sin \theta &\text{on } C^+,\\

 V_3[\Lambda, a_1,a_2]=\widetilde V_3[\Lambda, a_1,a_2]=a_2\sin \theta &\text{on } C^-.
\end{array}\end{equation*}
As  $V_i[\Lambda, a_1,a_2]$ and $\widetilde V_i[\Lambda, a_1,a_2]$ are linear on $a_1,a_2$ and the roles of $C^+,C^-$ can be exchanged, it is enough to understand
 $V_i[\Lambda, 1,0]$, $\widetilde V_i[\Lambda, 1,0]$. We observe
\addtocounter{theorem}{1}
\begin{equation}\label{Vlambda}
 \widetilde V_1[\Lambda, 1, 0] = \frac{\overline x}{\ell_\Lambda}, \: \widetilde V_2[\Lambda, 1,0] = \frac{\sinh (P_{\taue}\overline x/P_{\uthte})}{\sinh  (P_{\taue}\ell_\Lambda/P_{\uthte})}\cos \theta,\: \widetilde V_3[\Lambda,1,0] = \frac{\sinh  (P_{\taue}\overline x/P_{\uthte})}{\sinh  (P_{\taue}\ell_\Lambda/P_{\uthte})}\sin \theta.
\end{equation}

\end{definition}
\addtocounter{equation}{1}
\begin{prop}\label{LambdaKernel}
 $V_1[\Lambda, a_1,a_2]$ is constant on each meridian and $V_2[\Lambda, a_1,a_2], V_3[\Lambda, a_1,a_2]$ are constant multiples of $\cos \theta, \sin \theta$ respectively on each meridian.
Moreover, for $\epsilon_1$ sufficiently small in terms of $\epsilon_3>0$, $\beta\in(0,1), \gamma \in (1,2)$, there are constants $A_1^+,A_1^-, A_2, A_3$ such that
\begin{enumerate}
 \item $|A_1^+-1|\leq \epsilon_3$, $A_1^-\leq \epsilon_3/\ell$.
\item $||V_1[\Lambda,1,0]-\widetilde V_1[\Lambda, A_1^+, A_1^-]:C^{2,\beta}(\Lambda, \chi, e^{-\gamma\underline x} + \ell^{-1} e^{-\gamma \overline x})||\leq \epsilon_3$.
\item $|A_2-1|\leq \epsilon_3$, $|A_3-1|\leq \epsilon_3$.
\item $||V_2[\Lambda,1,0]-\widetilde V_2[\Lambda, A_2, 0]:C^{2,\beta}(\Lambda, \chi, e^{-(\gamma+1)\underline x} + e^{-\ell})||\leq \epsilon_3$.
\item $||V_3[\Lambda,1,0]-\widetilde V_3[\Lambda, A_3, 0]:C^{2,\beta}(\Lambda, \chi, e^{-(\gamma+1)\underline x} + e^{-\ell})||\leq \epsilon_3$.
\end{enumerate}
\end{prop}
\begin{proof}The rotational invariance of the solutions follows from the rotational invariance of the operator $\mathcal L_\chi$.

 By \eqref{decaymain}, and the definition of $\widetilde V_1$ the decay takes the form of:
\[
 ||\mathcal L_\chi \widetilde V_1[\Lambda,1,0]: C^{0,\beta}(\Lambda, \chi, e^{-\gamma\underline x} + \ell^{-1} e^{-\gamma \overline x})|| \leq \epsilon_1 C(\beta,\gamma)
\]and
\[
 ||\mathcal L_\chi \widetilde V_1[\Lambda,0,1]: C^{0,\beta}(\Lambda, \chi, e^{-\gamma\overline x} + \ell^{-1}e^{-\gamma \underline x})|| \leq \epsilon_1 C(\beta,\gamma).
\]
 We now apply Proposition \ref{RLambda}, but with the weaker estimates available here. Upon application, we determine the existence of $\hat V_1[\Lambda, 1,0], \hat V_1[\Lambda,0,1]\in C^{2,\beta}(\Lambda)$ such that
\[
 \mathcal L_\chi \hat V_1[\Lambda,1,0] = -\mathcal L_\chi \widetilde V_1[\Lambda,1,0]; \qquad \mathcal L_\chi \hat V_1[\Lambda,0,1]=-\mathcal L_\chi \widetilde V_1[\Lambda,0,1],
\]both functions are constant on $C^+,C^-$ by rotational invariance, and finally
\[
 || \hat V_1[\Lambda,1,0]: C^{2,\beta}(\Lambda, \chi, e^{-\gamma\underline x} + \ell^{-1} e^{-\gamma\overline x})|| \leq \epsilon_1 C(\beta,\gamma)
\]and
\[
 || \hat V_1[\Lambda,0,1]: C^{2,\beta}(\Lambda, \chi, e^{-\gamma\overline x} + \ell^{-1}e^{-\gamma\underline x})|| \leq \epsilon_1 C(\beta,\gamma).
\]
We choose $A_1^+,A_1^-$ so that
\[
 V_1[\Lambda,1,0] = \widetilde V_1[\Lambda, A_1^+,A_1^-] + A_1^+\hat V_1[\Lambda,1,0] + A_1^- \hat V_1[\Lambda,0,1]
\]is achieved on $\partial \Lambda = C^+ \cup C^-$. Notice both sides of the equation are in the kernel of $\mathcal L_\chi$ on $\Lambda$. The uniqueness of the
Dirichlet solution on $\Lambda$ implies equality holds on all of $\Lambda$. 
The estimates on $A_1^+, A_1^-$ come simply from applying the given estimates on the appropriate meridians. 

For $\widetilde V_i[\Lambda,1,0]$ with $i=2,3$, we determine
\[
 ||\mathcal L_\chi \widetilde V_i[\Lambda,1,0]: C^{0,\beta}(\Lambda, \chi, e^{-(\gamma+1)\underline x} + e^{-\ell})|| \leq \epsilon_1 C(\beta,\gamma).
\]
The rest of the proof follows the strategy outlined previously.

\end{proof}

\addtocounter{equation}{1}
\begin{corollary}\label{LambdaCloseCor}
If $f \in C^{2,\beta}(\Lambda)$ satisfies $\mathcal{L}_\chi f = 0$ on $\Lambda$ and $f|_{C^-}=0$, then 
\begin{enumerate}
\item \label{LCC1}$||f:C^{2,\beta}(\Lambda, \chi, (\overline x+1)/\ell)|| \leq C(b,\beta)||f:C^{2, \beta}(C^+, d\theta^2)||.$
\item \label{LCC2}$||\partial_t f:C^0(\Lambda \cap \widetilde C, \chi)||\leq C(b,\beta)||f:C^{2, \beta}(C^+, d\theta^2)||/|\log \outau|$.
\end{enumerate}
\end{corollary}

\begin{proof}
On $C^+$, decompose $f= f_1 + f_2 + f_3 + f^\perp$ where $f_1=a_1, f_2 = a_2\cos \theta, f_3 = a_3 \sin \theta$ and $f^\perp$ is in the domain of $\mathcal R_\partial$ defined in Corollary \ref{Rpartial}. Let $V = \mathcal R_\partial f^\perp$ given by Corollary \ref{Rpartial}. Then using the estimates from Corollary \ref{Rpartial} and Proposition \ref{LambdaKernel} implies \eqref{LCC1}. 
Using this same decomposition and analysis, \eqref{LCC2} follows by recalling \eqref{elllength}.
\end{proof}

\subsection*{The approximate kernel on the standard regions}We now consider the nature of the approximate kernel for the operator $\mathcal L_h$ on the standard regions. By approximate kernel, we mean the span of eigenfunctions of $\mathcal L_h$ that have eigenvalues close to $0$. Following the methods of Appendix B in \cite{KapAnn}, we understand the approximate kernel by comparing each standard region with the $h$ metric to a round sphere with the standard metric. Because we do not impose symmetry, the span of the approximate kernel is three dimensional on each standard region. We determine a basis for the approximate kernel on each standard region using a function defined on that domain of $M$. The first functions defined below are normalizations of the components of the unit normal to the immersion $\hYtdz$. The second are also based on the unit normal to the immersion, but are defined relative to a Delaunay building block with axis positioned on the $x_1$-axis.

\addtocounter{equation}{1}
\begin{definition}
Let $\hat f_i:M \to \Real$ for $i=1,2,3$ such that
\addtocounter{theorem}{1}
\begin{equation}\label{kernelcomparison}
\hat f_i :=\frac{\mathbf e_i \cdot \Ntdz}{||\mathbf e_i \cdot N_{\Ss^2} :L^2(\Ss^2)||}= \frac{1}{\pi}\mathbf e_i \cdot \Ntdz
\end{equation} where $\Ntdz$ is the normal to the immersion $\hYtdz$.

Let $\cf_i: \cup_{e \in E(\Gamma) \cup R(\Gamma)}M[e] \to \Real$ for $i=1,2,3$ such that for $x \in M[e]$
\addtocounter{theorem}{1}
\begin{equation}\label{cfeq}
\cf_i(x):= \frac{\Be_i \cdot \left(\RRR'[e]^{-1}\Ntdz(x)\right)}{||\mathbf e_i \cdot N_{\Ss^2} :L^2(\Ss^2)||}= \frac 1\pi \Be_i \cdot \left(\RRR'[e]^{-1}\Ntdz(x)\right).
\end{equation}
\end{definition}
Recall that the functions 
\[\frac{N_{\Ss^2} \cdot \Be_i}{||\mathbf e_i \cdot N_{\Ss^2} :L^2(\Ss^2)||}\]
form an orthonormal basis for the kernel of the operator $\Delta_{\Ss^2} + 2$. This observation and an understanding of the geometric limiting objects (in the $h$ metric) of the standard regions as $\outau \to 0$ motivates the definition of the functions. The incorporation of a rotation in the definition of $\cf_i$ makes the choice of constants in Definition \ref{tildefdef}, \eqref{phipen} and the estimates they satisfy more immediately obvious.
\addtocounter{equation}{1}
\begin{prop}\label{approxkerprop}
For fixed $b$ large, let $\epsilon>0$. There exists $\tau_\epsilon$ such that for $0<\outau<\tau_\epsilon$, 
\begin{enumerate}
\item $\mathcal{L}_h$ acting on $\widetilde S\pen$, with vanishing Dirichlet conditions, has exactly three eigenvalues in $[-\epsilon, \epsilon]$ and no other eigenvalues in $[-1/2, 1/2]$.  There exist $\{ f_{1}[p,e,n], f_{2}[p,e,n], f_{3}[p,e,n] \}$ that denote an orthonormal basis 
of the \emph{approximate kernel} where $f_{i}[p,e,n] \in C^\infty(\widetilde{S}[p,e,n])$. Moreover $f_{i}\pen$ depends continuously on the parameters of the construction and each $f_i\pen$ satisfies
\addtocounter{theorem}{1}
\begin{equation}
\label{approxkerest}
||f_{i}[p,e,n] - \cf_i: C^{2,\beta}(S_5[p,e,n],h) || < \epsilon.
\end{equation}
\item $\mathcal L_h$ acting on $\widetilde S[p]$  has exactly three eigenvalues in $[-\epsilon, \epsilon]$ and no other eigenvalues in $[-1/2, 1/2]$.  There exist $\{ f_{1}[p], f_{2}[p], f_{3}[p] \}$ that denote an orthonormal basis 
of the \emph{approximate kernel} where $f_{i}[p] \in C^\infty(\widetilde{S}[p])$.  Moreover, $f_{i}[p]$ depends continuously on  the parameters of the construction and satisfies
\[||f_{i}[p] - \hat f_i: C^{2,\beta}(S_5[p],h) || < \epsilon.
\]
\end{enumerate}
\end{prop}

Before we proceed with the proof, we make the following relabeling of transition regions.
\addtocounter{equation}{1}
\begin{definition}\label{genericlabel}
 For any standard region $S\pen$, define $\Lambda_{close}, \Lambda_{far}$ such that 
 \[
 \Lambda_{close}:= \Lambda \pen
 \]
 \[
 \Lambda_{far}:=\left\{ \begin{array}{ll} \Lambda[p,e,n+1]& \text{if } n < l(e),\\
 \Lambda[p^-[e],e,l(e)]& \text{if } n=l(e).\end{array}\right.
 \]
  \end{definition}
Notice that the wording of this definition makes sense everywhere except on the middle of each edge. That is, except when $n=l(e)$, the labeling of ``close'' and ``far'' correspond to relative distance to the nearest boundary circle on each $M[e]$. At the center of $M[e]$, we keep the labeling but note that both transition regions are closer to a boundary than the standard region between them.

We now prove the proposition.
\begin{proof}
We prove both items in the proposition by taking advantage of the results of \cite{KapAnn}, Appendix B, which determine relationships between eigenfunctions and eigenvalues of two Riemannian manifolds that are shown to be
close in some reasonable sense. Before proceeding, note that the inequality B.1.6 in \cite{KapAnn} should read
\[
 ||F_if||_\infty \leq 2 ||f||_\infty.
\]

To simplify the language of the proof, we let $\widetilde S$ denote a general $\widetilde S\pen$. Moreover, let $\widetilde C_{close}, \widetilde C_{far}$ denote the center meridian circle of $\Lambda_{close}, \Lambda_{far}$ respectively. Finally, let $\widetilde Y, \widetilde Y_{close}, \widetilde Y_{far}$ denote the immersions defined in \eqref{tildeYpen} for the 
particular $\pen$ of interest and those that also share $\Lambda_{close}, \Lambda_{far}$, respectively, on their extended standard regions. Thus, for $2 \leq n \leq l(e)$, $\widetilde Y_{close} := \widetilde Y[p,e,n-1]$, and for $1 \leq n \leq l(e)-1$, $\widetilde Y_{far}:= \widetilde Y[p,e,n+1]$. Moreover, when $n=1$, $\widetilde Y_{close} := \widetilde Y[p]$ and when $n=l(e)$, $\widetilde Y_{far}:= \widetilde Y[p^-[e],e,l(e)]$.

We begin by considering $\widetilde S \backslash \left(\widetilde C_{close} \cup \widetilde C_{far}\right)$.
Let $\widetilde S^+$ denote the connected component containing $S$ and let $\Lambda_{close}^-, \Lambda_{far}^-$ denote the two remaining connected components where naturally we have $\Lambda_{close}^- \subset \Lambda_{close}$. Let 
$D_{close}, D_{far}$ denote the smallest geodesic disks in $\mathbb S^2$ that contain, respectively, $\widetilde Y_{close}(\Lambda_{close}^-), \widetilde Y_{far}(\Lambda_{far}^-)$. Consider two abstract copies of $\Ss^2$, $\{1\}\times \Ss^2, \{2\}\times \Ss^2$
and their disjoint union $\{1,2\}\times \Ss^2$.

We first determine a comparison Riemannian manifold $(\hat S, \hat g)$ on which the linear operator of interest is well understood. In order to apply the results of \cite{KapAnn}, we must show there exist linear
maps such that all of the assumptions of B.1.4 are satisfied. To that end, let 
\[
 \hat S := \left(\{1\} \times \mathbb S^2\right) \bigcup \left(\{2 \} \times \left( D_{close} \cup D_{far}\right)\right)
\]and define $\hat g$ on $\{1,2\}\times \Ss^2$ by making $\hat g$ the restriction to $g_{\Ss^2}$ on each copy of $\Ss^2$. Consider the operator $\Delta_{\Ss^2}$ on $\Ss^2$ with the standard metric. It it well 
known that the lowest eigenvalues for this operator are $0, 2, 6$, where the eigenvalue $2$ has multiplicity $3$. Therefore, the only eigenvalue of $\Delta_{\Ss^2}+2$  in the interval $[-1,1]$ is
zero, with multiplicity $3$. Choosing $b$ sufficiently large, thus ensuring the radii of $D_{close}, D_{far}$ are sufficiently small,
we can guarantee the lowest eigenvalue of $\Delta_{\hat g}$ on $(D_{close,far}, \hat g)$ with Dirichlet data must be larger than $3$. 

Define $\widetilde {\underline Y}: \widetilde S \to \hat S$ such that $\widetilde {\underline Y} (x) := (1,\widetilde Y(x))$ if $x \in \widetilde S^+$ and $\widetilde {\underline Y}(x) :=(2,\widetilde Y_{close,far}(x))$ if $x \in \Lambda_{close,far}^-$.
We now construct two functions $F_1, F_2$ that will satisfy the assumptions of B.1.4 in \cite{KapAnn}, thus allowing us to appeal to Theorem B.2.3 in that same paper. First, let $\overline \psi$ be a cutoff function on $\widetilde S$ such that $\overline \psi \equiv 1$ on $S$ and $\overline \psi := \psi[2d,d]\circ \oux$ on $\Lambda_{close}, \Lambda_{far}$. This provides a necessary logarithmic cutoff function that yields appropriate gradient bounds; see B.1.8 in \cite{KapAnn}. For $f \in C_0^\infty(\widetilde S)$, define $F_1(f) \circ \widetilde{\underline Y}=\overline \psi \cdot f$ and for $f \in C_0^\infty(\hat S)$ let $F_2(f)=\overline \psi \cdot f \circ \widetilde {\underline Y}^{-1}$. For $d$ sufficiently large, the assumptions of B.1.4 are easily checked and one can apply Theorem B.2.3 with $r=2, n=4, c'=1$. That is, for the operator $\Delta_{\Ss^2}$, there is a gap between the first and second eigenvalues and between the fourth and fifth eigenvalues (counting multiplicity).
 
 Now recall that an orthonormal basis of the kernel of $\Delta_{\Ss^2}+2$ are the functions
 \[\hat f_{i, \Ss^2}:= \frac{\Be_i\cdot N_{\Ss^2}}{||\Be_i\cdot N_{\Ss^2}:L^2(\Ss^2)||},\] $i=1,2,3$, where $\Be_i$ are the standard unit basis vectors for $\Real^3$ and $N_{\Ss^2}$ is the outward normal to the unit sphere. Note $\hat f_{i, \Ss^2} \in C_0^\infty(\hat S)$ and on $S_d$ we can determine 
\[F_2(\hat f_{i,\Ss^2}) =\hat f_{i,\Ss^2} \circ \widetilde Y.
\] Moreover, by Proposition \ref{geolimit}, for $\outau$ sufficiently small 
\[
||F_2(\hat f_{i,\Ss^2})-\cf_i:C^{2,\beta}( S_6, h)|| < \epsilon/2.
\]

To finish the proof we make the following observations. Because of the uniform geometry of $S_6$ in the $h$ metric -- for $\outau$ sufficiently small-- we can use standard linear theory on the interior to increase the $L^2$ norms of Appendix B to $C^{2,\beta}$ norms. To make the dependence continuous, we let $f_{i}$ denote the normalized $L^2(\widetilde S, h)$ projection of $\cf_i$ onto the span of $F_2(\hat f_{i, \Ss^2})$. 

The second half of the theorem is nearly identical. We sketch here a few of the main differences. Let $\widetilde S$ denote an $\widetilde S[p]$ and let $\Lambda_j:=\Lambda[p,e_j,1]$, $j=1,\dots, |E_p|$, denote the adjacent transition regions. Let $\widetilde C_j$ denote the center meridian circle of each $\Lambda_j$. As before $\widetilde S^+$ denotes the connected component of $\widetilde S \backslash \left(\cup_{j=1}^{|E_p|} \widetilde C_j\right)$ and $\Lambda^-_j$ denote the other $|E_p|$ components where the enumeration choice is the obvious one. Let $D_j$ be the smallest geodesic disk in $\Ss^2$ containing $\widetilde Y[p,e_j,1](\Lambda^-_j)$. We set 
\[
\hat S:= \left(\{1\} \times \widetilde S^+\right) \cup \left(\{2\} \times D_1\right) \cup \cdots \cup \left(\{|E_p|+1\} \times D_{|E_p|}\right)
\]and $\hat g$ is the restriction of $g_{\Ss^2}$ to the appropriate copy of $\Ss^2$.
Let $\widetilde{\underline Y}: \widetilde S \to \hat S$ such that $\widetilde {\underline Y} (x) := (1,\widetilde Y(x))$ if $x \in \widetilde S^+$ and $\widetilde {\underline Y}(x) :=(j+1,\widetilde Y[p,e_j,1](x))$ if $x \in \Lambda_j^-$.
Again, using $\overline \psi$ defined to be identically $1$ on $S$ and $\overline \psi := \psi[2d,d]\circ \oux$ on $\Lambda_j$ for each $j$, we have the appropriate log cutoff function. From here, after defining
$F_1, F_2$ in an analogous manner and compare with $\hat f_i$ rather than $\cf_i$, the result proceeds exactly as before. 
\end{proof}

\subsection*{The extended substitute kernel}
Following the general methodology of \cite{KapClay,KapWente}, we introduce the extended substitute kernel. The extended substitute kernel, which we denote by $\mathcal K$, is the direct sum of three-dimensional spaces corresponding to each standard region -- central and non-central -- and each attachment. We first define the substitute kernel, the space of functions which will allow us to solve the semi-local linearized problem.
\addtocounter{equation}{1}
\begin{definition}
We let\[
\mathcal K_V \oplus \mathcal K_S:=  \bigoplus_{p \in V(\Gamma)} \mathcal K[p]\bigoplus_{[p,e,n] \in S(\Gamma)} \mathcal K[p,e,n]
\]denote the \emph{substitute kernel}. 
The spaces on the right hand side of the expression are defined below.
\end{definition}

Let $\mathcal{K}[p,e,n] = \langle w_1[p,e,n], w_2[p,e,n], w_3[p,e,n] \rangle_{\mathbb{R}}$ where we define,
for $i =1,2,3$, $\pen \in S(\Gamma)$,
\[w_i[p^+[e],e,n] := c_i [p^+[e],e,n](\psi[2nP_{\uthte} - 1, 2nP_{\uthte}](t))(\psi[2nP_{\uthte}+1, 2nP_{\uthte}](t))\cf_i,\]
\begin{align*}w_i[p^-[e],e,n] := c_i[p^-[e],e,n] &(\psi[\RH-(2nP_{\uthte} -1),\RH-2nP_{\uthte}](t))\\ \cdot &(\psi[\RH-(2nP_{\uthte}+1), \RH-2nP_{\uthte}](t))\cf_i
\end{align*}
for each $[p,e,n]\in S(\Gamma)$, where the coefficients $c_i\pen$
depend smoothly on the parameters of the construction and are determined by requiring
\addtocounter{theorem}{1}
\begin{equation}\label{wpendef}
 \int_{S[p,e,n]}w_i[p,e,n] \cf_j \:dh = \delta_{ij}.
\end{equation}
Notice the symmetry of the functions $\cf_j$ on $S\pen$ in the $h$ metric and the symmetric choice of the $w_i[p,e,n]$ imply that the coefficients are not overdetermined. Moreover, the controlled geometry of $S\pen$ in the $h$ metric implies the coefficients are well controlled. 

To define $\mathcal K[p]=\langle w_1[p], w_2[p], w_3[p] \rangle_{\mathbb{R}}$, we first need to determine a new cutoff function. 
Let $Sym[p] \subset \Ss^2_{V_p}$ denote the largest region symmetric in all three functions $x_1, x_2, x_3$ such that $\dist_{\Ss^2}(\partial Sym[p], \partial \Ss^2_{V_p}) > \delta$. Because the graph is finite, for $\outau$ small there is a uniform lower bound on the area
of $Sym[p]$. In fact, under the assumptions of the construction, the distance between the centers of any two geodesic disks removed on $M[p]$ must be at least $\pi/3+\delta$.
The controlled geometry of $\Ss^2_{V_p}$ implies that there exists a symmetric subset $ {Sym}'[p] \subset Sym[p]$ such that the exponential map from $\partial {Sym}'[p]$ into $Sym'[p]$ is one to one on a
$2\delta$ neighborhood.

For $X \subset \Ss^2$, let $T_r(X)$ denote the $r$ tubular neighborhood of $X$ on $\Ss^2$ where $r$ is measured in the induced metric. Define the function $\psi[p]$ such that $\psi[p] \equiv 1$ on $Sym'[p] \backslash T_{2\delta}(\partial Sym'[p])$ and smoothly cuts off to zero by $\partial T_\delta(\partial Sym'[p])\cap Sym'[p]$, remaining zero thereafter. Moreover, we can choose this cutoff so that it preserves the symmetry of the coordinate functions. That is $\int_{\Ss^2} \psi[p] x_i x_j =0$ if $i \neq j$.

With this definition, we define the relevant functions on the central spheres.
For $i=1,2,3$, let $$w_i[p]:=c_i' \psi[p] \hat f_i$$ on $S[p]$ where the coefficients $c_i'$ depend smoothly on the parameters of the construction and are determined by requiring
\addtocounter{theorem}{1}
\begin{equation}\label{noncentw}
 \int_{S[p]} w_i[p] \hat f_j \: dh = \delta_{ij}.
\end{equation}
Because of the construction, the coefficients are again determined solely when $i=j$ as the other terms vanish; thus the problem is not overdetermined. Moreover, the uniform lower bounds on area and the choice 
of domain where $\psi[p]$ is non-zero imply that the coefficients are uniformly bounded, independent of $\outau$.


We now use elements of the substitute kernel to modify the inhomogeneous term and ensure that the modified term is orthogonal to the approximate kernel, see items \eqref{esl3}, \eqref{esl4} below. 

\addtocounter{equation}{1}
\begin{lemma}\label{extsubslemma}
For each $p\in V(\Gamma), [p,e,n]\in S(\Gamma)$, the following hold:
\begin{enumerate}
\item $w_i[p,e,n]$ is supported on $S[p,e,n]$ and $w_i[p]$ is supported on $S[p]$.
\item $||w_i[p,e,n], w_i[p]: C^{2, \beta}(M, \chi)|| \leq C$.
\item\label{esl3} For $E \in C^{0, \beta}(\widetilde S[p,e,n], \chi)$ there is a unique $\widetilde w \in \mathcal{K}[p,e,n]$ such that $\cfE E+ \widetilde w$ is $L^2(\widetilde S[p,e,n],h)$-orthogonal to the approximate
 kernel on $\widetilde S[p,e,n]$.  Moreover, if $E$ is supported on $S_1[p,e,n],$ then
\[||\widetilde w: C^{2, \beta}(M, \chi)|| \leq C(b, \beta)||E:C^{0, \beta}(\widetilde S[p,e,n], \chi)||.\]
\item\label{esl4} For $E \in C^{0,\beta}(\widetilde S[p],\chi)$ there is a unique $\widetilde w \in \mathcal{K}[p]$ such that $\cfE E+\widetilde w$ is $L^2(\widetilde S[p],h)$ orthogonal to the approximate kernel on $\widetilde S[p]$. Moreover, if $E$ is supported on $S_1[p]$, then
\[
||\widetilde w:C^{2,\beta}(M,\chi)||\leq C(b,\beta)||E:C^{0,\beta}(\widetilde S[p],\chi)||.
\]
\end{enumerate}

\end{lemma}
\begin{proof}
The first and second items follow by definition of each of the functions and the uniform control on the coefficients, independent of $\outau$. The third and fourth items follow immediately from Proposition \ref{approxkerprop}, the uniform equivalence of the $h, \chi$ metrics on $S_1[p], S_1\pen$, and the definition of elements of $\mathcal K[p],\mathcal K\pen$. 
\end{proof}

When solving the linearized equation, we must produce solutions on each extended standard region that satisfy sufficiently fast exponential decay conditions on one of the adjacent transition regions. (For central standard regions, the decay must hold on all adjacent transition regions.) This is the second use of the extended substitute kernel; we modify the solutions on $C_1^+[p,e,n'']$ by prescribing the low harmonics of the solution to ensure such fast decay. The modification to the solution on each $\widetilde S[p], \widetilde S\pen$ corresponds to an additional modification of the inhomogeneous term by elements of $\mathcal K$. To prescribe the low harmonics at each attachment to a central sphere, we introduce a three dimensional space of functions supported in a neighborhood of each attachment.

Let $\mathcal{K}\pe = \langle w_1\pe, w_2\pe, w_3\pe \rangle_{\mathbb{R}}$ where we define,
for $i=1,2,3$, $\pe \in A(\Gamma)$,
\addtocounter{theorem}{1}
\begin{equation}\label{tildecs}
w_i\pe:=\widetilde c_i\pe \mathcal L_h(\psi[0,1]\circ \underline x \: V_i[\Lambda[p,e,1],1,0]), 
\end{equation}
 where the $\widetilde c_i\pe$ are normalized constants so that
\addtocounter{theorem}{1}
\begin{equation}\label{tildecs2}
\begin{array}{ll}
 \widetilde c_1\pe V_1[\Lambda[p,e,1],1,0]=1,& \\ \widetilde c_2\pe V_2[\Lambda[p,e,1],1,0]=\cos \theta, &\text{ on } C_1^+[p,e,0].\\ \widetilde c_3\pe V_3[\Lambda[p,e,1],1,0]=\sin \theta &
 \end{array}
\end{equation}The $\widetilde c_i\pe$ depend smoothly on the parameter $\taue$. Moreover, Proposition \ref{LambdaKernel} and \eqref{Vlambda} imply uniform upper and lower bounds on $\widetilde c_i\pe$ for $i=1,2,3$. 

We now define the extended substitute kernel.
\addtocounter{equation}{1}
\begin{definition}
Let
\[
 \mathcal K := \mathcal K_V \oplus \mathcal K_S \oplus \mathcal K_A :=\bigoplus_{p \in V(\Gamma)} \mathcal K[p]\bigoplus_{[p,e,n] \in S(\Gamma)} \mathcal K[p,e,n]\bigoplus_{[p,e] \in A(\Gamma)}\mathcal K\pe
\]denote the \emph{extended substitute kernel}. 
\end{definition}
We introduce $\mathcal K_v \pen, \mathcal K_v \pe$ so that for $v \in \mathcal K_v\pen$, $\mathcal L_h v \in \mathcal K\pen$ and likewise for $\mathcal K_v\pe$. On each $\widetilde S\pen$, the substitute kernel itself can be used to modify the solutions to guarantee exponential decay. On the other hand, for the central standard regions we define the spaces $\mathcal K\pe$ solely to arrange for the decay away from a central sphere. The prescription of the low harmonics near a central standard region is understood by considering the functions $\cf_i$ rather than $\hat f_i$. Recall that on $C_1^+[p,e,0]$, the functions $\cf_i$ are multiples of $\nu_{\taue} \cdot \Be_i$ where $\nu_{\taue}$ refers to \eqref{normalvector}. 

\addtocounter{equation}{1}
\begin{definition}\label{tildefdef}
For $\penpp \in N^+(\Gamma)$ and $n''>0$,
let 
\[\mathcal{K}_v\penpp :=\langle v_{1}\penpp , v_{2}\penpp , v_{3}\penpp  \rangle_{\mathbb{R}}\]
where $v_{i}\penpp $
is supported on $\widetilde S\penpp $ and 
$v_{i}\penpp := c_i\penpp(f_{i}\penpp + u_{i}\penpp )$ where
\begin{itemize}
\item $f_i\penpp$ are as in Proposition \ref{approxkerprop}.
\item $\widetilde w_{i}\penpp $ is chosen such that $-\mathcal{L}_h f_{i}\penpp  + \widetilde w_{i}\penpp $ is $L^2(\widetilde S\pen, h)$ orthogonal to the approximate kernel on 
$\widetilde S\penpp $. 
\item $u_{i}\penpp $ solves the problem $\mathcal{L}_h u_{i}\penpp = -\mathcal{L}_h f_{i}\penpp  + \widetilde w_{i}\penpp $ on $\widetilde S\penpp $ with zero Dirichlet data.

\item $c_i\penpp$ is such that on $C_1^+\penpp $ we have $c_1 \penpp\cf_1=1,\: c_2 \penpp\cf_2 = \cos \theta, \: c_3 \penpp\cf_3 =  \sin \theta.$
\end{itemize}

Let $\mathcal K_v\pe:=\langle v_{1}\pe, v_{2}\pe, v_{3}\pe \rangle_{\mathbb{R}}$ where $v_{i}\pe$ is supported on $\Lambda[p,e,1]$ and 
\[v_{i}\pe:= \widetilde c_i\pe\psi[0,1]\circ \underline x V_i[\Lambda[p,e,1],1,0]\]
 where the $\widetilde c_i\pe$ are defined by \eqref{tildecs2}.
\end{definition}
The choice of construction implies that a solution to the linearized problem on $\widetilde S[p,e,l(e)]$ will not be modified to induce the exponential decay along either adjacent transition region. On this extended standard region, the solutions will possess the linear decay guaranteed by Corollary \ref{LambdaCloseCor}. This will prove sufficient for the following reason. Solving the global linearized problem requires that the semi-local solutions coming from adjacent extended standard regions must have one solution decay exponentially on the common transition region (away from the nearest central sphere) and one decay linearly on the common transition region (toward the nearest central sphere). The linear decay on both sides of $\widetilde S[p,e,l(e)]$ corresponds to the decay toward the nearest central sphere on each of its transition regions.

\addtocounter{equation}{1}
\begin{lemma}\label{vlemma}For $\epsilon_4>0$ there exists $b$ large enough so that the $v_i\pe,v_{i}\penpp, n''>0$ defined previously are smooth on the corresponding $\widetilde S[p],\widetilde S[p,e,n'']$ and satisfy
\begin{enumerate}
\item \label{v1}$\mathcal L_h v_{i}[p,e,n] \in \mathcal{K}[p,e,n]$ and $\mathcal L_h v_i \pe \in \mathcal K\pe$. The support of $\mathcal L_h v_{i}[p,e,n]$ is on $S[p,e,n]$, and the support of $\mathcal L_h v_i\pe$ is on $S_1[p]$.
\item \label{v2}$v_i[p,e,n]=0$ on $\partial \widetilde S[p,e,n]$, $v_i\pe=0$ on $\partial \widetilde S[p]$. 
\item \label{v3}$||v_{i}[p,e,n] : C^{2, \beta}(\widetilde S[p,e,n], \chi)||\leq C(b, \beta)$, $||v_{i}\pe : C^{2, \beta}(\widetilde S[p], \chi)||\leq C(b, \beta)$.
\item \label{v4}For $n''>0$, $||v_{1}[p,e,n''] -1, v_{2}[p, e,n''] -\cos \theta,v_{3}[p,e,n''] - \sin \theta: C^{2, \beta}(C^+_1[p,e,n''], \chi)|| < \epsilon_{4}$ and on $C^+_1[p,e,0]$, $v_1\pe=1, v_2\pe= \cos \theta, v_3\pe=\sin \theta$.
\end{enumerate}

\end{lemma}
\begin{proof}
 Items \eqref{v1}, \eqref{v2} follow immediately from the choice of construction. For $\pe$, \eqref{v3} and \eqref{v4} follow from \eqref{tildecs2} and Proposition \ref{LambdaKernel}. For $\pen$, we use 
the following. First, by Proposition \ref{approxkerprop}, $||\mathcal L_h f_i\pen:L^2(\widetilde S\pen,h)||< \epsilon$. 
Then the estimates from Lemma \ref{extsubslemma} for $\widetilde w_i\pen$ and standard theory imply $||u_i\pen:L^2(\widetilde S\pen,h)||<C\epsilon$. Upgrading the $L^2$ estimate to a $C^{2,\beta}$ estimate on the interior, and the uniform equivalence of the $h, \chi$ metrics on $ S_1\pen$, we get $||u_i\pen:C^{2,\beta}(S_1\pen, \chi)||\leq C(b,\beta)\epsilon$. As $\mathcal L_\chi v_i\pen=0$ on $\Lambda_{0,1}\pen, \Lambda_{1,0}[p,e,n+1]$, we can apply Corollary \ref{LambdaCloseCor} to get \eqref{v3}. Finally, \eqref{approxkerest} and the interior estimates on $u_i\pen$ imply \eqref{v4}.
\end{proof}

\subsection*{Solving the linearized equation semi-locally}
Using Lemmas \ref{extsubslemma}, \ref{vlemma}, we now solve the linearized problem on each extended standard region and obtain solutions with appropriate decay. Notice we presume the inhomogeneous term is supported in a small neighborhood of the standard region so the global proof must allow for precisely this setup. The proofs for central and non-central standard regions are essentially the same, though on each central standard region one must modify the solution on every adjacent transition region.
\addtocounter{equation}{1}
\begin{definition}Recalling Definition \ref{genericlabel},
on each extended standard region $\widetilde S\pen:=\widetilde S$, we define the function $f[\widetilde S]$:
\addtocounter{theorem}{1}
\begin{equation}
f[\widetilde S](x) = \left\{ \begin{array}{ll}
\frac{\underline x + 1}{\ell},& \text{if }x \in \Lambda_{close}\\
1 + \frac{1}{\ell}\psi(t),& \text{if } x=(t,\theta) \in S \\
 e^{-\gamma \underline x},& \text{if } x \in \Lambda_{far}, n<l(e)\\
 \frac{\underline x + 1}{\ell},& \text{if }x \in \Lambda_{far}, n=l(e).
\end{array}\right.
\end{equation}When $n< l(e)$, take $\psi(t)$ to be a standard cutoff function that is $1$ in a neighborhood of $\Lambda_{close}\cap S$ and $0$ in a neighborhood of $\Lambda_{far}\cap S$. For $n=l(e)$, take $\psi(t) \equiv 1$.
\end{definition}

\addtocounter{equation}{1}
\begin{lemma} \label{linearpartpen}Let $S:= S\pen$ denote a standard region on $M$. 
Given $\gamma\in (1,2), \beta \in (0,1)$, there is a linear map 
\[\mathcal{R}_{\widetilde S}:\{E\in C^{0, \beta}(\widetilde S,\chi): E \text{ is supported on } S_{1}\} \to C^{2, \beta}(\widetilde S, \chi) \times \mathcal{K}\pen \] 
such that the following hold for $E$ in the domain of $\mathcal{R}_{\widetilde S}$ above and $(\varphi, w) = \mathcal{R}_{\widetilde S}(E)$:
\begin{enumerate}
\item $\mathcal{L}_\chi \varphi = E +\cfw w$ or $\mathcal L_h \varphi = \frac{2\rho^2}{|A|^2}E + w$ on $\widetilde S$.
\item $\varphi$ vanishes on $\partial \widetilde S$.
\item $||\varphi:C^{2,\beta}(\widetilde S,\chi, f[\widetilde S])|| \leq C(\beta, b,\gamma)||E:C^{0,\beta}( S_1, \chi, f[\widetilde S])||$.
\item \label{mubound}For $i=1,2,3$, $|\mu_i|\leq C(b,\beta)||E:C^{0,\beta}(S_1,\chi)||$ where $w:=\sum_{i=1}^3 \mu_i w_i\pen$.
\item $\mathcal{R}_{\widetilde S}$ depends continuously on all of the parameters of the construction.
\end{enumerate}
\end{lemma}

\begin{proof}
 Lemma \ref{extsubslemma} implies there exists $\widetilde w \in \mathcal K\pen$ and $\varphi' \in C^{2,\beta}(\widetilde S)$ such that $\mathcal L_h \varphi' = \frac{2\rho^2}{|A|^2}E + \widetilde w$ and $\varphi'|_{\partial \widetilde S}=0$. The controlled geometry for $\outau$ small implies \eqref{mubound} holds for $\mu_i'$ where $\widetilde w:=\sum_i \mu_i' w_i\pen$. Now observe that the designated support of $E$ and the construction of $\widetilde w$ implies that we can apply Corollary \ref{LambdaCloseCor} to $\varphi'$ on $\Lambda_{close}$ to produce the necessary decay there. If $n=l(e)$, the same corollary provides necessary estimates on $\Lambda_{far}$ and the proof is complete. 
 
For $n<l(e)$, we consider how to modify $\varphi'$ on $\Lambda_{far}$ to prescribe the fast exponential decay.  Let $\varphi'_T$ denote the projection of $\varphi'$ onto $\mathcal H_1[C^+_1]$ where the $C^+_1$ of interest here is in $\Lambda_{far}$. Let $\varphi'_\perp =\varphi'-\varphi'_T$ on $C_1^+\subset \Lambda_{far}$ and let $V_{\varphi'} := \mathcal R_\partial (\varphi'_\perp|_{C_1^+})$.  Notice that $(\varphi' - V_{\varphi'})|_{C_1^+} \in \mathcal H_1[C_1^+]$ and we denote 
\[
(\varphi' - V_{\varphi'})|_{C_1^+} = \alpha_1 + \alpha_2 \cos \theta + \alpha_3 \sin \theta.
\]
Standard theory implies $||\varphi':C^{2,\beta}(\widetilde S,\chi)||\leq C(b,\beta) ||E:C^{0,\beta}(\widetilde S,\chi)||$. Coupling this with Corollary \ref{Rpartial} \eqref{rp4} implies $|\alpha_i| \leq C(b,\beta) ||E:C^{0,\beta}(\widetilde S,\chi)||$.
We now choose $v \in \mathcal K_v[p,e,n]$ uniquely such that on $C_1^+$, $\varphi' + v = V_{\varphi'} + \mathcal R_\partial(v_\perp|_{C_1^+})$. 
We may write $v = \sum_{i=1}^3 \mu_i'' v_i[p,e,n]$. For each $i$, let $V_i :=\mathcal R_\partial(v_i[p,e,n]_\perp|_{C_1^+})$. It is enough to show that for each $i=1,2,3$, there exists a unique $\mu_i''$ such that 
\[
\mu_i''(V_i - v_i)|_{C_1^+} = \alpha_i(\delta_{i1}+\cos \theta \delta_{i2} + \sin \theta \delta_{i3}).
\] Lemma \ref{vlemma} \eqref{v4} and Corollary \ref{Rpartial} \eqref{rp3} give uniform (positive) lower bounds on the difference of the coefficients of $V_i - v_i$ at $C_1^+$. This, coupled with the bounds on the $|\alpha_i|$ implies
\[
|\mu_i''| \leq C(b,\beta)||E:C^{0,\beta}( S_1, \chi, f[\widetilde S])||.
\]
 Setting $\varphi = \varphi' +v$ and $w = \widetilde w + \mathcal L_h v$, we see $\mathcal R_\partial \varphi_\perp = \varphi$ on $C_1^+$ and thus Corollary \ref{Rpartial} provides the necessary decay on $\Lambda_{far}$. 
\end{proof}
The procedure for solving the linearized problem with appropriate decay at each central standard region is nearly identical. In this case, we use $\mathcal K\pe$ rather than $\mathcal K\pen$ to modify our solutions but all of the estimates will follow by the logic outlined in the previous proof.

Observe that in the statement of the lemma, the term $\cfw$ does not appear. This is a simple consequence of the fact that $\rho^2 \equiv |A|^2/2$ on $\supp(w)\subset M[p]$.
\addtocounter{equation}{1}
\begin{lemma}\label{linearpartp}
Given $\gamma \in (1,2), \beta \in (0,1)$, there is a linear map
\[
\mathcal R_{\widetilde S[p]}:\{E \in C^{0,\beta}(\widetilde S[p], \chi):E \text{ is supported on } S_1[p]\} \to C^{2,\beta}(\widetilde S[p], \chi) \times (\mathcal K[p]\cup_{e \in E_p} \mathcal K\pe) \] such that the following hold for $E$ in the domain of $\mathcal R_{\widetilde S[p]}$ and $(\varphi,w)=\mathcal R_{\widetilde S[p]}(E)$:
\begin{enumerate}
\item $\mathcal{L}_\chi \varphi = E +w$ or $\mathcal L_h \varphi = \frac{2\rho^2}{|A|^2}E + w$ on $\widetilde S[p]$.
\item $\varphi$ vanishes on $\partial \widetilde S[p]$.
\item $ ||\varphi: C^{2, \beta}(\widetilde S[p], \chi)|| \leq C(b, \beta)||E: C^{0,\beta}(S_1[p], \chi)||$.
\item $\sum_{i} |\mu_i[p]| + \sum_{i,j}|\mu_{i}[p,e_j]| \leq C(b,\beta)||E:C^{0,\beta}( S_1, \chi)||$ where \\$w=\sum_i \mu_i[p]w_i[p]+ \sum_{i,j} \mu_{i}[p,e_j] w_i[p,e_j]$.
\item $ ||\varphi:C^{2,\beta}(\Lambda[p,e,1], \chi, e^{-\gamma \underline x})|| \leq C(b, \beta, \gamma)||E:C^{0,\beta}(S_1[p], \chi)||$ for all $e \in E_p$. 
\item $\mathcal{R}_{\widetilde S[p]}$ depends continuously on all of the parameters of the construction.
\end{enumerate}
\end{lemma}

\begin{proof}Lemma \ref{extsubslemma} implies there exists $\widetilde w \in \mathcal K[p]$ and $\varphi' \in C^{2,\beta}(\widetilde S[p])$ such that $\mathcal L_h \varphi' = \frac{2\rho^2}{|A|^2}E + \widetilde w$ and $\varphi'|_{\partial \widetilde S[p]}=0$. For each $\Lambda[p, e_j,1]$ where $e_j \in E_p$ we modify $\varphi'$ in a manner identical to the previous proof. On each $C_1^+[p,e_j,0]$, find the unique $v_j \in \mathcal K_v[p,e_j]$ such that for $\varphi_j = \varphi' + v_j$, and $\varphi_{j,\perp}$ representing the part orthogonal to $\mathcal H_1[C_1^+[p,e_j,0]]$ on $C_1^+[p,e_j,0]$, one has $\mathcal R_\partial (\varphi_{j,\perp}) = \varphi_j$. The estimates on $\varphi_j, v_j$ and the unique choice of $v_j$ follow as they did previously.  Now set
\[
\varphi = \varphi' + \sum_{j=1}^{ |E_p|}v_j. 
\]Observe that $\supp(v_i) \cap \supp(v_j) = \emptyset$ for $i \neq j$ and thus $\varphi|_{\Lambda[p,e_j,1]} = \varphi_j$. Finally, setting $w= \widetilde w + \sum_{j=1}^{|E_p|} \mathcal L_h v_j$ gives the result.
\end{proof}

\subsection*{Solving the linearized equation globally}
We can now solve the linearized problem semi-locally on each transition region and each extended standard region. To get the estimates we desire, the inhomogeneous term for each extended standard region must be supported in a neighborhood of the standard region. Thus as a first step toward solving the global linearized problem, we determine a partition of unity on $M$ that we apply to the inhomogeneous term. Solving the problem on each region separately, we then patch the solutions back together to construct a global function. Of course, this process will introduce some error, which we correct by iteration.

We first introduce all of the cutoff functions we require.
\addtocounter{equation}{1}
\begin{definition} We define uniquely smooth functions
$\psi_{S[p]}, \psi_{\widetilde S[p]}, \psi_{S\pen}, \psi_{\widetilde S\pen}, \psi_{\Lambda[p,e,n']}$ such that
\begin{enumerate}[(i)]
\item $\psi_{S[p]} = \psi_{\widetilde S[p]} \equiv 1$ on $S[p]$.
\item $\psi_{S[p]}=\psi[1,0] \circ \underline x$ on each $\Lambda[p,e,1]$ for $e \in E_p$.
\item $\psi_{\widetilde S[p]}= \psi[0,1] \circ \overline x$ on each $\Lambda[p,e,1]$ for $e \in E_p$.
\item $\psi_{S\pen}=\psi_{\widetilde S\pen}\equiv 1$ on $S\pen$.
\item For $n<l(e)$, $\psi_{S[p,e,n]} = \psi[1,0] \circ \overline x$ on $\Lambda[p,e,n]$ and $\psi_{S[p,e,n]} = \psi[1,0] \circ \underline x$ on $\Lambda[p,e,n+1]$.
\item $\psi_{S[p,e,l(e)]}= \psi[1,0] \circ \overline x$ on $\Lambda[p,e,l(e)]$ and $\psi_{S[p,e,l(e)]}= \psi[1,0] \circ \overline x$ on $\Lambda[p^-[e],e,l(e)]$.
\item For $n<l(e)$, $\psi_{\widetilde S[p,e,n]} = \psi[0,1] \circ \underline x$ on $\Lambda[p,e,n]$ and $\psi_{S[p,e,n]} = \psi[0,1] \circ \overline x$ on $\Lambda[p,e,n+1]$.
\item $\psi_{\widetilde S[p,e,l(e)]}= \psi[0,1] \circ \underline x$ on $\Lambda[p,e,l(e)]$ and $\psi_{S[p,e,l(e)]}= \psi[0,1] \circ \underline x$ on $\Lambda[p^-[e],e,l(e)]$.
\item $\psi_{\Lambda[p,e,n']} = \psi[0,1]\circ \oux$ on $\Lambda[p,e,n']$. 
\end{enumerate}
\end{definition}
Observe that $\psi_{S[p]}, \psi_{S\pen}, \psi_{\Lambda[p,e,n']}$ form a partition of unity on $M$. Also note that each of the functions $\psi_{\widetilde S[p]}, \psi_{\widetilde S\pen}$ are identically 1 on almost all of $\widetilde S[p], \widetilde S\pen$, respectively. Near the boundary they transition smoothly to zero. Finally, $\supp(\psi_{S[p]}) \subset S_1[p], \supp (\psi_{S\pen} )\subset S_1\pen$.

We now set the notation for defining a global $C^{2,\beta}$ function by pasting together appropriately cutoff local functions.
\addtocounter{equation}{1}
\begin{definition}\label{patching}
Let $u[p] \in C^{k, \beta}( \widetilde S[p])$, $u\pen \in C^{k,\beta}(\widetilde S\pen)$, $p \in V(\Gamma), \pen \in S(\Gamma)$, be functions that are 
zero in a neighborhood of $\partial \widetilde S[p]$, $\partial \widetilde S\pen$. We define $U=\mathbf U(\{u[p],u\pen\})\in C^{k,\beta}(M)$ to be the unique function such that
\begin{enumerate}[(i)]
\item $U|_{S[p]}=u[p], U|_{S\pen}=u\pen$.
\item $U|_{\Lambda[p,e,1]}=u[p] + \sum_{e \in E_p}u[p,e,1]$.
\item For $n'<l(e)$, $U|_{\Lambda[p,e,n']} = u[p,e,n'-1] + u[p,e,n']$.
\item For $U|_{\Lambda[p^+[e],e,l(e)]} = u[p^+[e],e,l(e)-1] + u[p^+[e],e,l(e)]$ while $U|_{\Lambda[p^-[e],e,l(e)]}=u[p^-[e],e,l(e)-1]+u[p^+[e],e,l(e)]$. 
\end{enumerate}
\end{definition}

We define the global norms that we will use in the statement and proof of the global linearized problem and in the main theorem. 
\addtocounter{equation}{1}
\begin{definition}
For $k \in \mathbb N$, $\beta \in (0,1)$, and $\gamma \in (1,2)$, we define a norm $||\cdot||_{k, \beta, \gamma}$ on $C^{k,\beta}(M)$ by taking $||f||_{k,\beta,\gamma}$ to be the supremum of the following semi-norms with $p \in V(\Gamma)$, $\pen \in S(\Gamma)$, and $[p,e,n'] \in N(\Gamma)$ defined in \ref{pendef}:
\begin{enumerate}[(i)]
\item $||f:C^{k,\beta}(S_1[p], \chi)||$
\item $\outau^{-\gamma n}||f: C^{k,\beta}(S_1[p,e,n],\chi)||$
\item $\outau^{-\gamma (n'-1)}||f:C^{k,\beta}(\Lambda[p,e,n'], \chi, e^{-\gamma \underline x})||$
\end{enumerate}
For $w \in \mathcal K$ such that
\[
w= \sum_{i=1}^3  \left(\sum_{p \in V(\Gamma)}\mu_i[p] w_i[p]+\sum_{\pe \in A(\Gamma)} \mu_{i}\pe w_i\pe +\sum_{\pen \in S(\Gamma)} \mu_i\pen w_i\pen \right),
\] let $w^X$ denote the projection of $w$ onto $\mathcal K_X$. We define norms for each of these projections separately to help highlight the differences in scale for the different coefficients. To that end, let
\addtocounter{theorem}{1}
\begin{equation}
||w^V||_V:= \max_{\substack{i=1,2,3\\p } } \{|\mu_i[p]|\}, \; \; ||w^A||_A:=\max_{\substack{i=1,2,3\\ [p,e] }} \{|\mu_i[p,e]|\}, ||w^S||_\gamma:= \sup_{\substack{i=1,2,3\\ [p,e,n]}} \{\outau^{-\gamma n}|\mu_i[p,e,n]|\}.
\end{equation} 
\end{definition}

We are finally ready to state and prove the main result of this section. The strategy of proof is to first use Proposition \ref{RLambda} to solve the problem on each transition region. We then apply the cutoff functions to the inhomogeneous term and use Lemmas \ref{linearpartpen}, \ref{linearpartp} to solve semi-locally on each extended standard region. Finally, we patch together these semi-local solutions via \ref{patching}. This process introduces error which we show is small enough to correct by iteration.
\addtocounter{equation}{1}
\begin{prop}\label{LinearSectionProp}
There is a linear map $\mathcal{R}_M: C^{0, \beta}(M) \to C^{2,\beta}(M)\times \mathcal{K}$ such that for $E \in C^{0,\beta}(M)$ and $(\varphi, w)=\mathcal{R}_M(E)$ the following hold:
\begin{enumerate}
\item $\mathcal{L}_\chi \varphi = E+ \cfw w$ on $M$.
\item $||\varphi||_{2,\beta,\gamma}+||w^V||_{V}+||w^A||_A + ||w^S||_\gamma \leq C(b, \beta,\gamma)||E||_{0, \beta, \gamma}$
\item $\mathcal{R}_M$ depends continuously on all of the parameters of the construction.
\end{enumerate}
\end{prop}

\begin{proof}
We first prove the proposition, assuming $\supp(E) \subset \left(\cup_{V(\Gamma)} S_1[p]\right) \bigcup \left( \cup_{S(\Gamma)} S_1\pen\right)$. Using Lemmas \ref{linearpartp} and \ref{linearpartpen} we determine pairs 
\[(\varphi\pen, \widetilde w\pen)=\mathcal R_{\widetilde S\pen}(E|_{S_1\pen}), \qquad (\varphi[p], \widetilde w[p]) = \mathcal R_{\widetilde S[p]}(E|_{S_1[p]})\]
for each $\pen \in S(\Gamma), p \in V(\Gamma)$. (Note that $\widetilde w[p]$ will be a linear combination of elements of both $\mathcal K[p]$ and $\mathcal K\pe$.)  In order to patch together the local solutions $\varphi[p], \varphi\pen$ and maintain the regularity,
we have to cutoff each of these functions near the boundary. But this introduces some error.
We denote
\[
\mathcal{R}E := \mathbf U(\{\psi_{\widetilde S[p]}\varphi[p], \psi_{\widetilde S\pen}\varphi\pen\})\in C^{2,\beta}(M),\]
\[ \mathcal{W}E := \left(\sum_{V(\Gamma)} \widetilde w[p] + \sum_{S(\Gamma)} \widetilde w\pen\right)\in \mathcal K,
\]
\[
\mathcal{E}E:=\mathbf U(\{[\psi_{\widetilde S[p]},\mathcal L_\chi]\varphi[p], [\psi_{\widetilde S\pen},\mathcal L_\chi]\varphi\pen\})\in C^{0,\beta}(M),
\]where here $[\:,\:]$ denotes the commutator. That is, $[\psi_{\widetilde S[p]},\mathcal L_\chi]\varphi[p]= \psi_{\widetilde S[p]}\mathcal L_ \chi \varphi[p] -\mathcal L_\chi (\psi_{\widetilde S[p]}\varphi[p])$ and the like for $\varphi\pen$.

One easily checks that, as the support of $E$ implies $\psi_{\widetilde S[p]} \mathcal L_\chi \varphi[p]= \mathcal L_\chi \varphi[p]$ and the like for $\varphi\pen$, 
\addtocounter{theorem}{1}
\begin{equation}
\mathcal L_\chi \mathcal{R}E + \mathcal{E}E = E+\cfw \mathcal{W}E \qquad \text{on } M.
\end{equation}Moreover, for $\gamma' \in (\gamma, 2)$, we get the estimates
\begin{align*}
||\mathcal RE||_{2,\beta,\gamma} &\leq C(b,\beta,\gamma)||E||_{0,\beta,\gamma}\\
||\mathcal WE^V||_{V}+ ||\mathcal WE^A||_{A}+||\mathcal WE^S||_{\gamma}  &\leq C(b,\beta,\gamma)||E||_{0,\beta,\gamma}\\
||\mathcal EE||_{0,\beta,\gamma}&\leq \outau^{(\gamma'-\gamma)}C(b,\beta,\gamma, \gamma')||E||_{0,\beta,\gamma}.
\end{align*}
The first two estimates are immediate from Lemmas \ref{linearpartp}, \ref{linearpartpen} and the uniform geometry of the immersions on the support of $\mathcal WE$. To determine the third, we first observe that applying Lemma \ref{linearpartpen}
to $\mathcal EE$ for $\gamma' \in (\gamma,2)$ yields the estimate
\[
 ||\mathcal EE||_{0,\beta,\gamma} \leq e^{(\gamma-\gamma')\underline x}C(b,\beta,\gamma)||E||_{0,\beta,\gamma}.
\]Next observe $\supp(\mathcal EE)\subset
 \left(\cup_{p \in V(\Gamma} S_1[p] \backslash S[p]\right)\bigcup \left(\cup_{\pen \in S(\Gamma)} S_1[p,e,n] \backslash S\pen\right)$. 
Finally, for a fixed $e\in E(\Gamma) \cup R(\Gamma)$, on the supported regions, by \eqref{elllength}, we have the estimate
\[
 \exp({(\gamma- \gamma')\underline x}) \leq C(\gamma, \gamma', b)|\taue|^{(\gamma'-\gamma)}\leq  C(\gamma, \gamma', b)\outau^{(\gamma'-\gamma)}.
\]If $\outau$ is sufficiently small, we can guarantee $ C(\gamma, \gamma', b)\outau^{(\gamma'-\gamma)}<1/2$. This allows us to remove the error by iteration.

Now observe that since $\mathcal EE \in C^{0,\beta}\left(\cup_{V(\Gamma)} S_1[p]\right) \bigcup \left( \cup_{S(\Gamma)} S_1\pen\right)$, we can apply the steps above to $\mathcal EE$ to produce $\mathcal R^1E, \mathcal W^1E, \mathcal E^1E$. We then proceed by induction to produce, for all $r \in \mathbb Z^+$
\addtocounter{theorem}{1}
\begin{equation}
\mathcal L_\chi \mathcal R^rE + \mathcal E^{r+1}E = \mathcal E^{r}E +\cfw \mathcal W^rE \qquad \text{on } M.
\end{equation}
Defining
\addtocounter{theorem}{1}
\begin{equation}
\varphi := \sum_{r=0}^\infty \mathcal R^rE; \qquad w:=\sum_{r=0}^\infty \mathcal W^rE,
\end{equation}which converge based on the provided estimates, we prove the proposition in the first case.

Now consider the more general case. We begin by applying Proposition \ref{RLambda} to each $\Lambda\penp$ to determine
\[
 V\penp = \mathcal R_{\Lambda\penp}E|_{\Lambda\penp}
\]such that $V\penp|_{C^+[p,e,n'-1]} \in \mathcal H_1[C^+[p,e,n'-1]]$ and $V\penp|_{C^-\penp}=0$. We define
\[
 \widetilde E:= \mathbf U(\{\psi_{S[p]}E, \psi_{S\pen}E\}) + \mathbf U(\{0,[\psi_{\Lambda\penp},\mathcal L_\chi]V\penp\})\in C^{0,\beta}(M),
\]and observe that $\widetilde E \subset \left(\cup_{V(\Gamma)} S_1[p]\right) \bigcup \left( \cup_{S(\Gamma)} S_1\pen\right)$. We now apply this case
to the work done above to produce $(\widetilde \varphi, w)\in C^{2,\beta}(M) \times \mathcal K$ such that $\mathcal L_\chi \widetilde \varphi = \widetilde E + \cfw w$ on $M$.
We define $\varphi \in C^{2, \beta}(M)$ such that
\[
 \varphi = \widetilde \varphi + \mathbf U(\{0,\psi_{\Lambda\penp}V\penp\}).
\]Recalling that $\psi_{S[p]},\psi_{S\pen}, \psi_{\Lambda\penp}$ partition $M$, we see $\mathcal L_\chi \varphi = E+\cfw w$. Moreover, $(\varphi,w)$ satisfy
the necessary estimates so the proof is complete.
\end{proof}

\section{The Geometric Principle and proof of the Main Theorem}\label{GeoP}
Throughout this section, let $\gamma \in (1,2), \beta \in (0,1)$, $\gamma' \in (\gamma, 2)$, and $\beta' \in (0,\beta)$. Any constant depending only on $b, \beta, \gamma, \beta', \gamma'$ we will simply denote by $C$. If it depends on other constants, they will be made explicit in the notation.

\subsection*{Prescribing the extended substitute kernel on a non-central standard region}
As in the paper \cite{HaskKap}, prescribing the extended substitute kernel on the non-central standard regions can be done entirely on the linear level.
The main tool of the argument is Green's second identity \cite{GiTr}, which is a linearization of the force balancing.
Before we proceed, recall that $\cf_i$, defined in \eqref{kernelcomparison}, are simply translation vector fields that have been projected onto the normal bundle of $\Mtdz$ (multiplied by some constant scale factor). Since translations preserve the mean curvature, $\mathcal L \cf_i \equiv 0$ for $i=1,2,3$.

We begin by defining, for $i=1,2,3$, $\phi_i'\pen \in C^{2,\beta}(\widetilde S\pen)$ where
\addtocounter{theorem}{1}
\begin{equation}\label{phipen}
\begin{aligned}
&\phi_{1}'\pen = \psi[0,1]\circ \overline x V_1[\Lambda_{close},\widehat c_{1}\pen\ell_{\Lambda_{close}},0],\\
&\phi_{2}'\pen = \psi[0,1]\circ \overline x V_2[\Lambda_{close}, \widehat c_{2}\pen \sinh \ell_{\Lambda_{close}}, 0],\\
&\phi_{3}'\pen = \psi[0,1]\circ \overline x V_3[\Lambda_{close}, \widehat c_{3}\pen \sinh \ell_{\Lambda_{close}}, 0].
\end{aligned}
\end{equation}
We choose $\widehat c_{i}\pen$, $i=1,2,3$, such that -- recalling \eqref{wpendef}--
\addtocounter{theorem}{1}
\begin{equation}\label{precichoice}
\int_{\widetilde C_{close}}\frac{\partial \phi_i'\pen}{\partial t} \cf_j - \frac{\partial \cf_j}{\partial t} \phi_i'\pen d\theta = -\delta_{ij}
\end{equation}where here $\widetilde C_{close}$ represents the middle meridian circle on $\Lambda_{close}$. (Here $\Lambda_{close}, \Lambda_{far}$ are as in Definition \ref{genericlabel}.) For $\taue>0$,  \eqref{weq},\eqref{normalvector} imply that $\partial \cf_j/\partial t \equiv 0$ on $\widetilde C_{close}$ for $j=1,2,3$. For $\taue<0$, the same equations imply that $\cf_2, \cf_3 \equiv 0$, $\partial \cf_1/\partial t \equiv 0$ on this meridian. 

For the rest of the argument we drop the dependence on $\pen \in S(\Gamma)$ and presume the reader understands the implied dependence. Where there might be confusion, we state it explicitly.

Notice that $\supp(\phi_i') \subset \Lambda_{close}$ and $\supp(\mathcal L_\chi \phi_i' )\subset S_1 \cap \Lambda_{close}$.  Proposition \ref{LambdaKernel}, \eqref{Vlambda},  \eqref{elllength}, and the definition of $\cf_i$ imply 
\[
C^{-1} < \widehat c_i < C,\: i=1,2,3
\]
and
\addtocounter{theorem}{1}
\begin{equation}\label{phiprime}
\begin{aligned}
&||\phi_i':C^{2,\beta}(\Lambda_{close}, \chi, e^{-\underline x})|| \leq Ce^\ell\leq C\outau^{-1}
\\
&||\mathcal L_\chi \phi_i':C^{0, \beta}(\widetilde S, \chi)|| \leq C.
\end{aligned}
\end{equation}Observing $\supp(\mathcal L_\chi \phi_i') \subset S_1 \cap \Lambda_{close}$, 
let $(\phi_i'', \underline w_i) = \mathcal R_{\widetilde S}(-\mathcal L_\chi \phi_i')$ and finally, define
\[
\phi_i := \phi_i'' + \phi_i'.
\]
\addtocounter{equation}{1}
\begin{lemma}\label{prescribelinear}
Given $\phi_i, \underline w_i$ as defined above, for $\widetilde S:= \widetilde S\pen$,
\begin{enumerate}
\item \label{pl1}$\mathcal L_\chi \phi_i =\cfw \underline w_i$ on $\widetilde S$. 
\item \label{pl3}$||\phi_i:C^{2,\beta}(\widetilde S \backslash \Lambda_{close}, \chi)|| \leq C$.
\item \label{pl4}$||\phi_i:C^{2, \beta}(\Lambda_{close}, \chi, e^{-\underline x})|| \leq C \outau^{-1}$.
\item \label{pl5}$||\phi_i: C^{2,\beta}(\Lambda_{far},\chi,e^{-\gamma' \underline x})||\leq C$ when $n<l(e)$.
\item \label{pl5b}$||\phi_i: C^{2,\beta}(\Lambda_{far},\chi,\frac{\underline x +1}{\ell_\Lambda})||\leq C$ when $n=l(e)$. 
\item \label{pl6}$|\underline w_i - w_i| \leq C/|\log \outau|$ where $w_i$ represents the appropriate $w_i\pen$.
\item \label{pl7}$\phi_i, \underline w_i$ are unique by their construction and depend continuously upon all of the parameters of the construction.
\end{enumerate}
\end{lemma}
\begin{proof}
Items \eqref{pl1} and \eqref{pl7} follow immediately from the definition.  Items \eqref{pl3} and \eqref{pl4} follow from \eqref{phiprime}, \eqref{elllength} and \eqref{pl5}, \eqref{pl5b} follow from the definition and Lemma \ref{linearpartpen}. 

For \eqref{pl6}, we require the use of Green's Theorem. For a domain $\Omega \subset \widetilde S$, we have
\[
\int_\Omega \cf_i \mathcal L_\chi \phi_j - \phi_j \mathcal L_\chi \cf_i \: d\chi = \int_{\partial \Omega} \cf_i \left( \frac{\partial \phi_j}{\partial \eta}\right) - \phi_j \left( \frac{\partial \cf_i}{\partial \eta}\right),
\]where here $\eta$ is the outward pointing conormal to $\partial \Omega$. Let $\Omega$ be the connected component of $\widetilde S \backslash \left(\widetilde C_{close} \cup \widetilde C_{far}\right)$ where here $\widetilde C_{far}$ is the middle meridian circle on $\Lambda_{far}$. 
As translations preserve the mean curvature equation, we know $\mathcal L_\chi \cf_i \equiv 0$. Further, using $\mathcal L_\chi \phi_i =\cfw \underline w_i$, the above equation reduces to
\[
\int_\Omega \cf_i \underline w_j \: \cfw d\chi= \int_\Omega \cf_i \underline w_j \: dh= \int_{\widetilde C_{far}}\left( \cf_i  \frac{\partial \phi_j}{\partial t}- \frac{\partial \cf_i}{\partial t} \phi_j\right) - \int_{\widetilde C_{close}}\left(\cf_i  \frac{\partial \phi_j}{\partial t}- \frac{\partial \cf_i}{\partial t} \phi_j\right).
\]Using \eqref{precichoice}, for $\taue>0$, since $\partial \cf_i/\partial t$ vanishes on each boundary curve it is enough to show appropriate control on $\frac{\partial \phi_j''}{\partial t}$. Recalling the construction of $\phi_j''$, observe that $\mathcal L_\chi \phi_j'' \equiv 0$ on $\Lambda_{far}$ and $\Lambda_{close} \backslash \left(S_1 \cap \Lambda_{close}\right)$. If we denote $\partial S_1 \cap \Lambda_{close}:= C^-_{close,1}$ then  $||\phi_j'':C^{2,\beta}(C_{far}^+\cup C^-_{close,1}, d\theta)|| \leq C$. Item \eqref{LCC2} in Corollary \ref{LambdaCloseCor} then provides the necessary estimate.

For $\taue<0$, the same argument works for $\cf_1$. For $\cf_2,\cf_3$, we observe that on $\widetilde C_{close}, \widetilde C_{far}$, $|\partial \cf_2/\partial t|, |\partial \cf_3/\partial t|\leq C|\taue|^{1/2}$. Since $\cf_2=\cf_3\equiv 0$ on the boundary curves, using the estimate for $\phi_j''$ on $C_{far}^+ \cup C^-_{close,1}$ coupled now with \eqref{LCC1} in Corollary \ref{LambdaCloseCor} is strong enough to provide \eqref{pl6}.
\end{proof}

\subsection*{Prescribing substitute kernel on the central spheres}The prescription of the extended substitute kernel on each central sphere is completed in two steps. We first prescribe substitute kernel based on the unbalancing parameter $d$. We then prescribe the extended substitute kernel at each attachment based on the dislocation parameter $\boldsymbol \zeta$. Part of the difficulty in what follows comes from the fact that the unbalancing is understood by considering the graph $\Gd$ and an immersion based on that structure while the dislocation is understood by considering dislocations from the graph $\Gudl$, as the dislocations are introduced on an immersion positioned relative to that graph. 

We state here an important lemma without proof. The lemma depends upon a flux argument \cite{Kus,KKS}; the proof for our specific setup can be found in Appendix A.3 of \cite{KapAnn}.
We state the lemma in the form exactly as we need for our purposes. Throughout this subsection, for any $p \in V(\Gamma)$, let $S^+[p]$ denote the connected component of $\widetilde S[p] \backslash \left(\cup_{e \in E_p} \widetilde C[p, e, 1]\right)$. 
\addtocounter{equation}{1}
\begin{lemma}\label{unbalancinglemma}
For $d,\boldsymbol \zeta$, $\hYtdz$ and the function $H_{gluing}$ induced from this immersion,
\[
\int_{S^+[p]} H_{gluing}[p]\:\Ntdz \: dg= \pi\sum_{e \in E_p} \taue\sep \RRR'[e]\Be_1 := \pi d_{\boldsymbol \zeta}(p).
\]Here $p'\in V(\Gudl)$ is the vertex corresponding to $p$ and $e'$ corresponds to $e$.
\end{lemma}
Notice we do not expect that $d$ equals $d_{\boldsymbol \zeta}$. However, the difference between the two is bounded by higher powers of $\outau$.

\addtocounter{equation}{1}
\begin{prop}\label{subsprescribep}
Given $d,\boldsymbol \zeta$, let $H_{gluing}$ denote the gluing function induced from the immersion $\hYtdz$. For each $p \in V(\Gamma)$, there exist $\phi_{gluing}[p]\in C^{2,\beta}(\widetilde S[p])$, $\mu_i'[p], \mu_i'[p,e_j]$, $i=1,2,3,j=1,\dots, |E_p|$ such that 
\begin{enumerate}
\item $\mathcal L_\chi \phi_{gluing}[p] +\rho^{-2}H_{gluing}= \sum \mu_i'[p] w_i[p] + \sum \mu_i'[p,e_j]w_i[p,e_j]$ on $\widetilde S[p]$.
\item $\phi_{gluing}[p] = 0$ on $\widetilde S[p]$.
\item $ ||\phi_{gluing}[p]: C^{2, \beta}(\widetilde S[p], \chi)|| \leq C\outau$.
\item $ ||\phi_{gluing}[p]:C^{2,\beta}(\Lambda[p,e,1], \chi, e^{-\gamma \underline x})|| \leq C\tau$ for all $e \in E_p$. 
\item \label{dsep}$\left| \sum_{i=1}^3 \mu_i'[p] \Be_i - d(p)\right| \leq \varepsilon \outau/16$.
\item $|\mu_i'[p,e_j]| \leq C\outau$.
\item Each of the $\phi_{gluing}, \mu_i'[p], \mu_i'[p,e_j]$ are all unique by their construction and depend continuously on all of the parameters of the construction.
\end{enumerate}
\end{prop}

\begin{proof}
Determine $\mu_i'$ such that for $j=1,2,3$, 
\[
\int_{\widetilde S[p]} \left(\sum_{i=1}^3 \mu_i' w_i[p] - \frac{2}{|A|^2}H_{gluing}\right) f_j[p] \: dh=0.
\]Rather than investigating this integral, we initially consider
\[
\int_{\widetilde S[p]} \left(\sum_{i=1}^3 \mu_i' w_i[p] - \frac{2}{|A|^2}H_{gluing}\right) \Ntdz\: dh,
\]recalling $\Ntdz \cdot \Be_i = \pi \hat f_i$.
Based on the definition of $w_i[p]$, 
\[
\int_{\widetilde S[p]} \left(\sum_{i=1}^3 \mu_i' w_i[p] \right) \Ntdz \: dh = \pi\sum_{i=1}^3 \mu_i' \Be_i.
\]An application of Lemma \ref{unbalancinglemma} implies
\[
\sum_{i=1}^3 \mu_i' \Be_i- d(p)=-\int_{S^+[p]} \frac{1}{\pi}\left(\sum_{i=1}^3 \mu_i' w_i[p] - \frac{2}{|A|^2}H_{gluing}\right)\left(f_j[p] - \hat f_j\right)\:dh+d_{\boldsymbol \zeta}(p)-d(p)
\]where $d_{\boldsymbol \zeta}(p)$ is defined in Lemma \ref{unbalancinglemma}.
By Proposition \ref{approxkerprop}, $||f_i[p] - \hat f_i||_{C^{k}(S_5[p],h)} \leq \epsilon$, where $\epsilon$ can be made as small as we like by making $\outau$ small. Coupled with Proposition \ref{Hestimates} and Lemma \ref{extsubslemma}, we estimate
\[
 \left|\int_{S^+[p]} \left(\sum_{i=1}^3 \mu_i' w_i[p] - \frac{2}{|A|^2}H_{gluing}\right)\left(f_j[p] - \hat f_j\right)\:dh \right|\leq C\epsilon\outau <\varepsilon \outau /32.
\]Notice $C$ is independent of $\outau$ so we can easily achieve the inequality $C\epsilon < \varepsilon/32$.
Observe also that
\[
d_{\boldsymbol \zeta}(p) -d_{\Gudl}(p) + d_{\Gudl}(p)- d(p)= \sum_{e\in E_p} \taue \sep \left(\RRR'[e]\Be_1-\RRR'[e]\Be_1[e] + \RRR'[e]\Be_1[e] - \Bv_1[e;d,0]\right).
\]
Using \eqref{ddifftau} and Proposition \ref{zetaframe} we see
\[
|d_{\boldsymbol \zeta}(p) - d(p)|\leq C \outau^{3/2} + C \underline C \outau^2
\] 
and item \eqref{dsep} follows.

The rest of the estimates follow immediately from Lemma \ref{linearpartp} and the bound on $\rho^{-2}H_{gluing}$.
\end{proof}

\subsection*{Prescribing the extended substitute kernel on the central spheres}
We are now ready to describe the prescription of elements of $\mathcal K\pe$ on each extended central standard region. The proof relies on the estimates of Appendix \ref{quadapp} and uses the functions defined in Appendix \ref{appendixb}. Given the parameters $d,\boldsymbol \zeta$, we modify the inhomogeneous term of the linearized equation by elements of $\oplus_{e \in E_p} \mathcal K\pe$ such that the coefficients of these elements are close to the parameters $\boldsymbol \zeta \pe$. Notice that when we solve for $\rho^{-2}H_{dislocation}$, the elements of the substitute kernel have coefficients bounded by $C ||\boldsymbol \zeta||/|\log \tau|$ which reflects the fact that $\boldsymbol \zeta \pe$ is essentially orthogonal to the approximate kernel. 

\addtocounter{equation}{1}
\begin{definition}\label{wfdef}We normalize the functions $\cf_i$ defined in \eqref{cfeq} on the meridian circle $C_1^+[p,e,0]$. Precisely, for $\pe \in A(\Gamma)$, choose $c'_i[p,e]$ such that on $C_1^+[p,e,0]$,
\[
c'_1[p,e]\pi \cf_1 =1; \: c_2'[p,e]\pi\cf_2 = \cos \theta; \: c_3'[p,e]\pi\cf_3 = \sin \theta.
\]
 Let 
\addtocounter{theorem}{1}
\begin{equation}\label{cprimedef}
c':= \max_{\substack{i=1,2,3,\\ [p,e] \in A(\Gamma)}} |c'_i[p,e]|.
\end{equation}
\end{definition}The normalization implies that on $C_1^+[p,e,0]$ the function that describes $\hYtdz$ as a graph over $Y_0 + \boldsymbol \zeta\pe$ is approximately $\zeta_1\pe/c_1'[p,e] + (\zeta_2\pe/c_2'[p,e])\cos \theta +(\zeta_3\pe/c_3'[p,e])\sin \theta$. (See Appendix \ref{appendixb}.) Observe that as $\outau \to 0$, $c'$ converges to a value that depend upon the geometric limit (in this case $\Ss^2$). That is, $c'$ depends on $b$ and is independent of $\outau$.

\addtocounter{equation}{1}
\begin{prop}\label{prescribequad}
For each $p \in V(\Gamma)$, 
there exists $\phi_{dislocation}[p] \in C^{2, \beta}(\widetilde S[p])$ and $\mu_{i}''[p], \mu_{i}''[p,e_j]$ with $i=1,2,3$ and $j=1, \dots, |E_p|$ such that, for parameters $d,\boldsymbol \zeta$, the immersion $\hYtdz$, and the induced function $H_{dislocation}$,
\begin{enumerate}
\item $\mathcal L_\chi \phi_{dislocation}[p]+ \rho^{-2}H_{dislocation}= \sum_{i=1}^3 \left(\mu_{i}''[p]w_i[p] +  \sum_{j=1}^{|E_p|}\mu_{i}''[p,e_j]w_i[p,e_j]\right)$ on $\widetilde S[p]$.
\item $\phi_{dislocation}[p] = 0$ on $\partial \widetilde S[p]$.
\item $|\mu_i''[p]| + | \zeta_i[p,e_j]/c_i'[p,e_j]-\mu_{i}''[p,e_j]|\leq C||\boldsymbol \zeta||/|\log \outau|$ for $i=1,2,3$ and each $e_j \in E_p$.
\item $|| \phi_{dislocation}[p]:C^{2,\beta}(\widetilde S[p], \chi)||\leq C||\boldsymbol \zeta||$.
\item $||\phi_{dislocation}[p]:C^{2,\beta}(\Lambda[p, e_j, 1], \chi, e^{-\gamma' \underline x})||\leq C||\boldsymbol \zeta||/|\log \outau|$  for each $e_j \in E_p$.
\item $\phi_{dislocation}[p], \mu_i''[p], \mu_{i}''[p,e_j]$ are all unique by their construction and depend continuously on the parameters of the construction.

\end{enumerate}
\end{prop}
\begin{proof}Throughout the proof, $i \in \{1,2,3\}, j \in \{1, \dots, |E_p|\}$. 
Let $\underline f[p]$ represent the function from Definition \ref{underlinef}. We construct $\uphipp \in C^{2, \beta} (\widetilde S[p])$ in the following way. On $M[p]$, let $\uphipp=\underline f[p]$. On $M[e_j] \cap S[p]$, for $p=p^+[e_j]$ set
$\check \psi_j:=\psi[a+3,a+2](t)$ and for $p=p^-[e_j]$ set $\check \psi_j:= \psi[\RHj-(a+3), \RHj-(a+2)](t)$. On $M[e_j] \cap S[p]$, let $\uphipp = \check \psi_j \underline f[p]- (1-\check \psi_j)\pi \sum_{i=1}^3 \zeta_i[p,e_j] \cf_i$.  Finally, on each $\Lambda[p,e_j,1]$, find $\underline V_j$ such that $\mathcal L_\chi \underline V_j =0$, $\underline V_j = - \pi\sum_{i=1}^3 \zeta_i[p,e_j] \cf_i$ on $C^-[p,e_j,1]$ and $\underline V_j = 0$ on $C^+[p,e_j,0]$. On $\Lambda[p,e_j,1]$, let $\uphipp =-\pi\sum_{i=1}^3 \zeta_i[p,e_j] \cf_i+ (1-\psi_{S[p]})\underline V_j$.
The construction of $\underline V_j$ implies

\[||\underline V_j:C^{2,\beta}([b,b+2] \times \Ss^1 \cap M[e_j],\chi)||\leq C||\boldsymbol \zeta||/|\log \outau|.\]

Noting the estimates provided by Corollary \ref{corundf}, we have the following estimates for $\uphipp$:
\begin{enumerate}
\item $\mathcal L_\chi \uphipp$ is supported on $S_1[p]$ and $\uphipp=0$ on $\partial \widetilde S[p]$.
\item For each $j$, $||\mathcal L_\chi \uphipp - \rho^{-2}H_{dislocation}:C^{0,\beta}(S_1[p] \cap M[e_j], \chi)|| \leq C||\boldsymbol \zeta|/|\log \outau|$.

\noindent Recall the argument used in Lemma \ref{linearpartpen} to determine the coefficients $\mu_i$. In this case we determine $\underline\mu_{i}[p,e_j]$  by noting that on $C_1^+[p, e_j,0]$ we must have 
\[
-\sum_i \underline \mu_{i}[p,e_j]v_i[p, e_j] = \underline V_j - \pi\sum \zeta_i[p,e_j] \cf_i - \mathcal R_\partial (\underline V_{j,\perp}).
\]This is a simpler expression than the one in Lemma \ref{linearpartpen} because $v_i, \cf_i \in \mathcal H_1[C_1^+]$. Then Corollary \ref{Rpartial}, items \eqref{rp3} and \eqref{rp4} coupled with the estimate on $\underline V_j$ imply 
\[|\zeta_i[p,e_j] / c_i'[p,e_j]-\underline\mu_{i}[p,e_j] | \leq C ||\boldsymbol \zeta||/|\log \outau|, \text{ and}
\]
\[
||\uphipp+ \sum_{i=1}^3 \underline \mu_{i}[p,e_j] v[p,e_j]:C^{2,\beta}(\Lambda[p,e_j,1], \chi, e^{-\gamma'\underline x})|| \leq C ||\boldsymbol \zeta||/|\log \outau|.
\] Using Lemma \ref{linearpartp} with $E: = \mathcal L_\chi \uphipp -\rho^{-2} H_{dislocation}$, let $$(\uphip, w_{dislocation}):= \mathcal R_{\widetilde S[p]}(E)$$ where
\[
w_{dislocation} = \sum_i \mu_i''[p] w_i[p] + \sum_{i,j} \mu_{i}'''[p,e_j]w_i[p,e_j].
\]
 Then:
\item $\mathcal L_\chi \uphip= \mathcal L_\chi \uphipp - \rho^{-2}H_{dislocation}+ w_{dislocation}$ on $\widetilde S[p]$ 
\item $|\mu_i''[p]|, |\mu_{i}'''[p,e_j]|\leq C||\boldsymbol \zeta||/|\log \outau|$.
\end{enumerate}

Set
\[
\phi_{dislocation}[p] = \uphip - \uphipp - \sum_{i}\underline \mu_{i}[p,e_j]v[p,e_j]
\]and
\[
\mu_{i}''[p,e_j] = \mu_{i}'''[p,e_j]-\underline\mu_{i}[p,e_j].
\]We complete the proof by appealing to all of the estimates above and those of Lemma \ref{linearpartp}.
\end{proof}

\subsection*{Prescribing the extended substitute kernel globally}
Let $\Real^{V}, \Real^A$ denote the finite dimensional vector spaces with components on $\Real$ and indexed over $i=1,2,3$ and $p \in V(\Gamma)$ or $i=1,2,3$ and $\pe \in A(\Gamma)$. Let $\Real^{S}$ denote the infinite dimensional vector space indexed over $i=1,2,3$ and $\pen \in S(\Gamma)$. Let $\boldsymbol \xi \in \Real^V \times \Real^A \times \Real^S$ and let $\boldsymbol \xi^X$ denote its projection onto $\Real^X$. 
We define norms such that 
\[
||\boldsymbol \xi^V||_V:=\max_{i,p}\{|\xi_i[p]|\}, \; \; ||\boldsymbol \xi^A||_A:=\max_{i,\pe}\{|\xi_i\pe|\}, \; \; ||\boldsymbol \xi^S||_\gamma:=\sup_{i,\pen}\{\outau^{-\gamma n}|\xi_i\pen|\}.
\]
We proceed by assuming throughout
\addtocounter{theorem}{1}
\begin{equation}
||\boldsymbol \xi^V||_V \leq \varepsilon \outau/\sqrt 3, \;\; ||\boldsymbol \xi^A||_A \leq \underline C \outau/c', \; \; ||\boldsymbol \xi^S||_\gamma \leq \outau.
\end{equation}

Now for $p \in V(\Gamma)$ and $\pe \in A(\Gamma)$, set
\addtocounter{theorem}{1}
\begin{equation}\label{zetaxi}
d(p) = \sum_i \xi_i[p] \Be_i   \: \:
\text{and}\: \:
 \zeta_i\pe = c'_i[p,e] \xi_i[p,e] .
\end{equation}Notice the scaling for the norms for the spaces $\Real^V, \Real^{A}$ implies that $||d||_D \leq \varepsilon \outau, ||\boldsymbol \zeta|| \leq \underline C  \outau$. With this definition of $d, \boldsymbol \zeta$ we find $\phi_{gluing}[p],\phi_{dislocation}[p]$ from Propositions \ref{subsprescribep}, \ref{prescribequad} and $\phi_i\pen$ from Lemma \ref{prescribelinear} and set
\[
\Phi'_{\boldsymbol \xi} = {\bf U}\left( \left\{\psi_{\widetilde S[p]} \cdot (\phi_{gluing}[p]+ \phi_{dislocation}[p]), \psi_{\widetilde S\pen} \cdot\sum_{i=1}^3 \xi_i\pen \phi_i\pen \right\}\right).
\] Setting $\mu_i[p]:=\mu_i'[p]+\mu_i''[p], \mu_i\pe:=\mu_{i}'\pe+\mu_i''\pe$ where these coefficients come from Propositions \ref{subsprescribep}, \ref{prescribequad}, define
\[
\underline w':=\sum_{i=1}^3\left( \sum_{p \in V(\Gamma)}\mu_i[p]w_i[p] + 
\sum_{\pe \in A(\Gamma)} \mu_i\pe w_i\pe+\sum_{\pen \in S(\Gamma)}\xi_i[p,e,n]\underline w_i[p,e,n]  \right) \in \mathcal K.
\]
Here $\underline w_i\pen$ comes from Lemma \ref{prescribelinear}. 
Using Proposition \ref{LinearSectionProp}, determine
 $$(\Phi_{\boldsymbol \xi}'' , \underline w''):=\mathcal R_M\left(-\mathcal L_\chi\Phi_{\boldsymbol \xi}'+\cfw\underline w' - \rho^{-2}H_{error}\right).$$  Now set 
\[
\Phi_{\boldsymbol \xi}:= \Phi_{\boldsymbol \xi}'' + \Phi_{\boldsymbol \xi}' \: \text{ and } \underline w_{\boldsymbol \xi}:= \underline w'' + \underline w'.
\]

\addtocounter{equation}{1}
\begin{prop}\label{prescribexi}
$\Phi_{\boldsymbol \xi}$ and $ \underline w_{\boldsymbol \xi}$ as defined above for the immersion $\hYtdz$ where $d,\boldsymbol \zeta$ are defined by the components chosen in \eqref{zetaxi}, depend continuously on $\boldsymbol \xi$ and satisfy:
\begin{enumerate}
\item $\mathcal L_\chi \Phi_{\boldsymbol \xi}  + \rho^{-2}H_{error}= \cfw \underline w_{\boldsymbol \xi}$ on $M$.
\item $||\Phi_{\boldsymbol \xi}||_{2,\beta,\gamma} \leq C(  ||\boldsymbol \xi^S||_{\gamma}+ \outau + ||\boldsymbol \zeta||)$.
\item $||(\underline w_{\boldsymbol \xi}- w_{\boldsymbol \xi})^V||_{V} \leq  \varepsilon \outau/ 8$,\\
$||(\underline w_{\boldsymbol \xi}- w_{\boldsymbol \xi})^A||_{A} \leq C \outau/c',$\\
$||(\underline w_{\boldsymbol \xi}- w_{\boldsymbol \xi})^S||_{\gamma} \leq C ||\boldsymbol \xi^S||_\gamma/|\log \outau|$. 
Here \[w_{\boldsymbol \xi}:= \sum_{i=1}^3 \left( \sum_{p \in V(\Gamma)}\xi_i[p] w_i[p]+\sum_{\pe \in A(\Gamma)} \xi_i\pe w_i\pe+\sum_{\pen \in S(\Gamma)}\xi_i\pen w_i\pen\right).\]
\end{enumerate}
\end{prop}
\addtocounter{equation}{1}
\begin{remark}
Notice that the definitions of the norms implies $||\boldsymbol \xi^A ||_A = ||w_{\boldsymbol \xi}^A||_A$, $||\boldsymbol \xi^V ||_V = ||w_{\boldsymbol \xi}^V||_V$, and $||\boldsymbol \xi^S ||_\gamma= ||w_{\boldsymbol \xi}^S||_\gamma$.
\end{remark}
\begin{proof}
From the definitions and previously determined estimates, we observe $\mathcal L_\chi \Phi_{\boldsymbol \xi} + \rho^{-2}H_{error} = \cfw \underline w_{\boldsymbol \xi}$ on $M$ and $||\Phi'_{\boldsymbol \xi}||_{2, \beta, \gamma} \leq C(  ||\boldsymbol \xi^S||_{\gamma}+ \outau + ||\boldsymbol \zeta||)$. Moreover, as the function $-\mathcal L_\chi\Phi_{\boldsymbol \xi}'+\cfw \underline w' - \rho^{-2}H_{error}$ is supported on $\cup_{p \in V(\Gamma)} S_1[p] \backslash S[p] \bigcup \cup_{\pen \in S(\Gamma)} S_1\pen \backslash S\pen$, we can apply the same technique as in the proof of Proposition \ref{LinearSectionProp} to determine
\[
||-\mathcal L_\chi\Phi_{\boldsymbol \xi}'+\cfw \underline w' - \rho^{-2}H_{error}||_{0, \beta, \gamma} \leq C \outau^{\gamma'-\gamma}(  ||\boldsymbol \xi^S||_{\gamma}+ \outau + ||\boldsymbol \zeta||).
\]Finally, observe that a term by term comparison implies 
\begin{align*}||( \underline w'- w_{\boldsymbol \xi})^V||_{V} &\leq  \varepsilon \tau/12 \\
||( \underline w'- w_{\boldsymbol \xi})^A||_{A} &\leq C\outau/c' \\
||( \underline w'- w_{\boldsymbol \xi})^S||_{\gamma} &\leq C ||\boldsymbol \xi^S||_\gamma/|\log \outau|. 
\end{align*}
\end{proof}

\section{The Main Theorem}\label{MainTheorem}
We are now ready to state and prove the main theorem of the paper. The proof of the theorem depends upon all of the estimates in the paper, but the main estimates we need come from Propositions \ref{prescribexi} and \ref{LinearSectionProp} as well as the estimates in Appendices \ref{quadapp},\ref{appendixb}.

\addtocounter{equation}{1}
\begin{theorem}\label{maintheorem}Let $\Gamma$ be a flexible, central graph. There exists $\underline C>0$ sufficiently large and $\tau'>0$ sufficiently small such that for all $0<\outau <\tau'$:

There exist $\boldsymbol \zeta \in \Real^{A}$ and $d \in D(\Gamma)$ with $||\boldsymbol \zeta||\leq \underline C \outau, ||d||_D \leq \varepsilon \tau$, and $f \in C^{2,\beta}(M)$ such that $(\hYtdz)_f:M \to \Real^3$ is a well-defined immersed surface of constant mean curvature equal to 1 and $||f||_{2,\beta,\gamma} \leq \outau^{15/16}$. Moreover if $\Gamma$ is pre-embedded then $(\hYtdz)_f$ is embedded. Here $M$ is the abstract surface defined in \eqref{mdef} depending on $\Gamma$ and $\outau$, $\hYtdz$ is the initial immersion described in Section \ref{InitialSurface} depending on $\Gamma$ and the parameters $\tau, d, \boldsymbol \zeta$, and $(\hYtdz)_f$ is defined in Appendix \ref{quadapp}. 
\end{theorem}
\begin{proof}Choose $\underline C$ so that $\underline C > 2C$ for all $C$ appearing in the statement of Proposition \ref{prescribexi}. Choose $\tau'$ small enough such that the value $\max_{e \in E(\Gamma) \cup R(\Gamma)}|\tau' \hat \tau(e)|$ satisfies all geometric requirements throughout the paper. Notice that none of the geometric estimates $\tau'$ must satisfy depend upon the structure of a graph but only on the function $\hat \tau$. For any $0<\tau <\tau'$, let
\begin{align*}
B:= &\{ u \in C^{2, \beta}(M) : \: ||u||_{2,\beta,\gamma} \leq \outau^{3/2}\} \times\\& \{\boldsymbol \xi \in \Real^V \times \Real^{A}\times \Real^S :\:||\boldsymbol \xi^V||_V \leq \varepsilon \outau/\sqrt 3, ||\boldsymbol \xi^A||_A \leq \underline C \outau/c',  ||\boldsymbol \xi^S||_\gamma \leq \outau\}.
\end{align*}We define a map $\mathcal J:B \to B$ in the following manner. 
Let $(u,\boldsymbol \xi) \in B$. Recall $\boldsymbol \xi^V$ induces a $d \in D(\Gamma)$ and $\boldsymbol \xi^A$ induces a $\boldsymbol \zeta\in \Real^{A}$ as outlined in \eqref{zetaxi}. With this $d, \boldsymbol \zeta$, we find a $\Gudl \in \mathcal F(\Gamma)$ and determine the immersion $\hYtdz$.
Set $v:= \Phi_{\boldsymbol \xi} -u $. Then by Proposition \ref{prescribexi},
$\mathcal L_\chi v = \cfw \underline w_{\boldsymbol \xi} -\rho^{-2}H_{error}-\mathcal L_\chi u $ and
$||v||_{2, \beta, \gamma}\leq C(  ||\boldsymbol \xi^S||_{\gamma}+ \outau + ||\boldsymbol \zeta||)+ \outau^{3/2}  \leq \outau^{15/16}$.
Now apply Proposition \ref{LinearSectionProp}  to obtain $(u', w'):= \mathcal R_M(\rho^{-2}(H_v-\Htdz)-\mathcal L_\chi v)$ where here $H_v$ is the mean curvature of the surface induced by $(\hYtdz)_v$. Then
\begin{enumerate}
\item $\mathcal L_\chi u'=\rho^{-2}(H_v-\Htdz+H_{error}) +\cfw(w'-\underline w_{\boldsymbol \xi}) + \mathcal L_\chi u$

\noindent and by Proposition \ref{globalquad},
\item \label{uest}$||u'||_{2,\beta,\gamma}+||w'||_{\gamma} \leq C\outau^{15/8}\outau^{-1/8} \leq \outau^{3/2}$.

Define $\boldsymbol \mu \in \Real^V \times \Real^{A}\times \Real^S$ such that
\[
\sum \mu_i[p] w_i[p] + \sum \mu_i\pe w_i\pe+ \sum \mu_i\pen w_i\pen= w' + w_{\boldsymbol \xi}-\underline w_{\boldsymbol \xi}.
\]By the estimates of Proposition \ref{prescribexi} and the equivalence of the norms of $\boldsymbol \mu$ and $w_{\boldsymbol \mu}$,
\item \label{muest}$||\boldsymbol \mu^V||_V \leq \varepsilon \outau/\sqrt 3, ||\boldsymbol \mu^A||_A \leq \underline C \outau, ||\boldsymbol \mu^S||_{\gamma} \leq  \outau$.
\end{enumerate}
Define the map $\mathcal J(u, \boldsymbol \xi) = (u', \boldsymbol \mu)$. Then by \eqref{uest}, \eqref{muest} $\mathcal J(u, \boldsymbol \xi) \in B$ and the map $\mathcal J: B \to B$ is well defined. Moreover, for some $\beta' \in (0, \beta)$, $B$ is a compact, convex subset of $C^{2,\beta'}(M) \times (\Real^V \times \Real^A \times \Real^S)$ and one can easily check that $\mathcal J$ is continuous in the induced topology. Thus, Schauder's fixed point theorem, \cite{GiTr} Theorem 11.1, implies there exists a fixed point $(u', \boldsymbol \mu') \in B$. Let $f=\Phi_{\boldsymbol \mu'}-u'$.

Since $(u', \boldsymbol \mu')$ is a fixed point, the definition of $\boldsymbol \mu'$ implies $w' - \underline w_{\boldsymbol \xi} =0$. Moreover, $\mathcal L_\chi u = \mathcal L_\chi u'$. Thus
\[
H_f -\Htdz + H_{error} =0, \: \: \text{or} \:\: H_f = \Htdz - H_{error}\equiv 1
\] and the surface $(\hYtdz)_f$ has mean curvature identically 1.
Embeddedness follows when $\Gamma$ is pre-embedded as in this case $\Mtdz$ is embedded and $||f||_{2,\beta,\gamma} \leq \outau^{15/16}$.
\end{proof}

\appendix
\section{Quadratic estimates}\label{quadapp}
For completeness, we include here a proposition we will need, the proof of which appears in the appendices of \cite{KapYang,HaskKap}.
Let $X:D \to U$ be an immersion of the unit disk in $\Real^2$ into an open cube $U\subset \Real^3$ equipped with a metric $g$. Assume $\dist_g(X(D), \partial U)>1$ and there exists $c_1>0$ such that:
\addtocounter{theorem}{1}
\begin{equation}\label{quadconditions}
||\partial X: C^{2,\beta}(D,g_0)|| \leq c_1, \:\:\:||g_{ij}, g^{ij}:C^{4, \beta}(U,g_0)||\leq c_1, \:\:\:g_0 \leq c_1X^*g,
\end{equation}where here $\partial X$ represents the partial derivatives of the coordinates of $X$, $g^{ij}$ are the components of the inverse of the metric $g$, and $g_0$ denotes the standard Euclidean metric on $D$ or $U$ respectively. We note that \eqref{quadconditions} can be arranged by an appropriate magnification of the target, which we will exploit in order to make use of the following proposition. 

Let $\nu:D \to \Real^3$ be the unit normal for the immersion $X$ in the $g$ metric. Given a function $\phi:D \to \Real$ which is sufficiently small, we define $X_\phi:D \to U$ by
\addtocounter{theorem}{1}
\begin{equation}\label{xphi}
X_\phi(p):= \exp_{X(p)}(\phi(p)\nu(p))
\end{equation}where here $\exp$ is the exponential map with respect to the $g$ metric. Then the following holds:
\begin{prop}\label{quadboundsunscaled}
There exists a constant $\epsilon(c_1)>0$ such that if $X$ is an immersion satisfying \eqref{quadconditions} and the function $\phi:D \to \Real$ satisfies
\[
||\phi:C^{2,\beta}(D,g_0)|| \leq \epsilon(c_1)
\]then $X_\phi:D \to U$ is a well defined immersion by \eqref{xphi} and satisfies
\[
||X_\phi - X - \phi \nu:C^{1,\beta}(D,g_0)||\leq C(c_1)||\phi:C^{2,\beta}(D,g_0)||^2
\]and
\[
||H_\phi - H - \mathcal L_{X^*g} \phi:C^{0, \beta}(D,g_0)||\leq C(c_1)||\phi: C^{2, \beta}(D,g_0)||^2.
\]Here $H=tr_g A$ is the mean curvature of $X$, defined with respect to the metric $X^*g$ where $A$ is the second fundamental form, $H_\phi$ the mean curvature of $X_\phi$, and $\mathcal L_{X^*g}:= \Delta_g + |A|^2$.
\end{prop}
As mentioned, away from the central spheres, we need a scaled version in order to satisfy the hypotheses. We first prove a local estimate.
\begin{prop}\label{localquad}
Let $D \subset M$ be a disk of radius $1$ in the $\chi$ metric, centered at $x$. Let $\hYtdz:M\to \Real^3$ be the immersion from Definition \ref{MainImmersionDef}. If $v \in C^{2, \beta}(D,\chi)$ satisfies
\[
||v:C^{2,\beta}(D,\chi)||\leq \frac{ {\epsilon(c_1)}}{\rho(x)}
\]then
\[
||\rho^{-2}(H_v-\Htdz)-\mathcal L_\chi v:C^{0,\beta}(D,\chi)||\leq C(c_1){\rho(x)}||v:C^{2,\beta}(D,\chi)||^2.
\]
\end{prop}
\begin{proof}
To apply Proposition \ref{quadboundsunscaled}, we rescale the target space by $\rho(x)$. Thus, recalling item \eqref{nc1} from Lemma \ref{normcomparisons}, the conditions \eqref{quadconditions} are satisfied and the hypothesis on the norm is satisfied for the dilated function $\rho(x)v$.  Observe that $H_{\rho(x)v} = \frac{1}{\rho(x)}H_v$ and $\mathcal L_{(\rho(x)\hYtdz)^*(\rho^2(x)g)} = \rho(x)^{-2} \mathcal L_{(\hYtdz)^*g}$. Thus, the proposition implies
\[
||\rho^{-1}(x)\left(H_v - \Htdz- \mathcal L_{X^*g}v\right) :C^{0,\beta}(D,\chi)||\leq C(c_1)||\rho(x) v: C^{2, \beta}(D,\chi)||^2.
\] By the multiplicative property of the norms involved and again using item \eqref{nc1} of Lemma \ref{normcomparisons}, we observe
\[
||\rho^{-2}\left(H_v - \Htdz- \mathcal L_{X^*g}v\right) :C^{0,\beta}(D,\chi)||\leq C(c_1) {\rho(x)}|| v: C^{2, \beta}(D,\chi)||^2.
\]
\end{proof}
To apply the quadratic part of the proposition in the main theorem, we need a global form. While better estimates can be shown to hold, we state the proposition in a way that is sufficient for our use.
\begin{prop}\label{globalquad}Given $\alpha \in (0,1)$,
let $v \in C^{2,\beta}(M, \Real^3)$ such that $||v||_{2,\beta,\gamma} \leq \outau^{1-\alpha/14}$.
Then
\[
||\rho^{-2}\left(H_v - \Htdz\right)- \mathcal L_{\chi}v ||_{0,\beta,\gamma}\leq \outau^{\alpha-1}|| v||_{2,\beta,\gamma}^2.
\]
\end{prop}
\begin{proof}Notice that the definition of the global norm and of the function $\rho$ implies that the hypothesis of Proposition \ref{localquad} is satisfied for each $x \in M$. The inequality from Proposition \ref{localquad} implies that for any weighting $f$, 
\[
||\rho^{-2}\left(H_v - \Htdz- \mathcal L_{X^*g}v\right) :C^{0,\beta}(D,\chi,f)||\leq C(c_1){f(x)\rho(x)}|| v: C^{2, \beta}(D,\chi,f)||^2.
\]
We now recall that the global norm implies that for $S_1[p], S_1\pen$, we have $f\equiv 1$ and for $\Lambda \pen$, $f=e^{-\gamma \underline x}$. Moreover, $\rho$ is uniformly bounded above and below on each $S_1[p]$ and $S_1\pen$ for $n$ even. Elsewhere, $\rho(x)$ closely follows $r(x)^{-1}$ (up to some small reparameterization). Considering the various regions separately, we first observe that for $D\subset S_1[p]$,
\[
C(c_1) {f(x)\rho(x)} || v: C^{2, \beta}(D,\chi,f)||^2 \leq  C||v:C^{2,\beta}(D,\chi)||^2 \leq \outau^{\alpha-1} ||v||^2_{2,\beta, \gamma}
\]for any $0<\alpha<1$ and $\outau$ chosen small enough. The same inequality holds for $S_1\pen$ when $n$ is even. Now consider that for $n$ odd, $\rho(x) \leq C \outau^{-1}$ and $f =1$. Thus for $D \subset S_1\pen$ with $n$ odd, and $0<\alpha<1$,
\begin{align*}
C(c_1) {f(x)\rho(x)} || v: C^{2, \beta}(D,\chi,f)||^2 &\leq \frac {C}{\outau}||v:C^{2,\beta}(D,\chi)||^2 
\\&\leq \frac {C}{ \outau}\outau^{2\alpha-2\gamma n}||v:C^{2,\beta}(S_1\pen,\chi)||^2 \leq \outau^{\alpha-1} ||v||^2_{2,\beta, \gamma}.
\end{align*}
Finally, we consider $D \subset \Lambda\penp$. Again, when $n$ is even or odd, the problem changes. For $n$ even, both $f$ and $\rho$ are decaying as functions of $\underline x$ and thus $f(x)\rho(x) \leq C \outau^{-1}$ for all $x \in \Lambda \penp$ for $n$ even. But in this case
\begin{align*}
C(c_1) {f(x)\rho(x)} || v: C^{2, \beta}(D,\chi,f)||^2 &\leq \frac{ C}{ \outau}||v:C^{2,\beta}(D,\chi, e^{-\gamma \underline x})||^2 
\\&\leq \frac {C }{\outau}\outau^{2\alpha-2\gamma (n-1)}||v:C^{2,\beta}(\Lambda\penp,\chi, e^{-\gamma \underline x})||^2 \leq \outau^{\alpha-1} ||v||^2_{2,\beta, \gamma}.
\end{align*}(Notice here we've been helped because $n \geq 2$.)
If $n$ is odd, observe that while $f$ is decaying in $\underline x$, $\rho$ grows. But we note 
\[
\frac{\partial (f\rho)}{\partial \underline {x}} = \left(-\gamma - \frac{P_{\taue}}{P_{\uthte}} w'\right) f\cdot \rho.
\]
For $0<\underline x \leq \ell_\Lambda$, the control on the ratio $\frac{P_{\taue}}{P_{\uthte}}$ (coming from \eqref{diffeodifference}) and the fact that $-1\leq w' <0$ implies $f\rho$ is decreasing as a function of $\underline x$. Thus, $f(x)\rho(x) \leq C(b)$ on all of $\Lambda\penp$ for $n$ odd. It follows
\begin{align*}
C(c_1) {f(x)\rho(x)} || v: C^{2, \beta}(D,\chi,f)||^2 &\leq  C||v:C^{2,\beta}(D,\chi, e^{-\gamma \underline x})||^2 
\\&\leq \outau^{\alpha -1-2\gamma (n-1)}||v:C^{2,\beta}(\Lambda\penp,\chi, e^{-\gamma \underline x})||^2 \leq \outau^{\alpha-1} ||v||^2_{2,\beta, \gamma}.
\end{align*}
\end{proof}

\section{Dislocation at the Central Spheres}\label{appendixb}

For $e \in E(\Gamma)\cup R(\Gamma)$, let $f^+[e]:M[e] \cap\left( [a,a+3] \times \Ss^1\right)\to \Real$ and for $e \in E(\Gamma)$ let $f^-[e]: M[e] \cap\left(  [\RH -(a+3),\RH -a] \times \Ss^1\right)\to \Real$ such that on these regions 
\addtocounter{theorem}{3}
\begin{align}
\label{one}\UUU[e]^{-1}\circ\hYtdz\circ D^+ &=\left(Y_0 + \boldsymbol \zeta\ppe\right)_{f^+[e]}\\
\UUU[e]^{-1}\circ\hYtdz \circ D^-& =\left(Y_0 + \boldsymbol \zeta\pme\right)_{f^-[e]}\\
\UUU[e]^{-1}\circ\hYtdz\circ D &=\left(Y_0 + \boldsymbol \zeta\ppe\right)_{f^+[e]} 
\end{align}where $D, D^\pm$ are small perturbations of the identity map. Such functions exist, for $\outau$ small, by Proposition \ref{quadboundsunscaled}. We prove important estimates for $f^+[e]$, $e \in E(\Gamma)$ and note that the same estimates hold for the other two types of functions once we account for appropriate changes to the domain and notation. Notice that in the lemma below we use the convention of Section \ref{GeoP} and let $C$ be a constant that may depend on $b, \beta, \beta', \gamma, \gamma'$.

\addtocounter{equation}{1}
\begin{lemma}\label{quadflemma}
For $f^+[e]$ and $e \in E(\Gamma)$ as described above, we have the following:
\begin{enumerate}
\item \label{qf1}$f^+[e]= 0$ on $M[e] \cap \left([a,a+1]\times \Ss^1\right)$.
\item \label{qf2}$||f^+[e]:C^{2,\beta}(M[e] \cap \left([a,a+3] \times \Ss^1\right),\chi)||\leq C|| \boldsymbol\zeta||$.
\item \label{qf3}$||\mathcal L_{\chi} f^+[e] - \rho^{-2} H_{dislocation}[e]:C^{0,\beta}(M[e] \cap \left([a,a+3] \times \Ss^1\right),\chi)||\leq C|| \boldsymbol\zeta||^2$.
\item \label{qf4}For $\cf_i$ as described in \eqref{cfeq}, 
\[||f^+[e](x) + \pi\sum_{i=1}^3 \zeta_i\ppe\cf_i(x):C^{1,\beta}(M[e] \cap \left([a+2,a+3]\times \Ss^1\right), \chi)||\leq C|| \boldsymbol\zeta||^2.\]
\end{enumerate}
\end{lemma}
\begin{proof}First observe that on the domain of interest $$\UUU[e]^{-1}\circ\hYtdz= \underline Y_{edge}[\uthte,\taue,l(e),{\boldsymbol \zeta}\ppe,{\boldsymbol \zeta}\pme].$$
Item \eqref{qf1} follows from the definition of the immersion and \eqref{qf2} follows from Proposition \ref{geopropcentral} once we note the domain of interest and the uniform equivalence of $g,\chi$ on this domain. The local estimates of Proposition \ref{quadboundsunscaled} and the $C^{2,\beta}$ bound on $f^+[e]$ then imply
\[||\mathcal L_{g_0}f^+[e] - \rho^{-2}H_{dislocation}[e] \circ D^+:C^{0,\beta}(([a,a+3]\times \Ss^1) \cap M[e], \chi)||\leq C||\boldsymbol \zeta||^2.\]
Here $\mathcal L_{g_0} := \Delta_{\Ss^2} +2$.
We prove \eqref{qf3} by applying an appropriate modification Proposition \ref{geopropcentral}. Finally, recalling the linear bound given in Proposition \ref{quadboundsunscaled}, and substituting $Y_0 + \boldsymbol \zeta\ppe:=Y[e]$ for $X$ we observe
\[||f^+[e] N_{Y[e]} + \boldsymbol \zeta\ppe: C^{1,\beta}(([a+2,a+3]\times \Ss^1)\cap M[e], \chi)|| \leq C||\boldsymbol \zeta||^2.\] 
 Again using a modification of Proposition \ref{geopropcentral} and the fact that on $\left([a+2,a+3] \times \Ss^1\right) \cap M[e]$, $\UUU[e]^{-1}\Ntdz=\pi  (\cf_1,\cf_2,\cf_3) $,
we prove \eqref{qf4}.
\end{proof}
\addtocounter{equation}{1}
\begin{definition}\label{underlinef}For ease of notation, let
\begin{align*}
S^{x}[p]:=M[p]  &\cup_{\{j|p=p^+[e_j]\}}\left(M[e_j] \cap [a,x] \times \Ss^1\right)\\& \cup_{\{j|p=p^-[e_j]\}} \left(M[e_j] \cap [\RHj -x,\RHj -a]\times \Ss^1\right).
\end{align*}
For each $p \in V(\Gamma)$, and $\{e_j\}=E_p$, let 
\[
\underline f[p]: S^{a+3}[p]\to \Real
\]
such that
\begin{equation*}
\underline f[p](x)= \left\{\begin{array}{ll}
0,& \text{if } x \in M[p],\\
f^+[e_j](x),& \text{if }  p=p^+[e_j],x \in M[e_j]\cap [a,a+3] \times \Ss^1,\\
f^-[e_j](x),&\text{if } p=p^-[e_j],x\in M[e_j] \cap[\RHj -(a+3),\RHj -a]\times \Ss^1.
\end{array}\right.
\end{equation*}
\end{definition}
The definitions of $\underline f[p]$ and $H_{dislocation}$ immediately imply the following corollary.
\addtocounter{equation}{1}
\begin{corollary}\label{corundf}
\begin{enumerate}
\item $\underline f[p] =0$ on $S^{a+1}[p]$.
\item $||\mathcal L_\chi \underline f[p]-\rho^{-2}H_{dislocation}:C^{0,\beta}(S^{a+3}[p],\chi)||\leq C|| \boldsymbol\zeta||^2$.
\item if $p=p^+[e]$, $||\underline f[p]+ \pi\sum_{i=1}^3 \zeta_i\ppe\cf_i(x):C^{1,\beta}(M[e] \cap \left([a+2,a+3]\times \Ss^1\right), \chi)||\leq C|| \boldsymbol\zeta||^2$  and
\item if $p=p^-[e]$, 
\[||\underline f[p]+ \pi\sum_{i=1}^3 \zeta_i\ppe\cf_i(x):C^{1,\beta}(M[e] \cap \left([\RH-(a+3),\RH-(a+2)]\times \Ss^1\right), \chi)||\leq C|| \boldsymbol\zeta||^2.
\]
\end{enumerate}
\end{corollary}

\bibliographystyle{amsplain}
\bibliography{Main_Biblio}
\end{document}